\documentclass{amsart}
\usepackage{geometry}
\usepackage{a4wide}
\usepackage{graphicx}
\usepackage{float}
\usepackage{caption}
\usepackage{subcaption}
\graphicspath{ {images/} }
\usepackage{amsmath,mathrsfs,amssymb,amsthm,amscd}
\usepackage{amsfonts}
\usepackage{mathtools}
\usepackage[]{fontenc}
\usepackage[all]{xy}
\usepackage{hyperref}
\usepackage{enumerate}
\usepackage{tikz}
\usepackage{tikz-cd}
\usepackage{tikz-3dplot}
\usepackage{mathtools}
\usepackage{afterpage}
\usepackage{hyperref}
\usepackage{tensor}
\usepackage{xr}
\usepackage{enumitem}
\usepackage{xcolor}
\usepackage{orcidlink}
\usepackage{mathrsfs}
\usepackage{bbm}
\usepackage{orcidlink}
\usetikzlibrary{patterns}
\usetikzlibrary{3d}
\usetikzlibrary{arrows,automata,backgrounds}
\tikzstyle pt=[fill,black]
\tikzstyle ptbl=[fill,black]

\setlist[enumerate,1]{label=(\roman*)}
\setlist[enumerate,2]{label=(\alph*)}

\theoremstyle{plain}
\newtheorem{thm-int}{Theorem}
\newtheorem{thm}{Theorem}[subsection]
\newtheorem{prop}[thm]{Proposition}
\newtheorem{cor}[thm]{Corollary}
\newtheorem{lem}[thm]{Lemma}

\theoremstyle{definition}
\newtheorem{dfn}[thm]{Definition}
\theoremstyle{remark}
\newtheorem{ex}[thm]{Example}
\newtheorem{rem}[thm]{Remark}
\newtheorem{cons}[thm]{Construction}
\newtheorem{desc}[thm]{Description}

\begin{document}
	
	\title{Heights on toric varieties for singular metrics: Global theory}
	\author[G. Peralta]{Gari Y. Peralta Alvarez \orcidlink{0000-0002-1362-8126}}
	\thanks{The author acknowledges support from the Deutsche Forschungsgemeinschaft (DFG, German Research Foundation) under Germany’s Excellence Strategy – The Berlin Mathematics Research Center MATH+ (EXC-2046/1, project ID: 390685689). The author was also supported by the collaborative research center SFB 1085 Higher Invariants - Interactions between Arithmetic Geometry and Global Analysis, funded by the DFG}
	\email{gari.peralta@mathematik.uni-regensburg.de}
	\date{March 9, 2026}
	\begin{abstract}
		In this paper, we develop a toric analog of the theory of adelic divisors on quasi-projective arithmetic varieties introduced by Yuan and Zhang, and extend the convex-analytic descriptions of the Arakelov geometry of projective toric arithmetic varieties given by Burgos, Philippon, and Sombra. Our main result is that the arithmetic self-intersection number of a semipositive toric adelic divisor is given by the integral of a concave function on a compact convex set. These generalized arithmetic intersection numbers coincide with the ones introduced by Burgos and Kramer in 2024, and therefore, can be used to compute heights of toric arithmetic varieties with respect to line bundles equipped with toric singular metrics.
	\end{abstract}
	
	\maketitle
	\setcounter{tocdepth}{1}
	\tableofcontents
	
	\section{Introduction} 
	In this paper, we extend the convex-analytic descriptions of the Arakelov geometry of projective toric varieties developed in~\cite{BPS} and~\cite{BMPS} to the quasi-projective setting of Yuan and Zhang~\cite{Y-Z}, as further generalized by Burgos and Kramer~\cite{BK24}. This article is a sequel to~\cite{Per26}, where we treated the geometric and local arithmetic cases; here, we develop the global theory. In the following, we explain the motivation for our work and outline the results.
	
	\subsection{Heights and singular metrics}\label{1-1}
	The height of a polarized smooth projective variety $X$ over $\mathbb{Q}$ measures its arithmetic complexity, analogous to the way the intersection-theoretic degree measures algebraic complexity. Arakelov \cite{Ar}, and later Gillet and Soulé \cite{GS,BGS}, formalized this by interpreting the height of $X$ as the arithmetic self-intersection of a hermitian line bundle $\overline{\mathcal{L}}$ on a flat and projective integral model $\mathcal{X}$ of $X$. Here $\overline{\mathcal{L}}$ is a line bundle $\mathcal{L}$ on $\mathcal{X}$, together with a smooth hermitian metric $\| \cdot \|$ on the line bundle $L_{\mathbb{C}} = \mathcal{L} \otimes \mathbb{C}$ over the complex manifold $X(\mathbb{C})$. 
	
	The smoothness requirement on the metric is quite restrictive: in many natural situations the relevant metrics are singular, e.g. for the Hodge bundle $\overline{\omega}$ on a toroidal compactification $\overline{\mathcal{A}}_g$ of the moduli stack $\mathcal{A}_g$ of principally polarized abelian schemes of dimension $g$ over $\mathbb{Z}$, and for the line bundle of Siegel–Jacobi forms $\overline{\mathcal{J}}$ on a toroidal compactification $\overline{\mathcal{B}}_g$ of the universal abelian scheme $\mathcal{B}_g$ over $\mathcal{A}_g$. Investigations of these examples led to generalizations of Gillet–Soulé’s arithmetic intersection theory that incorporate singular metrics, which we mention below.
	
	For the Hodge bundle $\overline{\omega}$ on $\overline{\mathcal{A}}_g$, Bost~\cite{B98} and Kühn~\cite{K01} independently computed its arithmetic self-intersection for $g=1$, thereby motivating the arithmetic intersection theory with logarithmically singular metrics developed by Burgos, Kramer, and Kühn~\cite{BKK5,BKK7}. This theory, in turn, enabled the subsequent computations for $g=2$~\cite{JvP22, BBK07} and has further applications in the context of the Kudla program~\cite{KRY}.
	
	For the line bundle of Siegel–Jacobi forms $\overline{\mathcal{J}}$ on $\overline{\mathcal{B}}_g$, Burgos–Kramer–Kühn showed in~\cite{BKK16} that, for $g = 1$, the natural invariant metric has singularities worse than logarithmic, so their theory does not apply. They nevertheless computed the geometric self-intersection of $\mathcal{J}$ using b-divisors, expressing the degree of $\mathcal{J}$ as the limit of intersection numbers of divisors on blow-ups along the boundary of $\mathcal{B}_g$, where the metric is singular. This was later generalized to all $g \geq 1$ in~\cite{BBHdJ21}. Moving to arithmetic intersections~\cite{BK24}, Burgos and Kramer used Yuan–Zhang’s theory of adelic divisors on quasi-projective varieties~\cite{Y-Z}, viewed as an arithmetic analogue of b-divisors. The main idea of this theory is to fix a quasi-projective variety $\mathcal{U}$ where the metric under consideration is smooth. Then, \textit{the group of adelic divisors of }$\mathcal{U}$ \textit{over} $\mathbb{Z}$ is defined as the completion $\overline{\textup{Div}}(\mathcal{U}/\mathbb{Z})$, with respect to a suitable boundary norm, of the direct limit
	\begin{displaymath}
		\varinjlim_{\mathcal{X}} \overline{\textup{Div}}(\mathcal{X})_{\mathbb{Q}},
	\end{displaymath}
	where $\mathcal{X}$ runs over all projective models of $\mathcal{U}$. In this setting, arithmetic intersection numbers of arithmetically nef adelic divisors arise as limits of classical arithmetic intersection numbers. To compute the arithmetic self-intersection of $\overline{\mathcal{J}}$, Burgos and Kramer needed a further extension of this arithmetic intersection product compatible with weaker positivity conditions than nef, and they constructed it using the relative energy introduced by Darvas, Di Nezza, and Lu~\cite{Darvas}.
	
	Besides the height computation discussed above, it is worth noting that Yuan–Zhang’s formalism of adelic line bundles has already found significant applications. For instance, in~\cite{Yua26}, Yuan used this theory to give a new approach to the uniform version of the Bogomolov and Mordell–Lang conjectures, originally established by Dimitrov–Gao–Habegger~\cite{DGH21} and Kühne~\cite{Kuh21}. These striking results highlight the effectiveness of this theory. Consequently, to test new conjectures, it is important to have concrete examples of quasi-projective varieties $\mathcal{U}$ for which explicit computations are feasible. 
	
	However, the group $\overline{\textup{Div}}(\mathcal{U}/\mathbb{Z})$ is a large and intricate object: an adelic divisor $\overline{\mathcal{D}}$ is represented by an equivalence class of Cauchy sequences $\lbrace \overline{\mathcal{D}}_n \rbrace_{n \in \mathbb{N}}$, where each $\overline{\mathcal{D}}_n$ is an arithmetic divisor on some compactification $\mathcal{X}_n$ of $\mathcal{U}$. Therefore, computations on $\overline{\textup{Div}}(\mathcal{U}/\mathbb{Z})$ require a good understanding of both the category of projective models of $\mathcal{U}$ and the Arakelov theory on each model. To address this difficulty, we propose a toric version of Yuan–Zhang’s theory, in which we restrict to quasi-projective toric arithmetic varieties $\mathcal{U}$, toric projective models $\mathcal{X}$ of $\mathcal{U}$, and toric arithmetic divisors $\overline{\mathcal{D}}$ on these models $\mathcal{X}$. In this setting, we can exploit the explicit convex-analytic description of the Arakelov theory of projective toric varieties described below.
	
	\subsection{The global Arakelov theory of projective toric varieties}\label{1-2}
	An alternative approach to Gillet and Soulé’s arithmetic intersection theory, considered by Zhang \cite{Z95} and Gubler \cite{Gub98}, is to define heights \textit{adelically}, meaning that a global arithmetic object is fully described by its local pieces. For instance, a hermitian line bundle $\overline{\mathcal{L}}$ on $\mathcal{X}$ induces a metrized line bundle $\overline{L}$ on $X$: a line bundle $L$ on the generic fiber $X$ of $\mathcal{X}$ and a family of continuous $v$-adic metrics $\lbrace \| \cdot \|_v \rbrace_{v \in \mathfrak{M}(\mathbb{Q})}$ on the line bundles $L_v \coloneqq L \otimes \mathbb{Q}_v$ over $X_v \coloneqq X \times_{\mathbb{Q}} \textup{Spec}(\mathbb{Q}_v)$, where $v$ runs over all places $\mathfrak{M}(\mathbb{Q})$ of $\mathbb{Q}$. One then defines a local height $\textup{h}(X_v; \overline{L}_v)$, and the global height $\textup{h}(X; \overline{L})$ of Gillet and Soulé is recovered by summing over all places.
	
	Moving to the toric setting, we fix a positive integer $d$ and denote by $M$ and $N$ the groups of characters and cocharacters of a $d$-dimensional split torus $U \cong \mathbb{G}_{m}^{d}$ over $\mathbb{Q}$. Given a smooth projective fan $\Sigma$ in $N_{\mathbb{R}}$, there is an associated smooth projective toric variety $X_{\Sigma}/\mathbb{Q}$. A well-known fact in the theory of toric varieties is that a nef toric line bundle $L$ on $X_{\Sigma}$ corresponds to a rational convex polytope $\Delta$ in the euclidean space $M_{\mathbb{R}}$. Then, using the properties of the tropicalization map $\textup{Trop}_v \colon X_{\Sigma,v}^{\textup{an}} \rightarrow N_{\Sigma}$ (\ref{trop-functor}), Burgos, Philippon and Sombra \cite{BPS} give a bijection between:
	\begin{enumerate}
		\item The set of semipositive toric continuous $v$-adic metrics $\lbrace \| \cdot \|_v \rbrace_{v \in \mathfrak{M}(\mathbb{Q})}$ on $L$,
		\item The set of $v$-tuples $(\vartheta_v)$ of concave continuous functions $\vartheta_{v} \colon \Delta \rightarrow \mathbb{R}$ such that $\vartheta_v$ is the constant $0$ for almost all places $v$. 
	\end{enumerate}
	In particular, the second condition implies that the function $\vartheta \colon \Delta \rightarrow \mathbb{R}$ given by
	\begin{displaymath}
		\vartheta \coloneqq \sum_{v \in \mathfrak{M}(\mathbb{Q})} \vartheta_v
	\end{displaymath}
	is well-defined, continuous, and concave. It is known as the \textit{global roof function} associated to $\overline{L}$. Many important arithmetic properties of the semipositive toric metrized line bundle~$\overline{L}$ are encoded by its global roof function. For instance, it was shown in \cite{BMPS} that $\overline{L}$ is arithmetically nef if and only if the global roof function $\vartheta$ does not attain negative values. Most importantly, the height of $X_{\Sigma}$ with respect to $\overline{L}$ can be computed using the global roof function $\vartheta$. Indeed, Theorem~5.2.5~of~\cite{BPS} gives the following integral formula:
	\begin{displaymath}
		\textup{h}(X_{\Sigma}; \overline{L}) = (d+1)! \int_{\Delta} \vartheta = (d+1)! \sum_{v \in \mathfrak{M}(\mathbb{Q})} \int_{\Delta} \vartheta_v.
	\end{displaymath}
	This is an adelic formula: for each place $v$, the integral of $\vartheta_v$ computes the local toric height $\textup{h}^{\textup{tor}}(X_{\Sigma,v};\overline{L}_v)$ introduced by Burgos-Philippon-Sombra. In~\cite{Per26}, we extended the local analogues of these results to the toric and quasi-projective local setting. Furthermore, we showed that, at archimedean places, the local toric height is a special case of the relative energy studied by Burgos–Kramer in~\cite{BK24}. This naturally suggests that the convex-analytic descriptions above should extend to our proposed formalism of toric adelic divisors on quasi-projective toric arithmetic varieties. In this paper, we prove that this is indeed the case.
	
	\subsection{Main results}
	We begin by describing our toric analogue of Yuan–Zhang’s adelic divisors, as defined in~\ref{adelic-dfn}. We restrict to the case of a $d$-dimensional split torus $\mathcal{U}_S$ over $\mathbb{Z} \left[ 1/S \right]$, where $S$ is a positive integer, viewed as a quasi-projective variety over $\mathbb{Z}$. This situation is the most relevant, because for any projective toric arithmetic variety $\mathcal{X}$ over $\mathbb{Z}$, there exists some $S$ such that a torus $\mathcal{U}_S$ embeds into $\mathcal{X}$. Then, the group $\overline{\textup{Div}}_{\mathbb{T}} (\mathcal{U}_{S} /\mathbb{Z})$ of toric adelic divisors on $\mathcal{U}_S$ over $\mathbb{Z}$ is defined as the completion, with respect to a suitable boundary norm, of the direct limit 
	\begin{displaymath}
		\varinjlim_{\mathcal{X}} \overline{\textup{Div}}_{\mathbb{T}}(\mathcal{X})_{\mathbb{Q}},
	\end{displaymath}
	where $\mathcal{X}$ ranges over all toric projective models of $\mathcal{U}_S$, and $ \overline{\textup{Div}}_{\mathbb{T}}(\mathcal{X})_{\mathbb{Q}}$ denotes the group of torus-invariant arithmetic divisors on $\mathcal{X}$. We point out that this construction is an example of an ``\textit{abstract divisorial space}'' in the sense of Cai and Gubler~\cite{CGub24}. 
	
	As mentioned earlier, the results of Burgos–Philippon–Sombra provide an explicit description of the Arakelov geometry on projective toric arithmetic varieties. Therefore, to understand the group $\overline{\textup{Div}}_{\mathbb{T}} (\mathcal{U}_{S} /\mathbb{Z})$, in Section~\ref{3} we systematically study the category of projective toric arithmetic varieties with proper toric morphisms and give a combinatorial description of it in terms of \textit{arithmetic fans}. This is an analogue over $\mathbb{Z}$ of the results of Kempf, Knudsen, Mumford, and Saint-Donat on toric schemes over a discrete valuation ring~\cite{KKMS}. Although our descriptions appear to be folklore, we could not find a reference that goes beyond the so-called \textit{canonical models}~\ref{canonical-OK-prop}, which appear, for instance, in the work of Demazure~\cite{Dmz70} and Rohrer~\cite{Roh10}.
	
	After a careful study of the boundary topology on $\overline{\textup{Div}}_{\mathbb{T}} (\mathcal{U}_{S} /\mathbb{Z})$ (\ref{4--2}), our first main result is a convex-analytic characterization of the cone of semipositive toric adelic divisors on $\mathcal{U}_S$~(\ref{semipositive-OKS-roofs}).
	\begin{thm-int}\label{thm-1}
		There is a bijection between the cone of semipositive toric adelic divisors $\overline{\mathcal{D}}$ of $\mathcal{U}_S$ over $\mathbb{Z}$ and the cone of $v$-tuples $(\vartheta_v)$ of closed concave functions $\vartheta_v \colon M_{\mathbb{R}} \rightarrow \mathbb{R}_{-\infty}$ satisfying the following conditions:
		\begin{enumerate}
			\item For each finite place $v$ such that $|S|_v =1$, the supremum $\sup \vartheta_v$ is a rational number.
			\item There is a compact convex set $\Delta \subset M_{\mathbb{R}}$ such that, for each $v \in \mathfrak{M}(\mathbb{Q})$, the closure of the effective domain $\textup{dom}(\vartheta_v)\coloneqq \lbrace y \in M_{\mathbb{R}} \, \vert \, \vartheta_v (y) > - \infty \rbrace$ coincides with $\Delta$.
			\item For each $\varepsilon >0$ there is a positive integer $S_{\varepsilon}$ dividing $S$ such that for all finite places $v$ with $|S_{\varepsilon} |_v = 1$ we have $ \sup \vartheta_v = 0$ and $\vartheta_v \boxplus \iota_{\overline{\textup{B}}(0,\varepsilon)} \geq \iota_{\Delta}$. Here, $\boxplus$ is the sup-convolution of concave functions, $\iota_C$ is the indicator function of the set $C$, and $\overline{\textup{B}}(0,\varepsilon)$ is the closed ball of radius $\varepsilon$ centred at $0$.
		\end{enumerate}
	\end{thm-int}
	As in the projective case, we introduce the global roof function associated to a semipositive toric adelic divisor $\overline{\mathcal{D}}$ on $\mathcal{U}_S$, which plays a fundamental role in the computation of arithmetic intersection numbers. This is the closed concave function $\vartheta_{\overline{\mathcal{D}}} \colon \Delta_D \rightarrow \mathbb{R}_{-\infty}$ given by
	\begin{displaymath}
		\vartheta_{\overline{\mathcal{D}}} \coloneqq \sum_{v \in \mathfrak{M}(\mathbb{Q})} \vartheta_{\overline{\mathcal{D}}_v},
	\end{displaymath}
	where $(\vartheta_{\overline{\mathcal{D}}_v})$ is the $v$-tuple of concave functions corresponding to $\overline{\mathcal{D}}$ by Theorem~\ref{thm-1}. In \ref{global-roof-OKS}, we show that for each point $y$ in the interior of the convex set $\Delta_D$ associated to $\overline{\mathcal{D}}$, the above series is actually a finite sum. In contrast with Theorem~6.1~of~\cite{BMPS}, the number of non-zero terms $\vartheta_{\overline{\mathcal{D}}_v}(y)$ may increase as $y$ approaches the boundary of $\Delta_D$. Moreover, the function $\vartheta_{\overline{\mathcal{D}}}$ characterizes the nef property in the global arithmetic setting (\ref{nef-torus-OKS}).
	\begin{thm-int}
		A semipositive toric adelic divisor $\overline{\mathcal{D}}$ is nef if and only if its global roof function $\vartheta_{\overline{\mathcal{D}}}$ is non-negative.
	\end{thm-int}
	Finally, in Section~\ref{6}, we study arithmetic intersection numbers. Building on observations from the toric local setting (Remark~4.3.6 and Definition~4.3.7 of~\cite{Per26}), we generalize the arithmetic intersection numbers defined by Yuan–Zhang and Burgos–Kramer (\ref{global-intersection-numbers}). This compatibility is established in Remarks~\ref{almost-finale}~and~\ref{finale}, where we also highlight our improvements over their respective formalisms. We illustrate these improvements through a series of examples culminating in~\ref{worst-case}, where we demonstrate the full strength of our theory by constructing a semipositive toric adelic divisor whose associated $v$-adic Green functions are all highly singular along the boundary of the torus $\mathbb{G}_{m}^{2}/ \mathbb{Z}$ and its arithmetic self-intersection number is finite. We stress that our generalization is specific to the toric setting, as it relies on properties of the global roof function and on explicit approximations of semipositive toric adelic divisors constructed in Subsection~\ref{5-3}. Within this framework, however, we obtain the main result of this article: an integral formula for the arithmetic self-intersection number of a semipositive toric adelic divisor $\overline{\mathcal{D}}$ on $\mathcal{U}_S$ (\ref{global-toric-height-OKS}), generalizing the formula from Subsection~\ref{1-2}.
	\begin{thm-int}
		Let $\overline{\mathcal{D}}$ be semipositive and suppose that $\vartheta_{\overline{\mathcal{D}}}$ is not the constant $-\infty$. Then, the arithmetic self-intersection number of $\overline{\mathcal{D}}$ is given by the formula
		\begin{align*}
			\overline{\mathcal{D}}^{d+1} = (d+1)!  \int_{\Delta_{D}} \vartheta_{\overline{\mathcal{D}}} = (d+1)! \sum_{v \in \mathfrak{M}(\mathbb{Q})} \int_{\Delta_{D}} \vartheta_{\overline{\mathcal{D}}_v}.
		\end{align*}
		This number can be finite or $-\infty$. It is finite if and only if $\vartheta_{\overline{\mathcal{D}}} \in L^1(\Delta_D)$.
	\end{thm-int}
	
	\subsection{Acknowledgments}
	We sincerely thank José I. Burgos Gil and Jürg Kramer for introducing the author to this subject and for the numerous discussions regarding the results presented here. We also thank Marco Flores, Walter Gubler, and Klaus Künnemann for their helpful observations. 
	
	\subsection{Notation and conventions}\label{notation}
	We introduce the notations and conventions that will appear throughout this manuscript. The natural numbers $\mathbb{N}$ include $0$. For each $a \in \mathbb{R}$ denote:
	\begin{itemize}
		\item The set $\mathbb{R}_{>a} \coloneqq \lbrace x \in \mathbb{R} \, \vert \, x > a \rbrace.$ When $a=0$, these are the positive reals.
		\item The set $\mathbb{R}_{\geq a} \coloneqq \lbrace x \in \mathbb{R} \, \vert \, x \geq a \rbrace.$ When $a=0$, these are the non-negative reals.
		\item The set $\mathbb{R}_{- \infty} \coloneqq \lbrace - \infty \rbrace \cup \mathbb{R}$ with the conventions $-\infty < x$ and $-\infty + x = -\infty$ for all $x \in \mathbb{R}$. We equip it with the topology that makes the map $\log : \mathbb{R}_{\geq 0} \rightarrow \mathbb{R}_{- \infty}$ a homeomorphism. We define $\mathbb{R}_{\infty} \coloneqq \mathbb{R} \cup \lbrace \infty \rbrace$ and $\mathbb{R}_{\pm \infty} \coloneqq \lbrace \pm \infty \rbrace \cup \mathbb{R} $ analogously.
	\end{itemize}
	A ring $R$ is always assumed to be commutative, with multiplicative unit $1$. Denote by $R^{\times}$ the group of invertible elements. If $A$ is a finitely generated Abelian group, we write $A_{R} \coloneqq A \otimes_{\mathbb{Z}} R$, the extension of scalars to $R$.
	
	By a \textit{lattice} $N$, we mean a free Abelian group of finite rank $\textup{rk}(N)=d$. Denote by $ M \coloneqq N^{\vee} = \textup{Hom}(N,\mathbb{Z})$ its dual lattice. The natural pairing $\langle \cdot , \cdot \rangle \colon M \times N \rightarrow \mathbb{Z}$ given by evaluation
	\begin{displaymath}
		\langle m , n \rangle \coloneqq m(n).
	\end{displaymath}
	To ease notation, we write $\langle \cdot, \cdot \rangle$ for the induced pairing $\langle \cdot, \cdot \rangle_{R}$ obtained by extending scalars to $R$. When $R = \mathbb{R}$, we equip $N_{\mathbb{R}}$ with the Euclidean topology induced by a norm $\| \cdot \|$. Let $x \in N_{\mathbb{R}}$ and $r \geq 0$, we denote by $\textup{B}(x,r)$ and $\overline{\textup{B}}(x,r)$ the open and closed balls of radius $r$ centred at $x$ respectively. Given a closed subset $C \subset N_{\mathbb{R}}$, the distance from $x$ to $C$ is denoted by $\textup{dist}(x, C) \in \mathbb{R}_{\geq 0}$. We abuse notation slightly by using the same symbols for the corresponding objects on $M_{\mathbb{R}}$, which we equip with the dual norm. Finally, let $\textup{Vol}_M$ be the Haar measure on $M_{\mathbb{R}}$, normalized so that $M$ has covolume $1$.
	
	If $X$ is a topological space, we denote by $C^0 (X)$ the space of continuous real-valued functions and by $C^{0}_{\textup{bd}} (X) \subset C^0 (X)$ the subspace of bounded functions. The space $C^{0}_{\textup{bd}} (X)$ is a Banach space with the uniform norm $\| \cdot \|_{\infty}$. If $X$ is compact, then $C^{0}_{\textup{bd}} (X) = C^0 (X)$. If $X$ admits a differentiable structure, denote by $C^k (X)$ the space of $k$-continuously differentiable real-valued functions on $X$, with $k=\infty$ for smooth functions.
	
	By a local field $K_v = (K,| \cdot |_v)$, we mean either $(\mathbb{R}, | \cdot |_{\infty})$, $(\mathbb{C}, |\cdot |_{\infty})$ with the archimedean absolute value $|\cdot |_{\infty}$, or a field $K$ complete with respect to a discrete absolute value $| \cdot |_v$. In the non-archimedean case, we denote by $v \coloneqq - \log | \cdot |_v$ the associated discrete valuation, $\mathcal{O}_{K_v}$ its valuation ring, and $\mathfrak{m}_v$ the unique prime ideal of $\mathcal{O}_{K_v}$. We let $k_v \coloneqq \mathcal{O}_{K_v} / \mathfrak{m}_v$ be its residue field. If the valuation $v$ is clear, we may omit it from the notation.
	
	By a number field $K$ we mean a finite extension of $\mathbb{Q}$. Let $\mathcal{O}_K$ be its ring of integers. Denote by $\mathfrak{M}(K)$ the set of all places of $K$; it is the union of the sets $\mathfrak{M}(K)_{\textup{fin}}$ and $\mathfrak{M}(K)_{\infty}$ of finite and infinite places, respectively. For each place $v$, we let $K_v$ be the local field given by completion of $K$ with respect to $v$. There is a correspondence between the finite places $\mathfrak{M}(K)_{\textup{fin}}$ and the maximal ideals of $\mathcal{O}_K$. Indeed, for each finite place $v$,
	\begin{displaymath}
		\mathfrak{p} \coloneqq \lbrace f \in K \, \vert \, |f|_{v} <1 \rbrace, \quad \mathcal{O}_{K,\mathfrak{p}} \coloneqq \lbrace f \in K \, \vert \, |f|_{v} \leq 1 \rbrace, \quad \textup{and} \quad k(\mathfrak{p}) \coloneqq \mathcal{O}_{K}/\mathfrak{p}
	\end{displaymath}
	determine a maximal ideal of $\mathcal{O}_K$, the localization of $\mathcal{O}_K$ at $\mathfrak{p}$, and the residue field $k(\mathfrak{p})$, respectively. On the other hand, there is a correspondence between infinite places $\mathfrak{M}(K)_{\infty}$ and the set $\Lambda \coloneqq \textup{Hom}(K,\mathbb{C})$ of complex embeddings, modulo complex conjugation. We may use these notations interchangeably, depending on the context. Given a finite subset $S \subset \mathfrak{M}(K)_{\textup{fin}}$, the ring $\mathcal{O}_{K,S}$ of $S$-integers of $K$ is the set of elements $f \in K$ satisfying $|f|_v \leq 1$ for all places $v  \in \mathfrak{M}(K)_{\textup{fin}} \setminus S$.
	
	All of our schemes are assumed to be Noetherian. We denote by $\mathcal{O}_{\mathcal{X}}$ and $K(\mathcal{X})$ the structure sheaf and function field of the (integral) scheme $\mathcal{X}$, respectively.  By a $d$-dimensional \textit{variety} $\mathcal{X}$ over a base scheme $\mathcal{S}$, denoted $\mathcal{X}/\mathcal{S}$, we mean an integral scheme $\mathcal{X}$ which is separated, flat and of finite type over $\mathcal{S}$, and has relative dimension $d = \textup{dim}_{\mathcal{S}} (\mathcal{X})$. We say that the variety $\mathcal{X}/\mathcal{S}$ is projective (resp. quasi-projective, smooth) if the structural morphism $\mathcal{X} \rightarrow \mathcal{S}$ is projective (resp. quasi-projective, smooth). A morphism of varieties over $\mathcal{S}$ (or $\mathcal{S}$-morphism) is a morphism $\mathcal{X}_1 \rightarrow \mathcal{X}_2$ compatible with the respective structural morphisms.  We reserve plain Roman letters for schemes over a field, e.g., $X \rightarrow \textup{Spec}(K)$.
	
	Let $\mathcal{X}/\mathcal{S}$ be a variety and $\textup{Div}(\mathcal{X})$ be the group of Cartier divisors on $\mathcal{X}$. We have a morphism $\textup{div}\colon K(\mathcal{X})^{\times} \rightarrow \textup{Div}(\mathcal{X})$ given by the assignment $f \mapsto \textup{div}(f)$. The image of the morphism $\textup{div}$ is the subgroup $\textup{Pr}(\mathcal{X})$ of \textit{principal} Cartier divisors on $\mathcal{X}$. The \textit{Cartier class group} of $\mathcal{X}$ is the quotient $\textup{Cl}(\mathcal{X})\coloneqq \textup{Div}(\mathcal{X}) / \textup{Pr}(\mathcal{X})$. Denote by $\textup{Pic}(\mathcal{X})$ the Picard group of $\mathcal{X}$, i.e., the group of isomorphism classes of line bundles over $\mathcal{X}$. Let $\mathcal{L}$ be a line bundle over $\mathcal{X}$ and $s \neq 0$ be a rational section of $\mathcal{L}$. The variety $\mathcal{X}$ is integral, therefore the assignment $s \mapsto \textup{div}(s)$ induces an isomorphism $\textup{Pic}(\mathcal{X}) \cong \textup{Cl} (\mathcal{X})$. Since $\mathbb{Q}$ is flat over $\mathbb{Z}$, this isomorphism is preserved after scalar extension. Given $\mathcal{D} \in \textup{Div}(\mathcal{X})$, we denote by $\mathcal{O}_{\mathcal{X}}(\mathcal{D})$ the associated line bundle.
	
	\section{Review of Yuan-Zhang's theory of adelic divisors}\label{global-arith-theory}
	In this section, we briefly recall basic definitions and properties of the global theory of adelic divisors, introduced by Yuan and Zhang in~Chapter~2~of~\cite{Y-Z}. We will use the notation in Subsection~\ref{notation} freely. We also assume some familiarity with the local theory of adelic divisors, which was introduced in~Chapter~3~of~\cite{Y-Z}, and summarized in~Section~2~of~\cite{Per26}. We stick to the notation in the latter; our definitions differ slightly, as we avoid the use of $(\mathbb{Q},\mathbb{Z})$-divisors.
	
	\subsection{Adelic divisors}\label{2-1} We first recall basic notions in Arakelov theory. Let $K$ be a number field, $\mathcal{O}_K$ its ring of integers, and $\Lambda$ its set of complex embeddings. By an \textit{arithmetic variety} we mean a normal quasi-projective variety $\mathcal{X} / \mathcal{O}_K$ with smooth generic fiber. Given such an object, we consider its base change
	\begin{displaymath}
		\mathcal{X}_{\Lambda} \coloneqq \coprod_{\lambda \in \Lambda}  \mathcal{X}_{\lambda}, \quad  \mathcal{X}_{\lambda} \coloneqq   \mathcal{X} \times_{\lambda} \textup{Spec}(\mathbb{C}).
	\end{displaymath}
	Its set of complex points $\mathcal{X}_{\Lambda}(\mathbb{C}) = \coprod_{\lambda} \mathcal{X}_{\lambda}(\mathbb{C})$ comes with the structure of a complex manifold via the analytification functor, it is compact if $\mathcal{X}$ is projective, and complex conjugation induces an involution $\zeta$ of $\mathcal{X}_{\Lambda}(\mathbb{C})$. A function $f$ on $\mathcal{X}_{\Lambda}(\mathbb{C})$ is \textit{invariant under complex conjugation} if, for each $\lambda \in \Lambda$, it satisfies $f_{\overline{\lambda}} = f_{\lambda} \circ \zeta$, where $f_{\lambda} \coloneqq f|_{\mathcal{X}_{\lambda}(\mathbb{C})}$. 
	
	Given a $\mathbb{Q}$-Cartier divisor $\mathcal{D}$ on $\mathcal{X}$, we let $\mathcal{D}_{\Lambda}$ be the $\mathbb{Q}$-divisor on $\mathcal{X}_{\Lambda}$ induced by base change. If $P$ is a property of functions (continuous, smooth, etc), a \textit{Green's function for }$\mathcal{D}_{\Lambda}$ \textit{of} $P$ \textit{type} is a function $g_{\mathcal{D},\Lambda} \coloneqq \mathcal{X}_{\Lambda}(\mathbb{C}) \rightarrow \mathbb{R}_{\pm \infty}$ invariant under complex conjugation such that, for each Zariski open $V \subset \mathcal{X}_{\Lambda}$ and each local equation $f$ of $\mathcal{D}_{\Lambda}$ on $V$, the function $g_{\mathcal{D},\Lambda} + \log |f|^2$ is finite and satisfies the property $P$ on $V(\mathbb{C})$. The pair $\overline{\mathcal{D}} \coloneqq (\mathcal{D}, g_{\mathcal{D},\Lambda})$ is called an  \textit{arithmetic divisor on} $\mathcal{X}$ \textit{of} $P$ \textit{type}. For simplicity, we often write $g_{\mathcal{D}}$ instead of $g_{\mathcal{D},\Lambda}$. We denote by $\overline{\textup{Div}}(\mathcal{X})_{\mathbb{Q}}$ the group of arithmetic divisors of continuous type on $\mathcal{X}$. 
	
	Now, we fix an arithmetic variety $\mathcal{U}/\mathcal{O}_K$. A \textit{projective model of} $\mathcal{U}$ \textit{over} $\mathcal{O}_K$ is an open immersion of $\mathcal{U}$ into a projective variety $\mathcal{X}$ over $\mathcal{O}_K$. We write $\textup{PM}(\mathcal{X}/\mathcal{O}_K)$ for the category of projective models of $\mathcal{U}$ over $\mathcal{O}_K$, where the arrows in this category are proper morphisms $\mathcal{X}_1 \rightarrow \mathcal{X}_2$ which commute with the open immersions $\mathcal{U} \rightarrow \mathcal{X}_i$. In this case, we say that $\mathcal{X}_1$ \textit{dominates} $\mathcal{X}_2$. This category is cofiltered, i.e., for any pair of models $\mathcal{X}_1$ and $\mathcal{X}_2$, there exists a third model $\mathcal{X}$ that dominates both. Given a morphism of projective models $f \colon \mathcal{X}_1 \rightarrow \mathcal{X}_2$, there is an induced pullback morphism
	\begin{displaymath}
		f^{\ast} \colon \overline{\textup{Div}}(\mathcal{X}_2)_{\mathbb{Q}} \longrightarrow \overline{\textup{Div}}(\mathcal{X}_1)_{\mathbb{Q}}.
	\end{displaymath}
	Then, the group of \textit{model arithmetic divisors of }$\mathcal{U}$ \textit{over} $\mathcal{O}_K$ is the direct limit
	\begin{displaymath}
		\overline{\textup{Div}}(\mathcal{U}/\mathcal{O}_K)_{\textup{mod}} \coloneqq \varinjlim_{\mathcal{X} \in \textup{PM}(\mathcal{U}/\mathcal{O}_K)} \overline{\textup{Div}}(\mathcal{X})_{\mathbb{Q}}.
	\end{displaymath}
	Elements of this group are called \textit{model divisors}, and they can be identified with equivalence classes of pairs $(\overline{\mathcal{D}},\mathcal{X})$ consisting of an arithmetic divisor $\overline{\mathcal{D}}$ on a projective model $\mathcal{X}$ of $\mathcal{U}$, where $(\overline{\mathcal{D}}_1,\mathcal{X}_1)$ is equivalent to $(\overline{\mathcal{D}}_2 ,\mathcal{X}_2)$ if the  $\overline{\mathcal{D}}_i$'s coincide after pullback to a projective model $\mathcal{X}$ dominating both $\mathcal{X}_1$ and $\mathcal{X}_2$. We abuse notation by writing $\overline{\mathcal{D}}$ instead of (the equivalence class of) $(\overline{\mathcal{D}},\mathcal{X})$. 
	\begin{rem}
		Observe that the full subcategory of $\textup{PM}(\mathcal{X}/\mathcal{O}_K)$ whose objects are projective arithmetic varieties is cofinal. Therefore, we may define the group $	\overline{\textup{Div}}(\mathcal{U}/\mathcal{O}_K)_{\textup{mod}}$ by taking the direct limit over this full subcategory. Indeed, given a projective model $\mathcal{X}$ of $\mathcal{U}$, there is an open immersion of generic fibers $U \rightarrow X$. Since $U$ is smooth, the singularities of $X$ must be contained in $\mathcal{X} \setminus \mathcal{U}$. Then, by taking successive blow-ups of $\mathcal{X}$ with center in $ X \setminus U \subset  \mathcal{X} \setminus \mathcal{U}$, we can resolve the singularities of $X$ and obtain a projective arithmetic variety $\mathcal{X}^{\prime}$ which is a projective model of $\mathcal{U}$ and dominates $\mathcal{X}$.
	\end{rem}
	The group $	\overline{\textup{Div}}(\mathcal{U}/\mathcal{O}_K)_{\textup{mod}}$ admits a natural topology compatible with the group structure. It is known as the \textit{boundary topology}, defined as follows. A \textit{boundary divisor of} $\mathcal{U}/\mathcal{O}_K$ is a pair $(\mathcal{X}, \overline{\mathcal{B}})$ consisting of a projective model $\mathcal{U} \rightarrow \mathcal{X}$ and an arithmetic divisor $\overline{\mathcal{B}} = (\mathcal{B},g_{\mathcal{B}})$ on $\mathcal{X}$ which is \textit{strictly effective}; that is, the divisor $\mathcal{B}$ satisfies the support condition $|\mathcal{B}| = \mathcal{X} \setminus \mathcal{U}$ and the Green's function $g_{\mathcal{B}}$ is positive. By Lemma~2.3.5~of~\cite{Y-Z}, there is an induced notion of effectivity at the level of model divisors; in particular, the notion of boundary divisor carries over to the corresponding equivalence class. Then, a boundary divisor $\overline{\mathcal{B}}$ induces a \textit{boundary norm} $\| \cdot \|_{\overline{\mathcal{B}}}$ on $\overline{\textup{Div}}(\mathcal{U}/\mathcal{O}_K)_{\textup{mod}}$, which is given by
	\begin{displaymath}
		\| \overline{D} \|_{\overline{\mathcal{B}}} \coloneqq \inf \lbrace \varepsilon \in \mathbb{Q}_{> 0} \, \vert \, - \varepsilon \cdot \overline{\mathcal{B}} \leq \overline{\mathcal{D}} \leq \varepsilon \cdot \overline{\mathcal{B}} \rbrace, \quad \inf \emptyset = \infty.
	\end{displaymath}
	This is not a norm in the usual sense; it takes the value $\infty$ if and only if $\mathcal{D}|_{\mathcal{U}} \neq 0$, where $\mathcal{D}$ is the divisorial part of $\overline{\mathcal{D}}$. Nevertheless, by Lemma~2.4.1~of~\cite{Y-Z}, it satisfies the properties of a norm: It vanishes only at $0$, satisfies the triangle inequality, and $\| a \cdot \overline{\mathcal{D}} \|_{\overline{\mathcal{B}}} = |a| \, \| \overline{D} \|_{\overline{\mathcal{B}}}$ for each $a \in \mathbb{Q}$. Moreover, given two boundary divisors, their respective boundary norms are equivalent. Therefore, a boundary norm induces a topology on the group $\overline{\textup{Div}}(\mathcal{U}/\mathcal{O}_K)_{\textup{mod}}$ which does not depend on the choice of boundary divisor. This justifies the definition below.
	\begin{dfn}\label{adelic-divisors}
		The \textit{group of adelic divisors of} $\mathcal{U}$ \textit{over }$\mathcal{O}_K$, denoted by $\overline{\textup{Div}}(\mathcal{U}/\mathcal{O}_K)$, is defined as the completion of the group $\overline{\textup{Div}}(\mathcal{U}/\mathcal{O}_K)_{\textup{mod}}$ of model divisors with respect to a boundary norm. An \textit{adelic divisor} is an element of this group.
	\end{dfn}
	Note that an adelic divisor $\overline{\mathcal{D}}$ identifies with an equivalence class of Cauchy sequences $\lbrace \overline{\mathcal{D}}_n \rbrace_{n \in \mathbb{N}}$, where two Cauchy sequences are equivalent if their difference converges to $0$. We abuse notation by writing $\overline{\mathcal{D}}=\lbrace \overline{\mathcal{D}}_n \rbrace_{n \in \mathbb{N}}$ whenever the Cauchy sequence $\lbrace \overline{\mathcal{D}}_n \rbrace_{n \in \mathbb{N}}$ is a representative of the equivalence class $\overline{\mathcal{D}}$. 
	\begin{rem}\label{decreasing}
		It will often be useful to assume that the Cauchy sequence $\lbrace \overline{\mathcal{D}}_n \rbrace_{n \in \mathbb{N}}$ is decreasing. This is always possible: By the Cauchy property and the definition of the boundary norm $\| \cdot \|_{\overline{\mathcal{B}}}$, there exists an increasing sequence $\lbrace k_n \rbrace_{n \in \mathbb{N}}$ such that
		\begin{displaymath}
			2^{-n-1} \cdot \overline{\mathcal{B}}_n \geq \overline{\mathcal{D}}_i - \overline{\mathcal{D}}_j \geq - 2^{-n-1} \cdot \overline{\mathcal{B}},
		\end{displaymath}
		for all $i,j \geq k_n$. Then, the Cauchy sequence $\lbrace \overline{\mathcal{D}}_{k_n} + 2^{-n} \cdot \overline{\mathcal{B}} \rbrace_{n \in \mathbb{N}}$ is decreasing and represents $\overline{\mathcal{D}}$.
	\end{rem}
	The \textit{group of compactified (geometric) divisors of} $\mathcal{U}$ \textit{over }$\mathcal{O}_K$, denoted by $\textup{Div}(\mathcal{U}/\mathcal{O}_K)$, is defined exactly in the same way as the adelic divisors, just forgetting about Green's functions. Each adelic divisor $\overline{\mathcal{D}} = \lbrace (\mathcal{D}_n, g_{\mathcal{D}_n}) \rbrace_{n \in \mathbb{N}}$ determines a sequence $\mathcal{D} = \lbrace \mathcal{D}_n \rbrace_{n \in \mathbb{N}}$ which is Cauchy in the geometric boundary topology. The compactified divisor $\mathcal{D}$ is called the \textit{divisorial part of} $\overline{\mathcal{D}}$. The assignment $\overline{\mathcal{D}} \mapsto \mathcal{D}$ induces a continuous group morphism $\textup{for} \colon \overline{\textup{Div}}(\mathcal{U}/\mathcal{O}_K) \rightarrow \textup{Div}(\mathcal{U}/\mathcal{O}_K)$, known as the \textit{forgetful map}. Given $\overline{\mathcal{D}} = \lbrace (\mathcal{D}_n, g_{\mathcal{D}_n}) \rbrace_{n \in \mathbb{N}}$, the restriction $\mathcal{D}_n |_{\mathcal{U}}$ determines the same divisor every $n$, which we denote by $\mathcal{D}|_{\mathcal{U}}$. Observe that the sequence $\lbrace g_{\mathcal{D}_n} |_{\mathcal{U}_{\Lambda} (\mathbb{C})} \rbrace_{n \in \mathbb{N}}$ converges uniformly on each compact subset of $\mathcal{U}_{\Lambda} (\mathbb{C})$. Therefore, it determines a Green's function $g_{\mathcal{D}} \colon \mathcal{U}_{\Lambda}(\mathbb{C}) \rightarrow \mathbb{R}_{\pm \infty}$ for $\mathcal{D}|_{\mathcal{U}}$ of continuous type. It is clear that the adelic divisor $\overline{\mathcal{D}}$ is uniquely determined by the pair $(\mathcal{D},g_{\mathcal{D}})$, so we may abuse notation by writing $\overline{\mathcal{D}} = (\mathcal{D},g_{\mathcal{D}})$. Then, Theorem~3.6.4~of~\cite{Y-Z} below describes the asymptotic growth of $g_{\mathcal{D}}$ as it approaches the boundary of $\mathcal{U}_{\Lambda}(\mathbb{C})$. 
	\begin{thm}\label{exact-seq-global}
		The forgetful map $\textup{for} \colon \overline{\textup{Div}}(\mathcal{U}/\mathcal{O}_K) \rightarrow \textup{Div}(\mathcal{U}/\mathcal{O}_K)$ is surjective. It induces a short exact sequence
		\begin{center}
			\begin{tikzcd}
				0 \arrow[r] & C^{0}(\mathcal{U}_{\Lambda}(\mathbb{C}))_{\textup{cptf}} \arrow[r] & \overline{\textup{Div}}(\mathcal{U}/\mathcal{O}_K)  \arrow[r] & \textup{Div}(\mathcal{U}/ \mathcal{O}_K) \arrow[r] & 0.
			\end{tikzcd}
		\end{center}
		Moreover, for each boundary divisor $(\mathcal{X}, (\mathcal{B},g_{\mathcal{B}}))$ of $\mathcal{U}/\mathcal{O}_K$, the assignment $ h \mapsto (h \cdot g_{\mathcal{B}})|_{\mathcal{U}_{\Lambda}(\mathbb{C})}$ induces a bijection between the space of continuous functions $h \colon \mathcal{X}_{\Lambda}(\mathbb{C}) \rightarrow \mathbb{R}$ which are invariant under complex conjugation and satisfy $h(\mathcal{X}_{\Lambda}(\mathbb{C}) \setminus \mathcal{U}_{\Lambda}(\mathbb{C})) = 0$, and the kernel of the forgetful map.
	\end{thm}
	Now, we state Corollary~3.4.2~of~\cite{Y-Z}. It describes the functorial behavior of shrinking the arithmetic variety.
	\begin{lem}\label{functoriality-general}
		Let $\mathcal{U} \rightarrow \mathcal{V}$ be an open immersion of arithmetic varieties over $\mathcal{O}_{K}$. Then, restriction of model divisors induces a continuous injective map $\overline{\textup{Div}}(\mathcal{V}/\mathcal{O}_K)  \rightarrow \overline{\textup{Div}}(\mathcal{U}/\mathcal{O}_K)$.
	\end{lem}
	\subsection{A global-local approach}\label{2-2} In this subsection, we assume familiarity with the notation introduced in Section~2~of~\cite{Per26}, the first article of the series. We note that all of these objects were introduced in Chapter~3~of~\cite{Y-Z}, but we reference our paper for consistency in notation. We continue to fix an arithmetic variety $\mathcal{U} / \mathcal{O}_K$ and let $U$ be its generic fiber. For each place $v$, we let $U_v \coloneqq U \times_K \textup{Spec}(K_v)$ be the quasi-projective variety obtained by base change. If $v$ is a finite place of $K$, we also consider the variety $\mathcal{U}_{v}\coloneqq \mathcal{U} \times_{\mathcal{O}_K} \textup{Spec} ( \mathcal{O}_{K_v})$ induced by localization, and its special fiber $\mathcal{U}_{k_v} \coloneqq \mathcal{U} \times_{\mathcal{O}_{K}} \textup{Spec}(k_v)$. The adjective ``\textit{adelic}'' strongly suggests the existence of a local-global principle relating the group $\overline{\textup{Div}}(\mathcal{U}/\mathcal{O}_K)$ with the product
	\begin{displaymath}
		\prod_{v \in \mathfrak{M}(K) } \overline{\textup{Div}}(U_v/ K_v ),
	\end{displaymath}
	where $ \overline{\textup{Div}}(U_v/ K_v )$ is the group of compactified arithmetic divisors of $U_v/K_v$ (Definition~2.2.5 of~\cite{Per26}). This is indeed the case. For each finite place $v \in \mathfrak{M}(K)$, localization and analytification of divisors (Remark~2.2.7~of~\cite{Per26}) induces continuous morphisms
	\begin{center}
		\begin{tikzcd}
			\overline{\textup{Div}}(\mathcal{U}/\mathcal{O}_K) \arrow[r, "{\textup{for}}"] & \textup{Div}(\mathcal{U}/\mathcal{O}_K) \arrow[r] & \textup{Div}(\mathcal{U}_v / \mathcal{O}_{K_v}) \arrow[r, "{\textup{an}}"]  & \overline{\textup{Div}}(U_v/ K_v ).
		\end{tikzcd}
	\end{center}
	For the archimedean places, the corresponding map $\overline{\textup{Div}}(\mathcal{U}/\mathcal{O}_K) \rightarrow \overline{\textup{Div}}(\mathcal{U}_{\Lambda} / \mathbb{C} )$ is given by the assignment $(\mathcal{D},g_{\mathcal{D} }) \mapsto (\mathcal{D}_{\Lambda}, g_{\mathcal{D} })$. Altogether, we obtain a map
	\begin{displaymath}
		\overline{\textup{Div}}(\mathcal{U}/\mathcal{O}_K) \longrightarrow \prod_{v \in \mathfrak{M}(K) } \overline{\textup{Div}}(U_v/ K_v ).
	\end{displaymath}
	For each adelic divisor $\overline{\mathcal{D}}$, we write $( D_v, g_{D,v})_v$ for its image under this map. The theory of Berkovich spaces over $\mathcal{O}_K$ allows us to study the properties of the above map geometrically. We refer to Chapter~1~of~\cite{LP24} for the study of Berkovich spaces over $\mathbb{Z}$ and more general bases. The pair $(\mathcal{O}_K, \| \cdot \| )$ is a Banach ring, where the norm $\| \cdot \|$ is given by 
	\begin{displaymath}
		\| \cdot \| \coloneqq \max_{\lambda \in \Lambda} | \lambda ( \cdot ) |_{\infty}.
	\end{displaymath}
	Then, there is an analytification functor from the category of $\mathcal{O}_K$-varieties to the category of analytic spaces over $(\mathcal{O}_K, \| \cdot \| )$. Applying this functor to structural morphism $\mathcal{U} \rightarrow \textup{Spec}(\mathcal{O}_K)$ gives rise to a continuous map $\mathcal{U}^{\textup{an}} \rightarrow \mathcal{M}(\mathcal{O}_K)$, where $\mathcal{M}(\mathcal{O}_K)$ is the Berkovich spectrum of $\mathcal{O}_K$. By Lemma~1.2.18~of~\cite{LP24}, the fibers of the map $\mathcal{U}^{\textup{an}} \rightarrow \mathcal{M}(\mathcal{O}_K)$ induce a decompostion
	\begin{displaymath}
		\mathcal{U}^{\textup{an}} = U^{\textup{an}}_{0} \cup \coprod_{\lambda \in \Lambda} (0,1] \times \mathcal{U}_{\lambda}(\mathbb{C}) \cup \coprod_{v \in \mathfrak{M}(K)_{\textup{fin}} } (U_{v}^{\textup{an}} \times (0,\infty) \cup \mathcal{U}_{k_v}^{\textup{an}} ).
	\end{displaymath}
	Here, $U_{0}^{\textup{an}}$ and $\mathcal{U}_{k_v}^{\textup{an}}$ are the analytifications with respect to the trivial absolute values on $K$ and $k_v$, respectively. Then, we define arithmetic divisors in this setting.
	\begin{dfn}\label{eqv}
		An arithmetic divisor on $\mathcal{U}^{\textup{an}}$ is a pair $(\mathcal{D},g_{\mathcal{D}})$ consisting of a Cartier divisor $\mathcal{D}$ on $\mathcal{U}$ and a Green's function $g_{\mathcal{D}}  \colon \mathcal{U}^{\textup{an}} \rightarrow \mathbb{R}_{\pm \infty}$ for ${\mathcal{D}}$. We say that the Green's function $g_{\mathcal{D}}$ and the divisor $\overline{\mathcal{D}}$ are \textit{equivariant} if, for all $p_1, p_2 \in  \mathcal{U}^{\textup{an}}$ satisfying $| \cdot |_{p_2} = | \cdot |_{p_1}^{t}$ for some $0 < t < \infty$, we have  $g_{\mathcal{D}} (p_2) = t \, g_{\mathcal{D}} (p_1)$. Denote by $\overline{\textup{Div}}(\mathcal{U}^{\textup{an}})_{\textup{eqv}}$ the group of equivariant arithmetic divisors of continuous type.
	\end{dfn}
	Each adelic divisor $\overline{\mathcal{D}}$ gives rise to a family of compactified geometric divisors $( D_v, g_{D,v})_v$. Then, by the decomposition of $\mathcal{U}^{\textup{an}}$ above, we may extend the Green's functions equivariantly to obtain a function defined over the non-trivially valued fibers. The following result combines Proposition~3.3.1 and Lemma~3.3.3~of~\cite{Y-Z}, showing that this procedure gives rise to an equivariant arithmetic divisor.
	\begin{prop}\label{global-analytification}
		With the previous notation, the assignment $\overline{\mathcal{D}} \mapsto ( D_v, g_{D,v})_v$ induces a natural continuous and injective group morphism $\textup{an}\colon \overline{\textup{Div}}(\mathcal{U}/\mathcal{O}_K) \rightarrow \overline{\textup{Div}}(\mathcal{U}^{\textup{an}})_{\textup{eqv}}$. Moreover, an adelic divisor $\overline{\mathcal{D}}$ is effective if and only if the Green's function $g_{\mathcal{D}}$ obtained by analytification is non-negative. 
	\end{prop}
	\begin{rem}\label{coherence}
		A related question is to find under which condition a family of compactified divisors $( D_v, g_{D,v})_v$ comes from an adelic divisor $\overline{\mathcal{D}}$. This is a so-called \textit{coherence condition}. In \cite{Son24}, Song has announced that the map $\textup{an}\colon \overline{\textup{Div}}(\mathcal{U}/\mathcal{O}_K) \rightarrow \overline{\textup{Div}}(\mathcal{U}^{\textup{an}})_{\textup{eqv}}$ is surjective, and therefore, an isomorphism. In other words, a family of compactified divisors $ ( D_v, g_{D,v})_v$ determines an adelic divisor if it can be extended equivariantly to the whole of $\mathcal{U}^{\textup{an}}$. In Subsection~\ref{4-2}, we give a coherence condition for families of nef toric compactified arithmetic divisors in terms of convex analysis.
	\end{rem}
	
	\subsection{Nefness, semipositivity and intersection numbers} First, we describe the case of a regular projective arithmetic variety $\mathcal{X} / \mathcal{O}_K$ of relative dimension $d$. Afterwards, we move to the quasi-projective setting. In~\cite{GS}, Gillet and Soul\'{e} introduced the \textit{arithmetic Chow ring of} $\mathcal{X}$. It is a graded commutative ring $\widehat{\textup{CH}}^{\bullet}(\mathcal{X})_{\mathbb{Q}}$ which satisfies the expected functorial properties from classical intersection theory~\cite{Ful}: proper pushforward, flat pullback, and a projection formula. We are interested in the homogeneous part of degree $1$, defined as follows. Let $K(\mathcal{X})$ be the function field of $\mathcal{X}$ and $f \in K(\mathcal{X})^{\times}$ be an invertible element. An arithmetic divisor of the form $\overline{\textup{div}}(f) \coloneqq (\textup{div}(f), -\log |f|^2)$ is said to be \textit{principal}. We denote by $\overline{\textup{Pr}}(\mathcal{X})_{\mathbb{Q}}$ and $\overline{\textup{Div}}^{\infty}(\mathcal{X})_{\mathbb{Q}}$ the $\mathbb{Q}$-vector subspaces of $\overline{\textup{Div}}(\mathcal{X})_{\mathbb{Q}}$ spanned by the principal arithmetic divisors and the arithmetic divisors of smooth type, respectively. Trivially, a principal divisor is of smooth type. Then, the \textit{first arithmetic Chow group of}$\mathcal{X}$ is defined as the quotient
	\begin{displaymath}
		\widehat{\textup{CH}}^{1}(\mathcal{X})_{\mathbb{Q}} \coloneqq \overline{\textup{Div}}^{\infty}(\mathcal{X})_{\mathbb{Q}} / \overline{\textup{Pr}}(\mathcal{X})_{\mathbb{Q}}.
	\end{displaymath}
	The morphism given by the composition $\mathcal{X} \rightarrow \textup{Spec}(\mathcal{O}_K) \rightarrow \textup{Spec}(\mathbb{Z})$ has relative dimension $d$. Taking its proper pushforward, we obtain an \textit{arithmetic degree map}
	\begin{displaymath}
		\overline{\textup{deg}} \colon \widehat{\textup{CH}}^{d+1}(\mathcal{X})_{\mathbb{Q}} \longrightarrow \widehat{\textup{CH}}^{1}(\textup{Spec}(\mathbb{Z}))_{\mathbb{Q}} \cong \mathbb{R}.
	\end{displaymath}
	Then, we get an \textit{arithmetic intersection pairing} $\overline{\textup{Div}}^{\infty}(\mathcal{X})^{d+1}_{\mathbb{Q}} \rightarrow  \mathbb{R}$ given by
	\begin{displaymath}
		\overline{\mathcal{D}}_0 \cdot \ldots \cdot \overline{\mathcal{D}}_d \coloneqq  \overline{\textup{deg}} (\left[ \overline{\mathcal{D}}_0 \right] \cdot \ldots \cdot \left[ \overline{\mathcal{D}}_d \right]),
	\end{displaymath}
	where $\left[ \overline{\mathcal{D}} \right]$ denotes the class of $\overline{\mathcal{D}}$ in the first arithmetic Chow group. We call $\overline{\mathcal{D}}_0 \cdot \ldots \cdot \overline{\mathcal{D}}_d$ the \textit{arithmetic intersection number} of the arithmetic divisors $\overline{\mathcal{D}}_0, \ldots, \overline{\mathcal{D}}_d$. This pairing is symmetric, multilinear, and satisfies the functorial properties mentioned above.
	
	Following \cite{Z95}, if we remove the regularity assumption on the projective arithmetic variety $\mathcal{X}$, we still have an intersection pairing $\overline{\textup{Div}}^{\infty}(\mathcal{X})^{d+1}_{\mathbb{Q}} \rightarrow  \mathbb{R}$ of arithmetic divisors of smooth type. This pairing can be extended to include arithmetic divisors of continuous type satisfying certain positivity conditions. Recall that a divisor $\mathcal{D}$ on $\mathcal{X}$ is \textit{relatively nef} if, for every vertical curve $\mathcal{C}$ in $\mathcal{X}$, the (geometric) intersection number $\mathcal{D} \cdot \mathcal{C}$ is non-negative. Now, let $g = g_{\mathcal{D},\Lambda}$ be a Green's function for $\mathcal{D}_{\Lambda}$ of continuous type and consider the $(1,1)$-current $\omega_{\mathcal{D}_{\Lambda}}(g) \coloneqq \textup{dd}^{\textup{c}} g + \delta_{\mathcal{D}_{\Lambda}}$ on $\mathcal{X}_{\Lambda}(\mathbb{C})$, where $\textup{d}^{\textup{c}} \coloneqq \frac{i}{4 \pi} (\overline{\partial} - \partial )$ and $\delta_{\mathcal{D}_{\Lambda}}$ is the current of integration along $\mathcal{D}_{\Lambda}$. We say that $g$ is of \textit{plurisubharmonic} type (psh type for short) if $\omega_{\mathcal{D}_{\Lambda}}(g)$ is a positive current. Then, an arithmetic divisor $(\mathcal{D},g_{\mathcal{D}})$ is \textit{semipositive} if $\mathcal{D}$ is relatively nef and $g_{\mathcal{D}}$ is of psh type. By Theorem~1.4~of~\cite{Z95}, given semipositive arithmetic divisors $\overline{\mathcal{D}}_0, \ldots , \overline{\mathcal{D}}_d$ on $\mathcal{X}$, their arithmetic intersection number is defined as the limit
	\begin{displaymath}
		\overline{\mathcal{D}}_0 \cdot \ldots \cdot \overline{\mathcal{D}}_d \coloneqq \lim_{n \rightarrow \infty} \overline{\mathcal{D}}_{0,n} \cdot \ldots \cdot \overline{\mathcal{D}}_{d,n},
	\end{displaymath}
	where each arithmetic divisor $\overline{\mathcal{D}}_{i,n}$ is semipositive and of smooth type, has the same divisorial part as $\overline{\mathcal{D}}_i$, and the sequence of Green's functions $\lbrace g_{\mathcal{D}_{i,n}} \rbrace_{n \in \mathbb{N}}$ converges uniformly to $g_{\mathcal{D}_i}$. Fix an algebraic closure $\overline{K}$ of $K$. Then, a semipositive arithmetic divisor $\overline{\mathcal{D}}$ is \textit{nef} if for each point $x \in \mathcal{X}(\overline{K})$ the height $\textup{h}(x; \overline{\mathcal{D}})$ is non-negative. See~3.1.1~of~\cite{BGS} for the definition of height.
	
	Now, we consider a quasi-projective arithmetic variety $\mathcal{U} / \mathcal{O}_K$. A model arithmetic divisor $\overline{\mathcal{D}}$ is semipositive (resp. nef) if it is represented by an arithmetic divisor satisfying the same property. We denote by $\overline{\textup{Div}}^{+} (\mathcal{U}/\mathcal{O}_K)_{\textup{mod}}$ and $\overline{\textup{Div}}^{\textup{nef}} (\mathcal{U}/\mathcal{O}_K)_{\textup{mod}}$ the cones of model divisors which are semipositive and nef, respectively. These notions naturally generalize to adelic divisors.
	\begin{dfn}\label{global-positivity}
		Let $\mathcal{U} / \mathcal{O}_K$ be a quasi-projective variety. The \textit{semipositive cone} $\overline{\textup{Div}}^{+} (\mathcal{U}/\mathcal{O}_K)$ is the closure in the boundary topology of the set $\overline{\textup{Div}}^{+} (\mathcal{U}/\mathcal{O}_K)_{\textup{mod}}$. The \textit{nef cone} $\overline{\textup{Div}}^{\textup{nef}} (\mathcal{U}/\mathcal{O}_K)$ is defined similarly. The vector space of integrable adelic divisors is the difference
		\begin{displaymath}
			\overline{\textup{Div}}^{\textup{int}} (\mathcal{U}/\mathcal{O}_K) \coloneqq \overline{\textup{Div}}^{\textup{nef}} (\mathcal{U}/\mathcal{O}_K) - \overline{\textup{Div}}^{\textup{nef}} (\mathcal{U}/\mathcal{O}_K).
		\end{displaymath}
	\end{dfn}
	Note that morphisms of projective models of $\mathcal{U}$ over $\mathcal{O}_K$ are proper and birational. Then, the projection formula yields an arithmetic intersection pairing $\overline{\textup{Div}}^{+}(\mathcal{U}/\mathcal{O}_K)_{\textup{mod}}^{d+1} \rightarrow \mathbb{R}$, where the arithmetic intersection number $\overline{\mathcal{D}}_0, \ldots , \overline{\mathcal{D}}_d$ is computed on a projective model of $\mathcal{U}$ on which all of the divisors $\overline{\mathcal{D}}_i$ are defined. Finally, we state Proposition 4.1.1~of~\cite{Y-Z}, where Yuan and Zhang continuously extended the arithmetic intersection pairing to the nef cone.
	\begin{thm}\label{global-intersection}
		The arithmetic intersection pairing on model arithmetic divisors extends continuously to a pairing $\overline{\textup{Div}}^{\textup{nef}}(\mathcal{U}/\mathcal{O}_K)^{d+1} \rightarrow \mathbb{R}$ on the nef cone. It is given by
		\begin{displaymath}
			\overline{\mathcal{D}}_0 \cdot \ldots \cdot \overline{\mathcal{D}}_d \coloneqq \lim_{n \rightarrow \infty} \overline{\mathcal{D}}_{0,n} \cdot \ldots \cdot \overline{\mathcal{D}}_{d,n},
		\end{displaymath}
		where $\overline{\mathcal{D}}_i = \lbrace  \overline{\mathcal{D}}_{i,n} \rbrace_{n \in \mathbb{N}}$ is represented by a Cauchy sequence of nef model arithmetic divisors. The arithmetic intersection numbers $\overline{\mathcal{D}}_0 \cdot \ldots \cdot \overline{\mathcal{D}}_d$ are independent of the choice of nef sequences $\lbrace \overline{\mathcal{D}}_{i,n} \rbrace_{n \in \mathbb{N}}$. By linearity, this extends to a pairing $	\overline{\textup{Div}}^{\textup{int}}(\mathcal{U}/\mathcal{O}_K)^{d+1} \rightarrow \mathbb{R}$.
	\end{thm}
	\begin{rem}
		In the article~\cite{BK24}, Burgos and Kramer extended Yuan and Zhang's arithmetic intersection pairing using the notion of \textit{mixed relative energy}. This construction is based on the \textit{relative energy}~of~\cite{Darvas}. Remark~\ref{finale}, we sketch this construction and show a similar result for toric arithmetic varieties.
	\end{rem}
	
	\section{Toric arithmetic varieties}\label{3}
	In this section, we describe the geometry of projective toric schemes over the ring of integers~$\mathcal{O}_K$ of a number field~$K$. To do this, we first recall the so-called canonical models: toric schemes associated to a fan. Afterwards, we introduce the notion of an \textit{arithmetic fan}, and we show that every projective toric arithmetic variety is given by an arithmetic fan. Finally, we study a special kind of projective toric arithmetic varieties. These are obtained by successive toric blow-ups of a canonical model with center in a finite number of fibers over maximal ideals of $\mathcal{O}_K$. The upshot is that these toric schemes are enough to describe the group of toric adelic divisors over a given toric arithmetic variety, which we introduce in the next section. 
	
	Our references are \cite{CLS} for the theory of toric varieties over a field. When the base scheme is a DVR, we refer to \S3~in~Ch.~IV~of~\cite{KKMS} (although it is hinted that some of their descriptions generalize to Dedekind domains), its summary in Subsection~3.3~of~\cite{Per26}, Sections 3.5--3.7~of~\cite{BPS}, and \S2~in~Ch.IV~of~\cite{FC90}. A modern treatment over any base ring can be found in Ch.~IV of Rohrer's dissertation~\cite{Roh10}. For the case of regular toric schemes over $\mathcal{O}_K$, we refer to \S4 of Demazure's paper~\cite{Dmz70} or Gualdi's Master's thesis~\cite{Gualdi} for an elementary approach over $\mathbb{Z}$. We assume familiarity with the objects from discrete geometry and convex analysis, which typically appear in toric geometry. For reference, we use Appendix~A~of~\cite{Per26}.
	
	\subsection{Toric schemes and canonical models}\label{3-1}
	First, we recall the definitions of toric schemes and toric morphisms. We follow the notation in~Section~3~of~\cite{Per26}. For an affine scheme $\mathcal{S}=\textup{Spec}(R)$, the \textit{multiplicative group} $\mathbb{G}_{m}/\mathcal{S}$ is the affine group scheme $\textup{Spec}(R[T, T^{-1}])$. In the following, the base scheme $\mathcal{S}$ will be either the spectrum of a field $K$ or a Dedekind domain $R$ with field of fractions $K$. For a positive integer $d$, we let $\mathcal{U} \cong \mathbb{G}^{d}_{m}/\mathcal{S}$ be a split $d$-dimensional torus over $\mathcal{S}$. When the base scheme is a Dedekind domain, we denote by $U \cong \mathbb{G}_{m}^{d}/K$ the split torus given by the generic fiber of $\mathcal{U}$.
	\begin{dfn}\label{toric-scheme}
		A $d$-dimensional \textit{toric scheme over} $\mathcal{S}$ is a triple $(\mathcal{X},\pi,\mu)$ consisting of a normal variety $\mathcal{X}/\mathcal{S}$ of relative dimension $d$, an open embedding $\pi \colon U \rightarrow X$ into the generic fiber $X$ of $\mathcal{X}$, and an $\mathcal{S}$-group scheme action $\mu \colon \mathcal{U} \times_{\mathcal{S}} \mathcal{X} \rightarrow \mathcal{X}$ extending the action of $U$ on itself by translations.
	\end{dfn}
	We often obviate the embedding and action to lighten up notation and say that $\mathcal{X}/\mathcal{S}$ is a toric scheme. When the base scheme is a field, $\mathcal{X}$ coincides with its generic fiber and we recover the usual definition of a toric variety (for instance, Definition~3.1.1~of~\cite{CLS}). We have the corresponding notions for divisors and morphisms.
	\begin{dfn}
		Let $\mathcal{X}/\mathcal{S}$ be a toric scheme. A divisor $\mathcal{D} \in \textup{Div}( \mathcal{X})_{\mathbb{Q}}$ is \textit{toric} if it satisfies $\mu^{\ast} \mathcal{D} = p_{2}^{\ast} \mathcal{D}$, where $\mu$ is the action of $\mathcal{U}$ and $p_{2} \colon \mathcal{U} \times_{\mathcal{S}}  \mathcal{X} \rightarrow \mathcal{X}$ is the projection to the second factor. Denote by $ \textup{Div}_{\mathbb{T}}( \mathcal{X})_{\mathbb{Q}}$ the group of toric Cartier $\mathbb{Q}$-divisors on $\mathcal{X}$.
	\end{dfn}
	\begin{dfn}
		For each $i \in \lbrace 1,2 \rbrace$, let $(\mathcal{X}_i, \pi_i, \mu_i)$ be a toric scheme over $\mathcal{S}$ and $\rho \colon \mathcal{U}_{1} \rightarrow \mathcal{U}_{2}$ be an $\mathcal{S}$-homomorphism between their tori. An $\mathcal{S}$-morphism $f \colon \mathcal{X}_1 \rightarrow \mathcal{X}_2$ is $\rho$\textit{-equivariant} if the following diagram
		\begin{center}
			\begin{tikzcd}
				\mathcal{U}_{1} \times_{\mathcal{S}} \mathcal{X}_1 \arrow[d, "{\rho \times f}"] \arrow[r, "{\mu_1}"] & \mathcal{X}_1  \arrow[d, "{f}"] \\
				\mathcal{U}_{2} \times_{\mathcal{S}} \mathcal{X}_2  \arrow[r, "{\mu_2}"] & \mathcal{X}_2 
			\end{tikzcd}
		\end{center}
		commutes. A $\rho$-equivariant morphism $f$ is $\rho$-\textit{toric} if it restricts to a morphism of tori $U_1 \rightarrow U_2$. If the homomorphism $\rho$ is clear, we say that $f \colon \mathcal{X}_1 \rightarrow \mathcal{X}_2$ is equivariant (resp. toric).
	\end{dfn}
	The notion of a projective model from Subsection~\ref{2-1} specializes to the toric setting.
	\begin{dfn}
		Let  $\mathcal{X}_0 /\mathcal{S}$ be a quasi-projective toric scheme. A projective model $\pi \colon \mathcal{X}_0 \rightarrow \mathcal{X}$ over $\mathcal{S}$ is \textit{toric} if both $\mathcal{X}/\mathcal{S}$ and $\pi$ are toric. A \textit{morphism of toric projective models of} $\mathcal{X}_0$ \textit{over} $\mathcal{S}$ is a morphism of projective models which is toric. We denote by $\textup{PM}_{\mathbb{T}}(\mathcal{X}_0 /\mathcal{S})$ the category of toric projective models of $\mathcal{X}_{0}$ over $\mathcal{S}$.
	\end{dfn}
	Denote by $M$ and $N$ the groups of \textit{characters} and \textit{cocharacters} of the torus $\mathcal{U}$. Both $M$ and $N$ are lattices of rank~$d$, and they are dual to each other. For each character $m \in M$, we denote by $\chi^{m}$ its corresponding element in the group algebra $R \left[ M \right]$. A fan $\Sigma$ in $N_{\mathbb{R}}$ is a strongly convex, conical, rational, polyhedral complex. Given such an object, there is an associated toric scheme.
	\begin{cons}\label{canonical-OK-prop}[Canonical models]
		Let the base scheme $\mathcal{S}$ be either the spectrum of a field $K$ or a Dedekind domain $R$ with field of fractions $K$. For a fan $\Sigma$ in $N_{\mathbb{R}}$, the toric scheme $\mathcal{X}_{\Sigma}/ \mathcal{S}$ is defined as follows:
		\begin{enumerate}
			\item For each $\sigma \in \Sigma$, denote by $M_{\sigma} \coloneqq \sigma^{\vee} \cap M $ the semigroup of lattice points in the dual cone $\sigma^{\vee} \coloneqq \lbrace x \in M_{\mathbb{R}} \, \vert \, \langle x, u \rangle \geq 0 \, \textup{for all } u \in \sigma\rbrace$. Then, $M_{\sigma}$ generates the commutative $R$-algebra $R\left[ M_{\sigma} \right]$. Define the affine scheme $\mathcal{X}_{\sigma} \coloneqq \textup{Spec}( R \left[ M_{\sigma}\right] )$. This is an affine scheme over $R$. The collection $\lbrace \mathcal{X}_{\sigma} \rbrace_{\sigma \in \Sigma}$ of affine schemes constitutes the building blocks of $\mathcal{X}_{\Sigma}$.
			\item Let $\tau$ be a face of $\sigma$. By definition, there exists $m \in M_{\sigma}$ such that $\tau = \sigma \cap H_m$, where $H_m$ is the hyperplane determined by the equation $\langle m, \cdot \rangle = 0$. The scheme $\mathcal{X}_{\tau}$ corresponds to the principal open set $(\mathcal{X}_{\sigma})_m$, given by the localization of $R \left[ M_{\sigma} \right]$ at $\chi^{m}$. This induces an open embedding $\mathcal{X}_{\tau} \rightarrow \mathcal{X}_{\sigma}$, identifying $\mathcal{X}_{\tau}$ with an open subset of $\mathcal{X}_{\sigma}$. Every pair of cones $ \sigma_1, \sigma_2 \in \Sigma$ intersect at a common face $ \tau$, hence the identifications $\mathcal{X}_{\tau} \rightarrow \mathcal{X}_{\sigma_i}$, $i  \in \lbrace 1,2 \rbrace$, allows us to glue the family $\lbrace \mathcal{X}_{\sigma} \rbrace_{\sigma \in \Sigma}$.
			\item The affine scheme $\mathcal{X}_{\lbrace 0 \rbrace}$ is isomorphic to the torus $\mathcal{U}$. For each cone $\sigma$, the assignment $\chi^m \mapsto \chi^m \otimes \chi^m$ determines a ring morphism $\mu^{\sharp}_{\sigma} \colon R \left[  M_{\sigma} \right] \rightarrow R \left[  M_{\sigma} \right] \otimes_R R \left[ M \right]$. This induces an action $\mu_{\sigma} \colon \mathcal{U} \times_{\mathcal{S}} \mathcal{X}_{\sigma} \rightarrow \mathcal{X}_{\sigma}$, and the family $\lbrace \mu_{\sigma} \rbrace_{\sigma \in \Sigma}$ glues into an action $\mu$ of $\mathcal{U}$ on $\mathcal{X}_{\Sigma}$, giving it the structure of toric scheme.
		\end{enumerate}
		The scheme $\mathcal{X}_{\Sigma}$ is noetherian, integral, normal, separated, quasi-compact, flat, and of finite presentation over $\mathcal{S}$, with relative dimension $\textup{dim}_{\mathcal{S}} (\mathcal{X}_{\Sigma}) = d$ (Propositions~IV.1.2.1--IV.1.2.5~of~\cite{Roh10}). Moreover, $\mathcal{X}_{\Sigma}/ \mathcal{S}$ is proper (resp. regular) if and only if the fan $\Sigma$ is complete (resp. smooth) (Theorem~IV.1.3.12~of~\cite{Roh10} and Proposition~1~in~\S4~of~\cite{Dmz70}). The assignment $\Sigma \mapsto \mathcal{X}_{\Sigma}$ induces a functor from the category of fans in $N_{\mathbb{R}}$ with compatible maps (Definition~A.3.5~of~\cite{Per26}) to the category of toric schemes over $\mathcal{S}$ with toric morphisms. Fixing a fan $\Sigma$, the assignment $\mathcal{S} \mapsto \mathcal{X}_{\Sigma}$ is functorial and isomorphic to the base change functor (Paragraph~IV.1.1.9~of~\cite{Roh10}). In particular, the generic fiber of the toric scheme $\mathcal{X}_{\Sigma}/\mathcal{S}$ is the toric variety $X_{\Sigma}/K$. Therefore, $\mathcal{X}_{\Sigma}$ is called the \textit{canonical model} over $\mathcal{S}$ of the toric variety $X_{\Sigma}/K$. 
	\end{cons} 
	\begin{rem}
		For simplicity, most of the time we will assume that the fan $\Sigma$ is smooth and complete, and therefore $\mathcal{X}_{\Sigma}/\mathcal{S}$ is proper and regular. For the purpose of studying adelic divisors, this is not a severe restriction; by Theorem~11.1.9~of~\cite{CLS}, every complete fan $\Sigma_2$ can be refined into a smooth fan $\Sigma_1$. This refinement induces a proper birational toric $\mathcal{S}$-morphism $\mathcal{X}_{\Sigma_1} \rightarrow \mathcal{X}_{\Sigma_2}$.
	\end{rem}
	\begin{desc}\label{canonical-fan}
		When the base $\mathcal{S}$ is a field $K$, the assignment $\Sigma \mapsto X_{\Sigma}/K$ induces an equivalence between the category of complete fans with compatible maps, and the category of proper toric varieties over $K$ with toric morphisms. If $\mathcal{S}$ is a DVR, the functor $\Sigma \mapsto \mathcal{X}_{\Sigma}$ fails to be essentially surjective. Instead, one must consider fans $\widetilde{\Sigma}$ in the upper half-space $N_{\mathbb{R}} \times \mathbb{R}_{\geq 0}$. Then, there is an associated toric scheme $\mathcal{X}_{\widetilde{\Sigma}}/\mathcal{S}$, and the functor $\widetilde{\Sigma} \mapsto \mathcal{X}_{\widetilde{\Sigma}}$ is an equivalence between the categories of complete fans and proper toric schemes (See~Construction~3.3.1 and Proposition~3.3.2~of~\cite{Per26}). Given a fan $\Sigma$ in $N_{\mathbb{R}}$, there are two obvious ways to obtain a fan in $N_{\mathbb{R}} \times \mathbb{R}_{\geq 0}$:
		\begin{enumerate}
			\item The first option is to consider the fan $\Sigma \times \lbrace 0 \rbrace$ in the hyperplane $N_{\mathbb{R}} \times \lbrace 0 \rbrace$. Regarding it as a fan in  $N_{\mathbb{R}} \times \mathbb{R}_{\geq 0}$, the toric scheme $\mathcal{X}_{\Sigma \times \lbrace 0 \rbrace} /\mathcal{S}$ is isomorphic to the toric variety $X_{\Sigma}/K$.
			\item The second option is to consider its \textit{canonical extension} $\Sigma_{\textup{can}}$ in $N_{\mathbb{R}} \times \mathbb{R}_{\geq 0}$. This is the fan consisting of the cones $\sigma \times \lbrace 0 \rbrace$ in $\Sigma \times \lbrace 0 \rbrace$, and all the cones spanned by $\sigma \times \lbrace 0 \rbrace$ and the vertical vector $(0,1)$. In this case, the associated toric scheme $\mathcal{X}_{\Sigma_{\textup{can}}}/\mathcal{S}$ is isomorphic to the toric scheme $\mathcal{X}_{\Sigma}/\mathcal{S}$ from Construction~\ref{canonical-OK-prop}.
		\end{enumerate}
		To avoid confusion, we will use the notation $\mathcal{X}_{\Sigma_{\textup{can}}}$ when dealing with toric schemes over a DVR.
	\end{desc}
	When the base scheme is the spectrum of a field or a DVR, there is a correspondence between the cones in a fan and the torus orbits in the associated toric scheme. We refer to Construction~3.2.4 and Remark~3.3.3 of~\cite{Per26} for a brief description of these correspondences. From now on and until the end of this subsection, we focus on the case of $\mathcal{S}$ being the spectrum of the ring of integers $\mathcal{O}_K$ of a number field $K$. We state below a result of Demazure, which establishes an analogous correspondence for a regular toric scheme $\mathcal{X}/\mathcal{O}_K$  (Proposition~2~in~\S4~of~\cite{Dmz70}). Then, we use this correspondence to characterize the torus-invariant divisors on $\mathcal{X}_{\Sigma}/\mathcal{O}_K$.
	\begin{cons}[The cone-orbit correspondence over $\mathcal{O}_K$]\label{cone-orbit-OK}
		Let $\Sigma$ be a smooth fan in $N_{\mathbb{R}}$ and let $\mathcal{X}_{\Sigma}/\mathcal{O}_K$ be its associated toric scheme. Given a cone $\sigma \in \Sigma$, we consider the lattices $N(\sigma) \coloneqq N/(N \cap \textup{Span}(\sigma))$ and $M(\sigma) \coloneqq N(\sigma)^{\vee} = M \cap \sigma^{\perp}$, where $\sigma^{\perp}$ is the linear space orthogonal to $\sigma$ with respect to the canonical pairing. For each cone $\tau$ containing $\sigma$, write $\overline{\tau}$ for its image in the space $N(\sigma)_{\mathbb{R}}$. Observe that the collection $\Sigma(\sigma) \coloneqq \lbrace \overline{ \tau } \, | \, \tau \in \Sigma, \,  \sigma \leq \tau \rbrace$ is a fan in $N(\sigma)_{\mathbb{R}}$. Then:
		\begin{enumerate}
			\item For each cone $\sigma \in \Sigma$, there is a locally closed immersion $\mathbf{O}(\sigma) \rightarrow \mathcal{X}_{\Sigma}$, where $\mathbf{O}(\sigma)$ is the split torus $\textup{Spec}(\mathcal{O}_K [M(\sigma)] )$. Moreover, we have a disjoint union of the underlying sets
			\begin{displaymath}
				\mathcal{X}_{\Sigma} = \bigcup_{\sigma \in \Sigma} \mathbf{O}(\sigma).
			\end{displaymath}
			Therefore, for each point $x \colon \textup{Spec}(L) \rightarrow \mathcal{X}_{\Sigma}$ there exist $\sigma \in \Sigma$ such that $x$ is a point of $\mathbf{O}(\sigma)$, and its $\mathcal{U}$-orbit is equal to the fiber product $\mathbf{O}(\sigma) \times_{\mathcal{O}_K} \textup{Spec}(L)$.
			\item  For each cone $\sigma \in \Sigma$, denote by $\mathbf{V}(\sigma)$ the closure of $\mathbf{O}(\sigma)$ in $\mathcal{X}_{\Sigma}$, equipped with its unique reduced scheme structure. Then, the closed immersion $\mathcal{X}_{\Sigma(\sigma)} \rightarrow \mathcal{X}_{\Sigma}$ identifies $\mathbf{V}(\sigma)$ with $\mathcal{X}_{\Sigma(\sigma)}$. By (i), we have a disjoint union of underlying sets
			\begin{displaymath}
				\mathcal{X}_{\Sigma(\sigma)} = \bigcup_{\sigma \leq \tau} \mathbf{O}(\tau).
			\end{displaymath}
		\end{enumerate}
		Observe that $\textup{dim}_{\mathcal{O}_K}(\mathcal{X}_{\Sigma(\sigma)}) = \textup{dim}_{\mathcal{O}_K}(\mathbf{O}(\sigma)) = d - \textup{dim}(\sigma)$. Then, there are two kinds of torus-invariant irreducible subvarieties of codimension $p \geq 1$ in $\mathcal{X}_{\Sigma}/\mathcal{O}_K$:
		\begin{enumerate}[label=(\alph*)]
			\item The \textit{horizontal} toric schemes $\mathcal{X}_{\Sigma(\sigma)}/\mathcal{O}_K$, where the cone $\sigma \in \Sigma$ has dimension $p$.
			\item The \textit{vertical} toric varieties $\mathcal{X}_{\Sigma(\tau), k(\mathfrak{p})} \coloneqq \mathcal{X}_{\Sigma(\tau)} \times_{\mathcal{O}_K} \textup{Spec}(k(\mathfrak{p}))$, where the cone $\tau \in \Sigma$ has dimension $p-1$ and $k(\mathfrak{p})$ is the residue field of a maximal ideal of $\mathcal{O}_K$.
		\end{enumerate}
		We call the assignment $\sigma \mapsto \mathbf{O}(\sigma)$ the \textit{cone-orbit correspondence} for the toric scheme $\mathcal{X}_{\Sigma} / \mathcal{O}_K$.
	\end{cons}
	\begin{prop}\label{toric-div-canonical-OK}
		Let $\Sigma$ be a smooth fan in $N_{\mathbb{R}}$. Then, the isomorphism between Cartier and Weil divisors on the regular scheme $\mathcal{X}_{\Sigma}/\mathcal{O}_K$ induces an isomorphism
		\begin{displaymath}
			\textup{Div}_{\mathbb{T}}(\mathcal{X}_{\Sigma}) \cong \bigoplus_{\rho \in \Sigma(1)} \mathbb{Z} \, \oplus \bigoplus_{\mathfrak{p} \in \textup{Max}(\mathcal{O}_K)} \mathbb{Z},
		\end{displaymath}
		where $\Sigma(1)$ is the set of rays of $\Sigma$ and $ \textup{Max}(\mathcal{O}_K)$ is the set of maximal ideals of $\mathcal{O}_{K}$. In particular, every toric divisor decomposes as a sum $\mathcal{D} = \mathcal{D}_{\textup{hor}} + \mathcal{D}_{\textup{vert}}$ of its horizontal and vertical components. 
	\end{prop}
	\begin{proof}
		Let $\mathcal{D}$ be a toric Cartier divisor on $\mathcal{X}_{\Sigma}/\mathcal{O}_K$ and $| \mathcal{D} |$ be its support. Then, for each point  $x \in |\mathcal{D}|$, its $\mathcal{U}$-orbit is contained in $|\mathcal{D}|$. By the cone-orbit correspondence~\ref{cone-orbit-OK} and noetherianity of the scheme $\mathcal{X}_{\Sigma}$, the closed set $|\mathcal{D}|$ is equal to the finite union
		\begin{displaymath}
			|\mathcal{D}| = \bigcup_{i=1}^{n_1} \mathcal{X}_{\Sigma(\rho_i)} \, \cup \, \bigcup_{j=1}^{n_2} \mathcal{X}_{\Sigma, k(\mathfrak{p}_j)},
		\end{displaymath}
		where each $\rho_i$ is a ray of $\Sigma$ and each $\mathfrak{p}_j$ is a maximal ideal of $\mathcal{O}_{K}$. Then, there exist integers $a_i, b_j \in \mathbb{Z}$ such that the associated Weil divisor $\mathcal{D}$ is of the form
		\begin{displaymath}
			\mathcal{D} = \sum_{i=1}^{n_1} a_i \cdot  \mathcal{X}_{\Sigma(\rho_i)}  + \sum_{j=1}^{n_2} b_j \cdot   \mathcal{X}_{\Sigma, k(\mathfrak{p}_j)}.
		\end{displaymath}
		Note that the first summand is a horizontal toric divisor, while the second is a vertical toric divisor. Conversely, every Weil divisor as above is a toric Cartier divisor. The result follows.
	\end{proof}
	\begin{desc}\label{toric-divisors-K-DVR}
		When the base scheme $\mathcal{S}$ is a field or a DVR, the above isomorphism can be strengthened to an isomorphism between the vector space of toric $\mathbb{Q}$-divisors and the vector space of rational support functions on the associated fan. These are real-valued functions that are linear on each cone of a fan, and their values at lattice points are rational numbers. Concretely:
		\begin{enumerate}
			\item Let $\Sigma$ be a complete fan in $N_{\mathbb{R}}$ and $X_{\Sigma}/K$ be its associated toric variety. Given a toric divisor $D$ on $X_{\Sigma}$, its associated Weil divisor is given by
			\begin{displaymath}
				D = \sum_{\rho \in \Sigma(1)} a_{\rho} \cdot X_{\Sigma(\rho)}.
			\end{displaymath}
			Then, its associated support function $\Psi_{D} \colon N_{\mathbb{R}} \rightarrow \mathbb{R}$ is determined by the fact that it is linear on each cone of $\Sigma$, and for each ray $\rho \in \Sigma(1)$ we have $\Psi_D(v_{\rho}) = - a_{\rho}$, where $v_{\rho}$ is the \textit{ray generator} of $\rho$ (meaning $\mathbb{N} \cdot v_{\rho} = \rho \cap N$). Then, the assignment $D \mapsto \Psi_D$ induces an isomorphism $\mathcal{SF} \colon \textup{Div}_{\mathbb{T}}(X_{\Sigma})_{\mathbb{Q}} \rightarrow \mathcal{SF}(\Sigma, \mathbb{Q})$, where $\mathcal{SF}(\Sigma, \mathbb{Q})$ is the vector space of rational support functions on the fan $\Sigma$. For details, see Definition~A.3.10 and Proposition~3.2.5~of~\cite{Per26}.
			\item Let $\widetilde{\Sigma}$ be a complete fan in $N_{\mathbb{R}}\times \mathbb{R}_{\geq 0}$ and $\mathcal{X}_{\widetilde{\Sigma}}/\mathcal{S}$ be its associated toric scheme. Given a toric divisor $\mathcal{D}$ on $\mathcal{X}_{\widetilde{\Sigma}}$, its support function $\Phi_{\mathcal{D}} \colon N_{\mathbb{R}}\times \mathbb{R}_{\geq 0} \rightarrow \mathbb{R} $ is defined in the same way as in (i). Then, there is an isomorphism  $\mathcal{SF} \colon \textup{Div}_{\mathbb{T}}(\mathcal{X}_{\widetilde{\Sigma}})_{\mathbb{Q}} \rightarrow \mathcal{SF}(\widetilde{\Sigma}, \mathbb{Q})$. This isomorphism can also be stated in terms of rational piecewise affine functions. Indeed, the function $\gamma_{\mathcal{D}} \colon N_{\mathbb{R}} \rightarrow \mathbb{R}$ given by restricting $\Phi_{\mathcal{D}}$ to the hyperplane $N_{\mathbb{R}} \times \lbrace 1 \rbrace$ is a rational piecewise affine function on the polyhedral complex $\Pi$ obtained by intersecting each cone of $\widetilde{\Sigma}$ with the same hyperplane. Then, there is an isomorphism $\mathcal{PA} \colon  \textup{Div}_{\mathbb{T}}(\mathcal{X}_{\widetilde{\Sigma}})_{\mathbb{Q}} \rightarrow \mathcal{PA}(\Pi, \mathbb{Q})$, where $\mathcal{PA}(\Pi, \mathbb{Q})$ is the space of rational piecewise affine functions on $\Pi$. For details, see Proposition~3.3.4~of~the~loc.~cit.
		\end{enumerate}
		Under these isomorphisms, a toric divisor is nef (resp. ample) if and only if the corresponding function is concave (resp. strictly concave). See Theorem~3.2.7~and~Proposition~3.3.6~of~\cite{Per26}.
	\end{desc}
	The next result relates the toric divisors on $\mathcal{X}_{\Sigma}/\mathcal{O}_K$ with the functions in the above description.
	\begin{cor}\label{div-fun-can-OK}
		Let $\Sigma$ be a smooth fan in $N_{\mathbb{R}}$ and $\mathcal{X}_{\Sigma}/\mathcal{O}_K$ be its associated toric scheme. For each toric divisor $\mathcal{D}$ on $\mathcal{X}_{\Sigma}$ and each maximal ideal $\mathfrak{p}$ of $\mathcal{O}_K$, denote by $D$ and $\mathcal{D}_{\mathfrak{p}}$ its restrictions to the generic fiber $X_{\Sigma}/K$, and to the base change $\mathcal{X}_{\Sigma,\mathfrak{p}} \coloneqq \mathcal{X}_{\Sigma} \times_{\mathcal{O}_K} \textup{Spec}(\mathcal{O}_{K,\mathfrak{p}})$. Then:
		\begin{enumerate}[label=(\roman*)]
			\item The map $\textup{Div}_{\mathbb{T}}(\mathcal{X}_{\Sigma})_{\mathbb{Q}} \rightarrow \mathcal{SF}(\Sigma ,\mathbb{Q})$ given by the assignment $\mathcal{D} \mapsto \Psi_{D}$ induces an isomorphism between the subspace of $\textup{Div}_{\mathbb{T}}(\mathcal{X}_{\Sigma})_{\mathbb{Q}}$ consisting of horizontal toric divisors and the space of rational support functions $\mathcal{SF}(\Sigma ,\mathbb{Q})$. 
			\item The map $\textup{Div}_{\mathbb{T}}(\mathcal{X}_{\Sigma})_{\mathbb{Q}} \rightarrow \mathcal{PA}(\Sigma ,\mathbb{Q})$ given by the assignment $\mathcal{D} \mapsto \gamma_{\mathcal{D}_{\mathfrak{p}}}$ is surjective. Its kernel consists of the vertical toric divisors supported outside of the vertical fiber $\mathcal{X}_{\Sigma, k(\mathfrak{p})}$. Moreover, the rational piecewise affine function of $\mathcal{D}_{\mathfrak{p}}$ satisfies $\gamma_{\mathcal{D}_{\mathfrak{p}}} = \Psi_D - b_{\mathfrak{p}}$, where $b_{\mathfrak{p}}$ is the order of $\mathcal{D}$ along the vertical fiber $\mathcal{X}_{\Sigma, k(\mathfrak{p})}$.
		\end{enumerate}
	\end{cor}
	\begin{proof}
		By Proposition~\ref{toric-div-canonical-OK}, we have a decomposition
		\begin{displaymath}
			\mathcal{D} = \mathcal{D}_{\textup{hor}} + \mathcal{D}_{\textup{vert}} = \sum_{\rho \in \Sigma(1)} a_{\rho}  \cdot  \mathcal{X}_{\Sigma(\rho)} + \sum_{\mathfrak{q} \in \textup{Max}(\mathcal{O}_K)} b_{\mathfrak{q}}  \cdot  \mathcal{X}_{\Sigma, k(\mathfrak{q})},
		\end{displaymath}
		where $a_{\rho},b_{\mathfrak{q}} \in \mathbb{Q}$ and $b_{\mathfrak{q}} = 0$ for almost all maximal ideals $\mathfrak{q}$ of $\mathcal{O}_K$. Then, we get
		\begin{displaymath}
			D = \sum_{\rho \in \Sigma(1)} a_{\rho} \cdot X_{\Sigma(\rho)}, \quad \mathcal{D}_{\mathfrak{p}} = \sum_{\rho \in \Sigma(1)} a_{\rho} \cdot \mathcal{X}_{\Sigma(\rho),\mathfrak{p}} + b_{\mathfrak{p}} \cdot \mathcal{X}_{\Sigma, k(\mathfrak{p})},
		\end{displaymath}
		where $X_{\Sigma(\rho)}/K$ and $\mathcal{X}_{\Sigma(\rho),\mathfrak{p}}/ \mathcal{O}_{K,\mathfrak{p}}$ are the generic fiber and base change to $\mathcal{O}_{K,\mathfrak{p}}$ of the horizontal toric scheme $\mathcal{X}_{\Sigma(\rho)}/ \mathcal{O}_K$.
		
		We first show \textit{(i)}. By Description~\ref{toric-divisors-K-DVR}.(i), the support function $\Psi_D$ is determined by the values $\Psi_D(v_{\rho}) = - a_{\rho}$ on the ray generators $v_{\rho}$. Since the fan $\Sigma$ is smooth, the ray generators $v_{\rho}$ in a cone $\sigma$ are $\mathbb{Z}$-linearly independent. Therefore, any choice of rational numbers $a_{\rho}$ determines a piecewise linear function on $\Sigma$. We conclude that the map $\textup{Div}_{\mathbb{T}}(\mathcal{X}_{\Sigma})_{\mathbb{Q}} \rightarrow \mathcal{SF}(\Sigma ,\mathbb{Q})$ is surjective, and its kernel consists of the vertical toric divisors. The isomorphism follows the fact that $\textup{Div}_{\mathbb{T}}(\mathcal{X}_{\Sigma})$ decomposes as the direct sum of the horizontal and vertical toric divisors. 
		
		Now, we show part~\textit{(ii)}. By~Description~\ref{canonical-fan}, the toric scheme $\mathcal{X}_{\Sigma,\mathfrak{p}}/ \mathcal{O}_{K,\mathfrak{p}}$ is isomorphic to $\mathcal{X}_{\Sigma_{\textup{can}}}/ \mathcal{O}_{K,\mathfrak{p}}$, and its special fiber $ \mathcal{X}_{\Sigma, k(\mathfrak{p})}$ corresponds with the ray in $\Sigma_{\textup{can}}$ spanned by $(0,1)$. Moreover, the polyhedral complex obtained by intersecting $\Sigma_{\textup{can}}$ with the hyperplane $N_{\mathbb{R}} \times \lbrace 1 \rbrace$ is exactly the fan $\Sigma$. The decomposition of $\mathcal{D}_{\mathfrak{p}}$ shows that the support function $\Phi_{\mathcal{D}_{\mathfrak{p}}} \colon N_{\mathbb{R}} \times \mathbb{R}_{\geq 0} \rightarrow \mathbb{R}$ is given by
		\begin{displaymath}
			\Phi_{\mathcal{D}_{\mathfrak{p}}}(x,t) = \Psi_{D} (x) - b_{\mathfrak{p}} \cdot t.
		\end{displaymath}
		By definition, $\gamma_{\mathcal{D}_{\mathfrak{p}}}(x) = \Phi_{\mathcal{D}_{\mathfrak{p}}}(x,1) = \Psi_{D} (x) - b_{\mathfrak{p}}$. Then, the result follows.
	\end{proof}
	The following result combines  Lemma~4 and Corollary~1 in \S4~of~\cite{Dmz70}, characterizing positivity properties of a horizontal toric divisor in terms of concavity of its support function.
	\begin{prop}\label{ample-OK}
		Let $\Sigma$ be a smooth complete fan in $N_{\mathbb{R}}$, $\mathcal{X}_{\Sigma}/\mathcal{O}_K$ be its associated toric scheme, and $\mathcal{D}$ be a horizontal toric divisor on $\mathcal{X}_{\Sigma}$. Then, the divisor $\mathcal{D}$ is ample (resp. relatively nef) if and only if the support function $\Psi_D$ of its restriction $D$ to the generic fiber is strictly concave (resp. concave).
	\end{prop}
	A toric scheme $\mathcal{X}_{\Sigma}/ \mathcal{O}_K$ is \textit{torically quasi-projective} if there exists a toric projective model of $\mathcal{X}_{\Sigma}$ over $\mathcal{O}_K$. Now, recall that a fan $\Sigma$ in $N_{\mathbb{R}}$ is \textit{projective} if it is complete and there exists a strictly concave rational support function on $\Sigma$ (Lemma~6.1.13~of~\cite{CLS}). Similarly, the fan $\Sigma$ is \textit{quasi-projective} if it is a subfan of a projective fan. As an application of the above proposition, we characterize (quasi-)projectivity of a canonical model. 
	\begin{cor}\label{quasi-proj-OK}
		Let $\Sigma$ be a smooth fan in $N_{\mathbb{R}}$. The following statements hold:
		\begin{enumerate}[label=(\roman*)]
			\item The toric scheme $\mathcal{X}_{\Sigma}/ \mathcal{O}_K$ is projective if and only if the fan $\Sigma$ is projective.
			\item The toric scheme $\mathcal{X}_{\Sigma}/ \mathcal{O}_K$ is torically quasi-projective if and only if $\Sigma$ is quasi-projective. Moreover, if the support $|\Sigma|$ is convex, then quasi-projective and torically quasi-projective are equivalent.
		\end{enumerate}
	\end{cor}
	\begin{proof}
		Part~\textit{(i)} follows from Proposition~\ref{ample-OK} and Remark~3.2.8~of~\cite{Per26}. The second part is proven in the same way as Theorem~7.2.4~of~\cite{CLS}.
	\end{proof}
	
	\subsection{Toric arithmetic varieties and arithmetic fans}\label{3-2}
	Now, we use the descriptions in the previous section to study projective toric schemes over the base scheme $\mathcal{S} \coloneqq \textup{Spec}(\mathcal{O}_K)$. We will use the conventions from Subsection~\ref{notation} freely, and introduce the following notation. Given a finite set of maximal ideals $S \subset \textup{Max}(\mathcal{O}_K)$, we denote by $\mathcal{O}_{K,S}$ the ring of $S$-integers of $K$. We emphasize that if $S= \lbrace \mathfrak{p} \rbrace$, the rings $\mathcal{O}_{K,\lbrace \mathfrak{p} \rbrace}$ and $\mathcal{O}_{K,\mathfrak{p}}$ are not the same. Then, for each scheme $\mathcal{X}$ over $\mathcal{O}_K$, we consider the following fiberd products:
	\begin{itemize}
		\item The generic fiber $\mathcal{X}_K \coloneqq \mathcal{X} \times_{\mathcal{O}_K} \textup{Spec}(K)$ of $\mathcal{X}$. In particular, $\mathcal{S}_K \coloneqq \textup{Spec}(K)$. If $\mathcal{X}/\mathcal{O}_K$ is a variety, we also use the notation $X \coloneqq \mathcal{X}_K$.
		\item The fiber $\mathcal{X}_{k(\mathfrak{p})} \coloneqq \mathcal{X} \times_{\mathcal{O}_K} \textup{Spec}(k(\mathfrak{p}))$ of $\mathcal{X}$ over $\mathfrak{p}$.
		\item The fiber $\mathcal{X}_{\mathfrak{p}} \coloneqq \mathcal{X} \times_{\mathcal{O}_{K}} \textup{Spec}(\mathcal{O}_{K,\mathfrak{p}})$ of $\mathcal{X}$ over $\mathcal{O}_{K,\mathfrak{p}}$. Sometimes, we also refer to it as the \textit{localization} of $\mathcal{X}$ at $\mathfrak{p}$. In particular, $\mathcal{S}_{\mathfrak{p}} = \textup{Spec}(\mathcal{O}_{K,\mathfrak{p}})$.
		\item The fiber $\mathcal{X}_S \coloneqq \mathcal{X} \times_{\mathcal{O}_{K}}  \textup{Spec}(\mathcal{O}_{K,S})$ of $\mathcal{X}$ over $\mathcal{O}_{K,S}$. In particular, $\mathcal{S}_S = \textup{Spec}(\mathcal{O}_{K,S})$ identifies with an open subscheme of $\mathcal{S} = \textup{Spec}(\mathcal{O}_{K})$.
	\end{itemize}
	The main idea of this subsection is to understand the geometry of a toric scheme $\mathcal{X}/\mathcal{O}_K$ in terms of the fibers $\mathcal{X}_S$ and $\mathcal{X}_{\mathfrak{p}}$ for a carefully chosen set $S$. To do this, we will use the theory of \textit{fpqc descent}. Our references are Chapter~6~of~\cite{BLR90} and Chapter~14~of~\cite{GW20}. We will also give a description in terms of toric blow-ups. First, we need some preparation. 
	\begin{lem}\label{generic-canonical}
		Let $\mathcal{X}/\mathcal{O}_K$ be a proper toric scheme which is also a variety (see Subsection~\ref{notation}), and let $\mu \colon \mathcal{U} \times_{\mathcal{O}_K} \mathcal{X} \rightarrow \mathcal{X}$ be its torus action. Then, there is a finite set $S \subset \textup{Max}(\mathcal{O}_K)$ of maximal ideals and a complete fan $\Sigma$ in $N_{\mathbb{R}}$ such that the toric schemes $\mathcal{X}_{S}$ and $\mathcal{X}_{\Sigma,S}$ are isomorphic.
	\end{lem}
	\begin{proof}
		This is a standard argument in the theory of schemes over a Dedekind domain. Since $\mathcal{X}/\mathcal{O}_K$ is proper, the toric variety $X/K$ given by its generic fiber is also proper. By the classification of proper toric varieties over a field, there exists a complete fan $\Sigma$ in $N_{\mathbb{R}}$ such that $X$ and $X_{\Sigma}$ are isomorphic as toric varieties. Let $f \colon X_{\Sigma} \rightarrow X$ be said toric isomorphism and consider its graph
		\begin{displaymath}
			\Gamma_f \colon X_{\Sigma} \longrightarrow X_{\Sigma} \times_K X.
		\end{displaymath}
		Since $X_{\Sigma}/K$ is a variety, $\Gamma_f$ is a closed immersion, and its scheme-theoretic image $Z$ inherits the structure of a toric variety. Now, let $\iota \colon X_{\Sigma} \times_K X \rightarrow \mathcal{X}_{\Sigma} \times_{\mathcal{O}_K} \mathcal{X}$ be the canonical morphism identifying the toric variety $X_{\Sigma} \times_K X$ as the generic fiber of the toric scheme $\mathcal{X}_{\Sigma} \times_{\mathcal{O}_K} \mathcal{X}$. Then, define $\mathcal{Z}$ as the scheme-theoretic image of the composition
		\begin{displaymath}
			\iota \circ \Gamma_f \colon X_{\Sigma} \longrightarrow \mathcal{X}_{\Sigma} \times_{\mathcal{O}_K} \mathcal{X}.
		\end{displaymath}
		Note that $\mathcal{Z}$ is a model of $Z$, and therefore, it is flat over $\mathcal{O}_K$ (See~Exercise~14.5~of~\cite{GW20}). Now, let $p_1 \colon \mathcal{Z} \rightarrow \mathcal{X}_{\Sigma}$ be the projection to $\mathcal{X}_{\Sigma}$, and it induces a toric isomorphism on generic fibers $p_1 \colon Z \rightarrow X_{\Sigma}$. By Proposition~14.45~of~the~loc.~cit., there is a finite set $S_1 \subset \textup{Max}(\mathcal{O}_K)$ such that the base change $p_{1,S_1} \colon \mathcal{Z}_{S_1} \rightarrow \mathcal{X}_{S_1}$ is an isomorphism over $\textup{Spec}(\mathcal{O}_{K,S_1})$, seen as an open subscheme of $\textup{Spec}(\mathcal{O}_K)$. Since $f$ is an isomorphism, we can repeat the argument with the second projection $p_2 \colon \mathcal{Z} \rightarrow \mathcal{X}$. Thus, we obtain a finite set $S_2 \subset \textup{Max}(\mathcal{O}_K)$ such that the base change $p_{2,S_2}$ is an isomorphism. We choose $S \coloneqq S_1 \cup S_2$ and define $g \colon \mathcal{X}_{\Sigma,S} \rightarrow \mathcal{X}_{S}$ to be the isomorphism given by the composition $p_{2,S} \circ (p_{1,S})^{-1}$. We claim that the isomorphism $g$ is toric. For this, we follow a similar argument to the one in~Proposition~1.3.14~of~\cite{CLS}. Indeed, the base change of $g$ to the generic fiber is $f$, which restricts to an isomorphism of tori over $K$. Moreover, we have an embedding of the torus $\mathcal{U}_{S}$ into $\mathcal{X}_{\Sigma,S}$. By definition of a toric scheme, this torus acts on itself by translations. Then, we get a commutative diagram
		\begin{center}
			\begin{tikzcd}
				\mathcal{U}_{S} \times_{\mathcal{O}_{K,S}} \mathcal{U}_{S} \arrow[d, "{\textup{id} \times g}"] \arrow[r, "{\mu_{\Sigma}}"] & \mathcal{U}_{S}  \arrow[d, "{g}"] \\
				\mathcal{U}_{S} \times_{\mathcal{O}_{K,S}} \mathcal{U}_{S}  \arrow[r, "{\mu}"] & \mathcal{U}_{S}
			\end{tikzcd}
		\end{center}
		where $\mu_{\Sigma}$ is the torus action on $\mathcal{X}_{\Sigma}$ and we have identified $\mathcal{U}_{S}$ as an open subscheme of $\mathcal{X}_{S}$. The above diagram shows that the morphism $g$ is equivariant on a schematically dense open subset of $\mathcal{X}_{\Sigma,S}$, and therefore, it is equivariant everywhere. Observe that, after fixing $f$, this construction is completely canonical.
	\end{proof}
	\begin{lem}\label{fpqc-descent-toric-schemes}
		Let $\mathcal{X}, \Sigma$ and $S$ be as in Lemma~\ref{generic-canonical}. For each $\mathfrak{p} \in S$, let $\widetilde{\Sigma}_{\mathfrak{p}}$ be the unique complete fan in $N_{\mathbb{R}} \times \mathbb{R}_{\geq 0}$ such that the localization $\mathcal{X}_{\mathfrak{p}}$ of $\mathcal{X}$ at $\mathfrak{p}$ is isomorphic to $\mathcal{X}_{\widetilde{\Sigma}_{\mathfrak{p}}}/ \mathcal{O}_{K,\mathfrak{p}}$ (See~Description~\ref{canonical-fan}). Then, the collection of fans consisting of $\Sigma$ and the $\widetilde{\Sigma}_{\mathfrak{p}}$'s uniquely determine the toric scheme $\mathcal{X}/\mathcal{O}_K$ (up to toric isomorphism).
	\end{lem}
	\begin{proof}
		We want to use fpqc descent to show that the collection of toric schemes consisting of $\mathcal{X}_{\Sigma,S}$ and the $\mathcal{X}_{\widetilde{\Sigma}_{\mathfrak{p}}}$'s uniquely determines the toric scheme $\mathcal{X}/\mathcal{O}_K$. First, we consider the scheme 
		\begin{displaymath}
			\mathcal{S}^{\prime} \coloneqq \mathcal{S}_S \amalg \coprod_{\mathfrak{p} \in S} \mathcal{S}_{\mathfrak{p}}.
		\end{displaymath}
		Then, the canonical morphism $\mathcal{S}^{\prime} \rightarrow \mathcal{S}$ is faithfully flat and quasi-compact. Observe that the scheme $\mathcal{X}^{\prime} \rightarrow \mathcal{S}^{\prime}$ obtained by base change is given by the disjoint union
		\begin{displaymath}
			\mathcal{X}^{\prime} = \mathcal{X}_S \amalg \coprod_{\mathfrak{p} \in S} \mathcal{X}_{\mathfrak{p}}.
		\end{displaymath}
		Now, we check that $\mathcal{S}^{\prime \prime} \coloneqq \mathcal{S}^{\prime} \times_{\mathcal{S}} \mathcal{S}^{\prime}$ is given by a disjoint union of the affine schemes:
		\begin{itemize}
			\item The product $\mathcal{S}_S \times_{\mathcal{S}} \mathcal{S}_S \cong \textup{Spec}(\mathcal{O}_{K,S} \otimes_{\mathcal{O}_K} \mathcal{O}_{K,S}) \cong \textup{Spec}(\mathcal{O}_{K,S}) = \mathcal{S}_S$.
			\item The product $\mathcal{S}_{S} \times_{\mathcal{S}} \mathcal{S}_{\mathfrak{p}} \cong \textup{Spec}(\mathcal{O}_{K,S} \otimes_{\mathcal{O}_K} \mathcal{O}_{K,\mathfrak{p}}) \cong \textup{Spec}(K) \cong \mathcal{S}_{\mathfrak{p}} \times_{\mathcal{S}} \mathcal{S}_{S}$.
			\item The product  $\mathcal{S}_{\mathfrak{p}} \times_{\mathcal{S}} \mathcal{S}_{\mathfrak{q}} \cong \textup{Spec}(\mathcal{O}_{K,\mathfrak{p}} \otimes_{\mathcal{O}_K} \mathcal{O}_{K,\mathfrak{q}})$. This fiber product is isomorphic to $\mathcal{S}_K = \textup{Spec}(K)$ if the primes $\mathfrak{p}$ and $\mathfrak{q}$ are different, or to $\mathcal{S}_{\mathfrak{p}}$ if they are the same.
		\end{itemize}
		It is easy to see that (the non-trivial part of) a descent datum on $\mathcal{X}^{\prime}$ consists of isomorphisms between the generic fibers of the schemes $ \mathcal{X}_S$ and $\mathcal{X}_{\mathfrak{p}}$, and the generic fibers of $\mathcal{X}_{\mathfrak{p}}$ and $\mathcal{X}_{\mathfrak{q}}$, satisfying the obvious cocycle condition (See~Definition~14.67~of~\cite{GW20}). By construction, the scheme $\mathcal{X}^{\prime}/\mathcal{S}^{\prime}$ descends to the scheme $\mathcal{X}/\mathcal{S}$. We may also consider the scheme $\mathcal{Y}^{\prime} \rightarrow \mathcal{S}^{\prime}$ given by
		\begin{displaymath}
			\mathcal{Y}^{\prime} = \mathcal{X}_{\Sigma,S} \amalg \coprod_{\mathfrak{p} \in S} \mathcal{X}_{\widetilde{\Sigma},\mathfrak{p}}.
		\end{displaymath}
		Then, we get an $\mathcal{S}^{\prime}$-isomorphism $\mathcal{Y}^{\prime} \rightarrow \mathcal{X}^{\prime}$. Observe that the intersection of $\widetilde{\Sigma}$ with the hyperplane $N_{\mathbb{R}} \times \lbrace 0 \rbrace$ is exactly $\Sigma$. This identifies the toric variety $X_{\Sigma}/K$ with the generic fibers of $\mathcal{X}_{\widetilde{\Sigma}_{\mathfrak{p}}}/ \mathcal{O}_{K,\mathfrak{p}}$ and $\mathcal{X}_{\Sigma,S}$. This identification gives a descent datum on $\mathcal{Y}^{\prime}$, which is isomorphic to the canonical descent datum on $\mathcal{X}^{\prime}$. By Theorem~14.70~of~\cite{GW20}, the descent functor is fully faithful. Therefore, the scheme $\mathcal{Y}^{\prime}/\mathcal{S}^{\prime}$ descends to $\mathcal{X}/\mathcal{S}$, and it is uniquely determined by the collection of fans consisting of $\Sigma$ and the $\widetilde{\Sigma}_{\mathfrak{p}}$'s. 
		
		We now show the compatibility of their toric structures. There is a torus action
		\begin{displaymath}
			\mu^{\prime}_{\Sigma}  \colon \mathcal{U}^{\prime} \times_{\mathcal{S}^{\prime}} \mathcal{Y}^{\prime} \longrightarrow \mathcal{Y}^{\prime}
		\end{displaymath}
		obtained by the disjoint union of the torus actions on $\mathcal{X}_{\Sigma,S}$ and $\mathcal{X}_{\widetilde{\Sigma}_{\mathfrak{p}}}$, where $\mathcal{U}^{\prime} \coloneqq \mathcal{U} \times_{\mathcal{S}} \mathcal{S}^{\prime}$. Then we need to show that this torus action descends to the torus action $\mu \colon \mathcal{U} \times_{\mathcal{S}} \mathcal{X} \rightarrow \mathcal{X}$. By Exercise~14.17~of~\cite{GW20} applied to the schematically dense open subset $\mathcal{U}_S \times_{\mathcal{S}} \mathcal{X}_{\Sigma,S}$ of $\mathcal{U} \times_{\mathcal{S}} \mathcal{X}$, we see that $\mu^{\prime}_{\Sigma}$ descends to a torus action $\mu_{\Sigma}$ on $\mathcal{X}$. The argument at the end of the proof of Lemma~\ref{generic-canonical} shows that $\mu_{\Sigma}$ and $\mu$ coincide.
	\end{proof}
	Our next task is to determine which sets $S$ and collections of fans $\Sigma$ and $\widetilde{\Sigma}_{\mathfrak{p}}$ give a toric scheme $\mathcal{X}/\mathcal{O}_K$ as above. In other words, we want the fpqc descent procedure we just described to be \textit{effective}. By Remark~14.71~\cite{GW20}, fpqc descent of schemes is not effective in general, and one must ask for extra conditions. For instance, Theorem~6.7~of~\cite{BLR90}  shows that descent of schemes equipped with ample line bundles is effective. We want to apply this result to our situation. By Description~\ref{toric-divisors-K-DVR} and Proposition~\ref{ample-OK}, we must restrict to collections consisting of a smooth projective fan $\Sigma$ and projective fans $\widetilde{\Sigma}_{\mathfrak{p}}$ whose restriction to $N_{\mathbb{R}} \times \lbrace 0 \rbrace$ is $\Sigma$. This motivates the following definition.
	\begin{dfn}\label{arithmetic-fan}
		An \textit{arithmetic fan} in $(N_{\mathbb{R}}, K)$ is an element $\overline{\Sigma} \coloneqq (\widetilde{\Sigma}_{\mathfrak{p}} )$ of the product
		\begin{displaymath}
			\prod_{\mathfrak{p} \in \textup{Max}(\mathcal{O}_K)} \textup{PF}(N_{\mathbb{R}} \times \mathbb{R}_{\geq 0}),
		\end{displaymath}
		where $\textup{PF}(N_{\mathbb{R}} \times \mathbb{R}_{\geq 0})$ is the set of projective fans in $N_{\mathbb{R}} \times \mathbb{R}_{\geq 0}$, satisfying the following conditions:
		\begin{enumerate}
			\item There exist a smooth projective fan $\Sigma$ in $N_{\mathbb{R}}$ such that, for each $\mathfrak{p} \in \textup{Max}(\mathcal{O}_K)$, the fan obtained by intersection of $\widetilde{\Sigma}_{\mathfrak{p}}$ with the hyperplane $N_{\mathbb{R}} \times \lbrace 0 \rbrace$ is the fan $\Sigma$.
			\item There exists a finite set $S(\overline{\Sigma}) \subset \textup{Max}(\mathcal{O}_K)$ such that, for all $\mathfrak{p} \not\in S(\overline{\Sigma})$, the fan $\widetilde{\Sigma}_{\mathfrak{p}}$ is the canonical extension $\Sigma_{\textup{can}}$ of $\Sigma$.
		\end{enumerate}
		An arithmetic fan $\overline{\Sigma}$ is \textit{effective} if there exist strictly convex rational support functions $\Psi$ on $\Sigma$ and $\Phi_{\mathfrak{p}}$ on $\widetilde{\Sigma}_{\mathfrak{p}}$ for each $\mathfrak{p} \in S(\overline{\Sigma})$ satisfying $\Phi_{\mathfrak{p}}|_{N_{\mathbb{R}} \times \lbrace 0 \rbrace} = \Psi$.
	\end{dfn}
	\begin{dfn}
		A \textit{projective toric arithmetic variety over} $\mathcal{O}_K$ is a projective toric scheme $\mathcal{X}/\mathcal{S}$ which is also an arithmetic variety. This means that it is a variety, as in Subsection~\ref{notation}, and that its generic fiber $X/K$ is smooth. If the context is clear, we often omit the adjective "projective".
	\end{dfn}
	\begin{thm}\label{effective-fpqc-des-tor}
		Let $\mathcal{X}/\mathcal{O}_K$ be a projective toric arithmetic variety. Then, there exists an arithmetic fan $\overline{\Sigma}$ in $(N_{\mathbb{R}},K)$ such that:
		\begin{enumerate}
			\item The fiber $\mathcal{X}_{S}/\mathcal{O}_{K,S}$ is isomorphic to the toric scheme $\mathcal{X}_{\Sigma,S}$, where $S=S(\overline{\Sigma})$ and $\Sigma$ are as in Definition~\ref{arithmetic-fan}. In particular, there is an isomorphism of generic fibers $X \cong X_{\Sigma}$.
			\item For each $\mathfrak{p} \in S$, the localization $\mathcal{X}_{\mathfrak{p}}/\mathcal{O}_{K,\mathfrak{p}}$ is isomorphic to the toric scheme $\mathcal{X}_{\widetilde{\Sigma}_{\mathfrak{p}}}$.
		\end{enumerate}
		The toric scheme $\mathcal{X}$ is uniquely determined by $\overline{\Sigma}$ (up to toric isomorphism), and we write $\mathcal{X}=\mathcal{X}_{\overline{\Sigma}}$.
		Conversely, for each effective arithmetic fan $\overline{\Sigma}$, there is a projective toric arithmetic variety $\mathcal{X}_{\overline{\Sigma}}$.
	\end{thm}
	\begin{proof}
		By Lemma~\ref{fpqc-descent-toric-schemes}, there exist a finite set $S \subset \textup{Max}(\mathcal{O}_K)$, a complete fan $\Sigma$ in $N_{\mathbb{R}}$, and a complete fan $\widetilde{\Sigma}_{\mathfrak{p}}$ in $N_{\mathbb{R}} \times \mathbb{R}_{\geq 0}$ for each $\mathfrak{p} \in S$, such that properties \textit{(i)} and \textit{(ii)} hold. By definition, the generic fiber $X \cong X_{\Sigma}$ of $\mathcal{X}$ is smooth and projective. Hence, the fan $\Sigma$ is smooth and projective. This immediately implies that its canonical extension $\Sigma_{\textup{can}}$ is smooth and projective. Now, let~$\mathfrak{p} \in S$. The localization $\mathcal{X}_{\mathfrak{p}}$ of $\mathcal{X}$ is projective. Therefore, the fan $\widetilde{\Sigma}_{\mathfrak{p}}$ is projective. The isomorphism on generic fibers means that the fan obtained by intersecting $\widetilde{\Sigma}_{\mathfrak{p}}$ with the hyperplane $N_{\mathbb{R}} \times \lbrace 0 \rbrace$ is exactly $\Sigma$. Declaring $\widetilde{\Sigma}_{\mathfrak{q}} = \Sigma_{\textup{can}}$ for every maximal ideal $\mathfrak{q} \not \in S$, we conclude that $\overline{\Sigma} \coloneqq (\widetilde{\Sigma}_{\mathfrak{p}})$ is an arithmetic fan in $(N_{\mathbb{R}},K)$. Again, Lemma~\ref{fpqc-descent-toric-schemes} implies that $\mathcal{X}$ is uniquely determined by $\overline{\Sigma}$.
		
		Conversely, suppose that the arithmetic fan $\overline{\Sigma}$ is effective. Then, multiplying the functions $\Psi$ and $\Phi_{\mathfrak{p}}$ by a large enough positive integer $m$, we may assume that they attain integer values at the ray generators of their respective fans. In particular, these functions correspond to ample toric divisors $D$ and $\mathcal{D}_{\mathfrak{p}}$ on $X_{\Sigma}$ and $\mathcal{X}_{\widetilde{\Sigma}_{\mathfrak{p}}}$, respectively. Now, by Corollary~\ref{div-fun-can-OK} and Proposition~\ref{ample-OK}, the ample toric (integral) divisor induces an ample toric integral divisor $\mathcal{D}$ on the canonical model $\mathcal{X}_{\Sigma}/\mathcal{O}_K$. Then, its restriction $\mathcal{D}_S$ to $\mathcal{X}_{\Sigma,S}$ is ample. Consider the ample line bundles
		\begin{displaymath}
			\mathcal{L}_{S} \coloneqq \mathcal{O}_{\mathcal{X}_{\Sigma,S}}(\mathcal{D}_S), \quad \mathcal{L}_{\mathfrak{p}} \coloneqq \mathcal{O}_{\mathcal{X}_{\widetilde{\Sigma}_{\mathfrak{p}}}}(\mathcal{D}_{\mathfrak{p}}), \quad \mathcal{L}_K \coloneqq \mathcal{O}_{X_{\Sigma}} (D).
		\end{displaymath}
		The disjoint union of $\mathcal{L}_S$ and the $\mathcal{L}_{\mathfrak{p}}$'s gives an ample line bundle $\mathcal{L}^{\prime}$ on the scheme $\mathcal{Y}^{\prime} \rightarrow \mathcal{S}^{\prime}$ defined in the proof of Lemma~\ref{fpqc-descent-toric-schemes}. Now, we argue as in the last part of the proof of Theorem~4.2~of~\cite{Kue98}: Observe that the canonical identification of $\mathcal{L}_K$ as the restriction of the line bundles $\mathcal{L}_S$ and $\mathcal{L}_{\mathfrak{p}}$ to the generic fiber $X_{\Sigma}$ defines a descent datum for $\mathcal{S}^{\prime}$ over $\mathcal{S}= \textup{Spec}(\mathcal{O}_K)$. By Theorem~6.7~of~\cite{BLR90}, this descent is effective. Therefore, we obtain a scheme $\mathcal{X}_{\overline{\Sigma}}/\mathcal{O}_K$, which is a projective toric arithmetic variety by virtue of Lemma~\ref{fpqc-descent-toric-schemes}.
	\end{proof}
	\begin{rem}\label{maybe-equivalent}
		Observe that the ``effective'' property of arithmetic fans is a crucial technical condition in the above proof. It translates into a descent datum on the ample line bundle $\mathcal{L}^{\prime}$. We wonder if this is a ``true condition'', or if we can show that every arithmetic fan $\overline{\Sigma}$ is effective. If the latter is true, Theorem~\ref{effective-fpqc-des-tor} would be a combinatorial description of projective toric arithmetic varieties. We present a strategy for proving this. By definition of projective fan, for each $\mathfrak{p} \in S = S(\overline{\Sigma})$ there exist a strictly convex support function $\Phi_{\mathfrak{p}}$ on $\widetilde{\Sigma}_{\mathfrak{p}}$. In particular, the restriction $\Psi_{\mathfrak{p}}$ of $\Phi_{\mathfrak{p}}$ to the hyperplane $N_{\mathbb{R}}\times \lbrace 0 \rbrace$ is a strictly convex rational support function on $\Sigma$. The issue here is that the functions $\Phi_{\mathfrak{p}}$ may differ. Now, fix a prime $\mathfrak{p} \in S$ and suppose that for each $\mathfrak{q} \in S$ distinct from $\mathfrak{p}$, we can find a convex support function $\Upsilon_{(\mathfrak{p},\mathfrak{q})}$ on $\widetilde{\Sigma}_{\mathfrak{p}}$ extending $\Psi_{\mathfrak{p}}$. Geometrically, this means that the ample toric divisor $D_{\mathfrak{q}}$ on $X_{\Sigma}$ determined by $\Psi_{\mathfrak{q}}$ has a (relatively) nef model $\mathcal{D}_{(\mathfrak{p},\mathfrak{q})}$ on the toric scheme $\mathcal{X}_{\widetilde{\Sigma}_{\mathfrak{p}}}$. Then, we may define the function
		\begin{displaymath}
			\Phi_{\mathfrak{p}}^{\prime} \coloneqq \Phi_{\mathfrak{p}} + \sum_{\mathfrak{q} \neq \mathfrak{p}} \Upsilon_{(\mathfrak{p},\mathfrak{q})}.
		\end{displaymath}
		It is rational and strictly convex on $\widetilde{\Sigma}_{\mathfrak{p}}$, and its restriction to $N_{\mathbb{R}} \times \lbrace 0 \rbrace$ is $\Psi = |S| \cdot \Psi_{\mathfrak{p}}$. Therefore, the collection of support functions $\Phi_{\mathfrak{p}}^{\prime}$ make the arithmetic fan $\overline{\Sigma}$ is effective. In particular, this strategy provides a method for producing examples of effective arithmetic fans. For instance, if for each $\mathfrak{p}$ the fan $\widetilde{\Sigma}_{\mathfrak{p}}$ refines the canonical fan $\Sigma_{\textup{can}}$, then $\Upsilon_{(\mathfrak{p},\mathfrak{q})}(x,t) \coloneqq \Psi_{\mathfrak{p}}(x)$ satisfies this condition. There are simple ways to extend $\Psi_{\mathfrak{p}}$ to a support function on some complete fan in $N_{\mathbb{R}} \times \mathbb{R}_{\geq 0}$ (see p.58~of~\cite{BPS}), but these methods do not give much control on the fan.
	\end{rem}
	\begin{ex}\label{canonical-mod-div-1}
		Given a smooth projective fan $\Sigma$ in $N_{\mathbb{R}}$, there is a canonical effective arithmetic fan $\Sigma_{\textup{can}}$ induced by $\Sigma$. Indeed, for each $\mathfrak{p} \in \textup{Max}(\mathcal{O}_K)$, we declare $\widetilde{\Sigma}_{\mathfrak{p}} = \Sigma_{\textup{can}}$. By the previous remark, this is effective. We abuse notation by identifying $\Sigma$ and the canonical arithmetic fan induced by it. Observe that its associated projective toric arithmetic variety is the canonical model $\mathcal{X}_{\Sigma}/\mathcal{O}_K$. If $D$ is a toric divisor on $X_{\Sigma}$, we denote by $\mathcal{D}_{\textup{can}}$ the horizontal toric divisor determined by the support function $\Psi_D$ (Corollary~\ref{div-fun-can-OK}). We call it the \textit{canonical model} of $D$ on $\mathcal{X}_{\Sigma}$.
	\end{ex}
	
	\subsection{Proper toric morphisms} Many properties are stable under faithfully flat descent. This means that if $\mathcal{S}^{\prime} \rightarrow \mathcal{S}$ is faithfully flat, the morphism $\mathcal{Y}^{\prime} \rightarrow \mathcal{X}^{\prime}$ satisfies the property $P$, and it descends to $\mathcal{Y} \rightarrow \mathcal{X}$, then $\mathcal{Y} \rightarrow \mathcal{X}$ has the property $P$. For instance, properness is stable under faithfully flat descent (Proposition~14.51~of~\cite{GW20}). As an application, we characterize proper toric morphisms between toric projective arithmetic varieties. We introduce some notation and then state the result.
	\begin{dfn}\label{morphism-arith-fan-def}
		Let $\overline{\Sigma}_1$ and $\overline{\Sigma}_2$ be arithmetic fans in $((N_1)_{\mathbb{R}},K)$ and $((N_2)_{\mathbb{R}},K)$, respectively. Given an element
		\begin{displaymath}
			\overline{\phi} \coloneqq (\phi_{\mathfrak{p}}) \in \prod_{\mathfrak{p} \in \textup{Max}(\mathcal{O}_K)} \textup{Hom}_{\mathbb{Z}} (N_{1} \times \mathbb{Z} , N_{2} \times \mathbb{Z})
		\end{displaymath}
		we denote by $\overline{\phi}_{\mathbb{R}} = (\phi_{\mathfrak{p},\mathbb{R}})$ be the $\mathfrak{p}$-tuple of $\mathbb{R}$-linear maps obtained by extension of scalars. We say that $\overline{\phi}$\textit{ is a morphism between the arithmetic fans} $\overline{\Sigma}_1$\textit{ and} $\overline{\Sigma}_2$, denoted by $\overline{\phi} \colon \overline{\Sigma}_1 \rightarrow \overline{\Sigma}_2$, if it satisfies the following conditions:
		\begin{enumerate}
			\item For each $\mathfrak{p}$, the map $\phi_{\mathfrak{p},\mathbb{R}}$ is compatible with the fans $\widetilde{\Sigma}_{1,\mathfrak{p}}$ and $\widetilde{\Sigma}_{2,\mathfrak{p}}$. That is, for each $\sigma_1 \in \widetilde{\Sigma}_{1,\mathfrak{p}}$ there is $\sigma_2 \in \widetilde{\Sigma}_{2,\mathfrak{p}}$ such that $\phi_{\mathfrak{p},\mathbb{R}}(\sigma_1) \subset \sigma_2$. Moreover, $(\phi_{\mathfrak{p},\mathbb{R}})^{-1}(|\widetilde{\Sigma}_{2,\mathfrak{p}}|) = |\widetilde{\Sigma}_{1,\mathfrak{p}}|$.
			\item There is a $\mathbb{Z}$-linear map $\phi \colon N_1 \rightarrow N_2$ such that $\phi_{\mathbb{R}}$ is compatible with the fans $\Sigma_1$ and $\Sigma_2$, and satisfies $(\phi_{\mathbb{R}})^{-1}(|\Sigma_2|) = |\Sigma_1|$.
			\item For each maximal ideal $\mathfrak{p}$, we have $(\phi_{\mathfrak{p},\mathbb{R}})|_{(N_1)_{\mathbb{R}} \times \lbrace 0 \rbrace} = \phi_{\mathbb{R}}$.
			\item There is a finite set $S \subset \textup{Max}(\mathcal{O}_K)$ containing $S(\overline{\Sigma}_1) \cup S(\overline{\Sigma}_2)$ such that $\phi_{\mathfrak{p},\mathbb{R}}$ is the \textit{canonical extension} of $\phi_{\mathbb{R}}$ whenever $\mathfrak{p} \not \in S$. That means, $\phi_{\mathfrak{p},\mathbb{R}} (0,1) = (0,1)$. 
		\end{enumerate}
		If additionally, all of the maps $\phi_{\mathfrak{p}}$ are identity, we say that $\overline{\Sigma}_1$\textit{ refines }$\overline{\Sigma}_2$ and write $\overline{\Sigma}_1 \geq \overline{\Sigma}_2$.
	\end{dfn}
	\begin{prop}\label{proper-toric-morphisms}
		Let $\mathcal{X}_{\overline{\Sigma}_1}$ and $\mathcal{X}_{\overline{\Sigma}_2}$ be projective toric arithmetic varieties given by the arithmetic fans $\overline{\Sigma}_1$ and $\overline{\Sigma}_2$. For each morphism of arithmetic fans $\overline{\phi} \colon \overline{\Sigma}_1 \rightarrow \overline{\Sigma}_2$, there is an induced proper toric morphism $f \colon \mathcal{X}_{\overline{\Sigma}_1} \rightarrow \mathcal{X}_{\overline{\Sigma}_2}$. Conversely, a proper toric morphism $f \colon \mathcal{X}_{\overline{\Sigma}_1} \rightarrow \mathcal{X}_{\overline{\Sigma}_2}$ induces a morphism of arithmetic fans, and the assignment $\overline{\phi} \mapsto f$ is bijective.
	\end{prop}
	\begin{proof}
		Let $\overline{\phi} \colon \overline{\Sigma}_1 \rightarrow \overline{\Sigma}_2$ be a morphism of arithmetic fans. By the description of the category of proper toric varieties with proper toric morphisms (See Proposition~3.2.2~of~\cite{Per26}), there is a proper toric morphism $f_K \colon X_{\Sigma_1} \rightarrow X_{\Sigma_2}$ induced by $\phi$. Similarly, by Proposition~3.3.2~of~\cite{Per26}, for each maximal ideal $\mathfrak{p}$ there is a proper toric morphism $f_{\mathfrak{p}} \colon \mathcal{X}_{\widetilde{\Sigma}_{1,\mathfrak{p}}} \rightarrow \mathcal{X}_{\widetilde{\Sigma}_{2,\mathfrak{p}}}$ induced by the compatible map $\phi_{\mathfrak{p},\mathbb{R}}$. By Definition~\ref{morphism-arith-fan-def}.(iii), we get an identification of $f_K$ as the base change to the generic fiber of the toric morphism $f_{\mathfrak{p}}$. 
		
		Now, let $S$ be as in Definition~\ref{morphism-arith-fan-def}.(iv). Arguing as in Theorem~3.3.4~of~\cite{CLS}, the map $\phi$ induces a toric morphism $g \colon \mathcal{X}_{\Sigma_1} \rightarrow \mathcal{X}_{\Sigma_2}$ between the canonical models. Note that the toric scheme $\mathcal{X}_{\overline{\Sigma}_i ,S}$ identifies with the toric scheme $\mathcal{X}_{\Sigma_i, S}$. Then, define $f_{S} \colon \mathcal{X}_{\overline{\Sigma}_1 , S} \rightarrow \mathcal{X}_{\overline{\Sigma}_2, S}$ as the base change of $g$ to $\mathcal{O}_{K,S}$. By Definition~\ref{morphism-arith-fan-def}.(iv), for each $\mathfrak{p} \not \in S$, the base change of $f_{S}$ to $\mathcal{O}_{K,\mathfrak{p}}$ coincides with the morphism $f_{\mathfrak{p}}$ we defined above.
		
		Consider the scheme $\mathcal{S}^{\prime}$ defined in the proof of Lemma~\ref{fpqc-descent-toric-schemes}, and denote by $\mathcal{X}_{\overline{\Sigma}_i}^{\prime} \rightarrow \mathcal{S}^{\prime}$ the scheme obtained by base change with respect to the fpqc morphism $\mathcal{S}^{\prime} \rightarrow \mathcal{S}$. The disjoint union of the toric morphisms $f_S$ and $f_{\mathfrak{p}}$'s, $\mathfrak{p} \in S$, gives a proper $\mathcal{S}^{\prime}$-morphism
		\begin{displaymath}
			f^{\prime} \colon  \mathcal{X}_{\overline{\Sigma}_1}^{\prime} \longrightarrow \mathcal{X}_{\overline{\Sigma}_2}^{\prime}.
		\end{displaymath}
		By Exercise~14.17~of~\cite{GW20} applied to the schematically dense open subset $\mathcal{X}_{\overline{\Sigma}_1,S}$, the morphism $f^{\prime} \colon  \mathcal{X}_{\overline{\Sigma}_1}^{\prime} \rightarrow \mathcal{X}_{\overline{\Sigma}_2}^{\prime}$ descends to a morphism $f \colon \mathcal{X}_{\overline{\Sigma}_1} \rightarrow \mathcal{X}_{\overline{\Sigma}_2}$. Moreover, for $i=1,2$, the split torus $\mathcal{U}_{i,S} / \mathcal{O}_{K,S}$ is an open set of $\mathcal{X}_{\overline{\Sigma}_i,S} \cong \mathcal{X}_{\Sigma_i, S}$, and therefore it is open in $\mathcal{X}_{\overline{\Sigma}_i}$. By the construction of $f_S$, it restricts to a group-scheme morphism $f_S \colon  \mathcal{U}_{1,S} \rightarrow \mathcal{U}_{2,S}$. Arguing as in the end of the proof of Lemma~\ref{generic-canonical}, the morphism $f$ is shown to be toric. The injectivity of the assignment $\overline{\phi} \mapsto f$ is obvious from the construction.
		
		Now, let $f \colon \mathcal{X}_{\overline{\Sigma}_1} \rightarrow \mathcal{X}_{\overline{\Sigma}_2}$ be a proper toric morphism. For each maximal ideal $\mathfrak{p}$, we consider the localization $f_{\mathfrak{p}}$ of $f$ at $\mathcal{O}_{K,\mathfrak{p}}$. By Proposition~3.3.2~of~\cite{Per26}, we obtain the $\mathbb{Z}$-linear map $\phi_{\mathfrak{p}}$ satisfying condition~(i) of Definition~\ref{morphism-arith-fan-def}. Similarly, by Proposition~3.2.2~of~the~loc.~cit., the $\mathbb{Z}$-linear map $\phi$ is induced by the base change $f_K$ of $f$ to the generic fiber. Since $f_K$ is proper, the map $\phi$ satisfies the condition~(ii). The property~(iii) follows from the identification of $f_K$ as the base change of $f_{\mathfrak{p}}$ to the generic fiber. It remains to verify condition~(iv). Let $S^{\prime} = S(\overline{\Sigma}_1) \cup S(\overline{\Sigma}_2)$. Observe that the base change $\mathcal{X}_{\overline{\Sigma}_i , S^{\prime}}$ identifies with the base change of the canonical model $\mathcal{X}_{\Sigma_i, S^{\prime}}$. Then, we identify the base change $f_{S^{\prime}}$ of $f$ to $\mathcal{O}_{K,S^{\prime}}$ with the toric morphism
		\begin{displaymath}
			f_{S^{\prime}} \colon \mathcal{X}_{\Sigma_1, S^{\prime}} \longrightarrow \mathcal{X}_{\Sigma_2,  S^{\prime}} 
		\end{displaymath}
		By definition, $f_{S^{\prime}}$ is equivariant, and it restricts to a group-scheme morphism $f_{K} \colon U_1 \rightarrow U_2$ on the tori embedded on the respective generic fibers. On the other hand, by the above construction, the map $\phi$ determined by $f_K$ induces a toric morphism
		\begin{displaymath}
			g_{S^{\prime}} \colon \mathcal{X}_{\Sigma_1, S^{\prime}} \longrightarrow \mathcal{X}_{\Sigma_2,  S^{\prime}}
		\end{displaymath}
		satisfying $f_K = g_K$. In particular, Proposition~10.52~of~\cite{GW20} shows that $f_{S^{\prime}}$ and $g_{S^{\prime}}$ agree on an open neighborhood of the generic fiber $X_{\Sigma_1}$. Without loss of generality (Lemma~\ref{generic-canonical}), we may assume that this neighbourhood is of the form $\mathcal{X}_{\Sigma_1, S}$, where $S$ is a finite set of maximal ideals containing $S^{\prime}$. By the discussion above, this set satisfies the condition~(iv) in Definition~\ref{morphism-arith-fan-def}. This shows that the assignment $\overline{\phi} \mapsto f$ is surjective, and therefore, a bijection.
	\end{proof}
	\begin{rem}
		It is obvious that the assignment $\overline{\phi} \mapsto f$ in the above result is functorial. Then, Theorem~\ref{effective-fpqc-des-tor} and Proposition~\ref{proper-toric-morphisms} identify the category of projective toric arithmetic varieties with proper toric morphisms as a full subcategory of the category of arithmetic fans. If one could show that every arithmetic fan is effective (See Remark~\ref{maybe-equivalent}), then these categories would be equivalent.
	\end{rem}
	\begin{lem}\label{refinement}
		Let $\overline{\Sigma} = (\widetilde{\Sigma}_{\mathfrak{p}})$ be an arithmetic fan in $(N_{\mathbb{R}},K)$ refining the canonical arithmetic fan induced by the fan $\Sigma$ from Definition~\ref{arithmetic-fan}. Then, $\overline{\Sigma}$ is effective and there is a proper toric morphism $f \colon \mathcal{X}_{\overline{\Sigma}} \rightarrow \mathcal{X}_{\Sigma}$. Moreover, the morphism $f$ is given by a finite number of successive blow-ups along torus-invariant ideal sheaves supported on the fibers over the primes $\mathfrak{p} \in S = S(\overline{\Sigma})$.
	\end{lem}
	\begin{proof}
		The fact that $\overline{\Sigma}$ is effective was shown in Remark~\ref{maybe-equivalent}. Since $\overline{\Sigma}$ refines the canonical fan induces by $\Sigma$, there is an induced toric proper morphism $f \colon \mathcal{X}_{\overline{\Sigma}} \rightarrow \mathcal{X}_{\Sigma}$ constructed in Proposition~\ref{proper-toric-morphisms}. We only need to show that $f$ is given by a finite number of successive blow-ups. We show the lemma for the case of a singleton $S = \lbrace \mathfrak{p} \rbrace$. The general case can be obtained by iterating the same procedure one prime at a time. By definition, there exists a strictly convex rational support function $\Phi_{\mathfrak{p}}$ on $\widetilde{\Sigma}_{\mathfrak{p}}$. Multiplying by a large enough positive integer if necessary, we may assume that $\Phi_{\mathfrak{p}}$ attains integer values at the ray generators of the rays in $\widetilde{\Sigma}_{\mathfrak{p}}$. By the proof of Proposition~\ref{proper-toric-morphisms}, the refinement $\widetilde{\Sigma}_{\mathfrak{p}} \geq \Sigma_{\textup{can}}$ induces a proper birational toric $\mathcal{O}_{K,\mathfrak{p}}$-morphism
		\begin{displaymath}
			f_{\mathfrak{p}} \colon \mathcal{X}_{\widetilde{\Sigma}_{\mathfrak{p}}} \longrightarrow \mathcal{X}_{\Sigma_{\textup{can}}},
		\end{displaymath}
		given by localization of $f$ at $\mathcal{O}_{K,\mathfrak{p}}$. The description in~§~IV.3.(h)--(j)~of~\cite{KKMS} implies that the strictly convex piecewise linear function $\Phi_{\mathfrak{p}}$ corresponds to a torus-invariant ideal sheaf $\mathcal{I}$ on $\mathcal{X}_{\Sigma_{\textup{can}}}$, such that the map $f_{\mathfrak{p}}$ is given by the blow-up of $\mathcal{X}_{\Sigma_{\textup{can}}}$ along $\mathcal{I}$. Since the map $f_{\mathfrak{p}}$ is an isomorphism on generic fibers, the torus-invariant closed subscheme $Z$ given by $\mathcal{I}$ must be contained in the special fiber $\mathcal{X}_{\Sigma, k(\mathfrak{p})}$ of $\mathcal{X}_{\Sigma_{\textup{can}}}$, seen as a closed subset of $\mathcal{X}_{\Sigma}$. It follows that $Z$ is a closed toric subscheme of $\mathcal{X}_{\Sigma}$. Define the proper birational $\mathcal{O}_K$-morphism 
		\begin{displaymath}
			g \colon \mathcal{Y} \longrightarrow \mathcal{X}_{\Sigma}
		\end{displaymath}
		to be the blow-up of $\mathcal{X}_{\Sigma}$ with center $Z$. Observe that $g \colon \mathcal{Y} \setminus g^{-1}(Z) \rightarrow \mathcal{X}_{\Sigma} \setminus Z$ is an isomorphism. In particular, the restriction of $g$ to $\mathcal{Y} \setminus \mathcal{Y}_{k(\mathfrak{p})}$ agrees with $f$. Since isomorphisms are stable under faithfully flat descent, to see that $f$ and $g$ agree everywhere, it is enough to show that the base change of $g$ to $\mathcal{O}_{K,\mathfrak{p}}$ is exactly $f_{\mathfrak{p}}$. The flat morphism $\mathcal{O}_K \rightarrow \mathcal{O}_{K,\mathfrak{p}}$ induces a flat morphism $\mathcal{X}_{\Sigma_{\textup{can}}} \rightarrow \mathcal{X}_{\Sigma}$. Then, we may apply Lemma~31.32.3~in~\href{https://stacks.math.columbia.edu/tag/01OF}{Section~01OF}~of~\cite{Stacks}, which shows that blow-ups commute with flat morphisms. Concretely, we have the following pullback diagram
		\begin{center}
			\begin{tikzcd}
				\mathcal{X}_{\widetilde{\Sigma}_{\mathfrak{p}}} \arrow[r, "{f_{\mathfrak{p}}}"] \arrow[d] & \mathcal{X}_{\Sigma_{\textup{can}}}  \arrow[d]\\
				\mathcal{Y} \arrow[r, "{g}"] & \mathcal{X}_{\Sigma}
			\end{tikzcd}
		\end{center}
		We can fit the above diagram into the commutative diagram
		\begin{center}
			\begin{tikzcd}
				\mathcal{X}_{\widetilde{\Sigma}_{\mathfrak{p}}} \arrow[r, "{f_{\mathfrak{p}}}"] \arrow[d] & \mathcal{X}_{\Sigma_{\textup{can}}}  \arrow[d] \arrow[r] & \textup{Spec}(\mathcal{O}_{K,\mathfrak{p}}) \arrow[d] \\
				\mathcal{Y} \arrow[r, "{g}"] & \mathcal{X}_{\Sigma} \arrow[r] & \textup{Spec}(\mathcal{O}_{K})
			\end{tikzcd}
		\end{center}
		where the left and right inner squares are pullback diagrams. Therefore, the outer diagram is a pullback diagram. The result follows.
	\end{proof}
	Now, we describe the regularity of a projective toric arithmetic variety $\mathcal{X}_{\overline{\Sigma}} / \mathcal{O}_K$ in terms of its arithmetic fan. First, we extend the notion of smoothness to arithmetic fans. Then, we give the corresponding result.
	\begin{dfn}
		Let $\overline{\Sigma} = (\widetilde{\Sigma}_{\mathfrak{p}})$ be an arithmetic fan in $(N_{\mathbb{R}}, K)$. We say that $\overline{\Sigma}$ is \textit{smooth} if every fan $\widetilde{\Sigma}_{\mathfrak{p}}$ is smooth.
	\end{dfn}
	\begin{lem}\label{regular}
		Let $\mathcal{X}_{\overline{\Sigma}} / \mathcal{O}_K$ be a projective toric arithmetic variety given by the arithmetic fan $\overline{\Sigma}$. If the arithmetic fan $\overline{\Sigma}$ is smooth, then the scheme $\mathcal{X}_{\overline{\Sigma}}$ is regular.
	\end{lem}
	\begin{proof}
		Let $S= S(\overline{\Sigma})$ be the set from Definition~\ref{arithmetic-fan}. Then, let $\mathcal{S}^{\prime} \rightarrow \mathcal{S}$ be the faithfully flat morphism from the proof of Lemma~\ref{fpqc-descent-toric-schemes}. The base change $\mathcal{X}_{\overline{\Sigma}}^{\prime} \rightarrow \mathcal{X}_{\overline{\Sigma}}$ is faithfully flat. If the arithmetic fan is smooth, it implies that the toric schemes $\mathcal{X}_{\Sigma,S}/\mathcal{O}_{K,S}$ and $\mathcal{X}_{\widetilde{\Sigma}_{\mathfrak{p}}}/\mathcal{O}_{K,\mathfrak{p}}$ are regular. Therefore, $\mathcal{X}_{\overline{\Sigma}}^{\prime}$ is regular. By Proposition~14.57~of~\cite{GW20}, $\mathcal{X}_{\overline{\Sigma}}$ is regular.
	\end{proof}
	We finish this subsection by describing the category of toric projective models for a fairly general class of quasi-projective toric schemes. We will need the following technical lemma on refinements of arithmetic fans.
	\begin{lem}\label{common-refinement}
		Let $\overline{\Sigma}_1 = (\widetilde{\Sigma}_{1,\mathfrak{p}})$ and $\overline{\Sigma}_2 = (\widetilde{\Sigma}_{2,\mathfrak{p}})$ be arithmetic fans in $(N_{\mathbb{R}},K)$. Then, there exists a smooth and effective arithmetic fan $\overline{\Sigma}$ which refines both $\overline{\Sigma}_1$ and $\overline{\Sigma}_2$.
	\end{lem}
	\begin{proof}
		Let $\Sigma_i$ and $S_i = S(\overline{\Sigma}_i)$ be as in Definition~\ref{arithmetic-fan}, and write $S= S_1 \cup S_2$. Consider the fan $\Sigma_1 \cdot \Sigma_2$ obtained by intersecting the cones in $\Sigma_1$ with the cones in $\Sigma_2$. By Theorem~11.1.9~of~\cite{CLS}, the fan $\Sigma_1 \cdot \Sigma_2$ admits a smooth projective refinement $\Sigma$. By construction, the fan $\Sigma$ refines both $\Sigma_1$ and $\Sigma_2$. Then, for each $\mathfrak{p} \not \in S$, define $\widetilde{\Sigma}_{\mathfrak{p}} \coloneqq \Sigma_{\textup{can}}$. Observe that $\widetilde{\Sigma}_{1,\mathfrak{p}} = \Sigma_{1,\textup{can}}$ and $\widetilde{\Sigma}_{2,\mathfrak{p}} = \Sigma_{2,\textup{can}}$, therefore, the fan $\widetilde{\Sigma}_{\mathfrak{p}}$ refines both of them. Now, for each $\mathfrak{p} \in S$, we consider the fan $\Sigma_{\textup{can}} \cdot \widetilde{\Sigma}_{1, \mathfrak{p}} \cdot \widetilde{\Sigma}_{2, \mathfrak{p}}$ obtained by intersections. Since $\Sigma$ refines both $\Sigma_1$ and $\Sigma_2$, the fan obtained by intersecting $\Sigma_{\textup{can}} \cdot \widetilde{\Sigma}_{1, \mathfrak{p}} \cdot \widetilde{\Sigma}_{2, \mathfrak{p}}$ with the hyperplane $N_{\mathbb{R}} \times \lbrace 0 \rbrace$ is exactly $\Sigma$. By Theorem~11.1.9~of~\cite{CLS}, the fan $\Sigma_{\textup{can}} \cdot \widetilde{\Sigma}_{1, \mathfrak{p}} \cdot \widetilde{\Sigma}_{2, \mathfrak{p}}$ admits a smooth projective refinement $\widetilde{\Sigma}_{\mathfrak{p}}$ which leaves fixed the smooth cones. In particular, the intersection of $\widetilde{\Sigma}_{\mathfrak{p}}$ with the hyperplane $N_{\mathbb{R}} \times \lbrace 0 \rbrace$ is exactly $\Sigma$. Therefore, $\overline{\Sigma} = (\widetilde{\Sigma}_{\mathfrak{p}})$ is a smooth arithmetic fan in $(N_{\mathbb{R}},K)$ which refines the $\overline{\Sigma}_i$'s and the canonical arithmetic fan induced by $\Sigma$. We conclude that $\overline{\Sigma}$ is effective.
	\end{proof}
	\begin{desc}\label{toric-models-OK}
		Let $\Sigma_0$ be a smooth quasi-projective fan in $N_{\mathbb{R}}$ and $S$ be a finite subset of $\textup{Max}(\mathcal{O}_K)$. By Corollary~\ref{quasi-proj-OK}, the toric scheme $\mathcal{X}_{\Sigma_0, S}/\mathcal{O}_K$ is regular and quasi-projective. Conversely, by Theorem~\ref{effective-fpqc-des-tor}, any torus-invariant open subset $\mathcal{Y}$ of a projective toric arithmetic variety $\mathcal{X}_{\overline{\Sigma}}$ contains an open, schematically dense subscheme of the form $\mathcal{X}_{\Sigma_0, S}$. Then:
		\begin{enumerate}
			\item By Theorem~\ref{effective-fpqc-des-tor}, a projective toric arithmetic variety $\mathcal{X}_{\overline{\Sigma}}$ is a toric projective model of $\mathcal{X}_{\Sigma_{0},S}$ if and only if, for each $\mathfrak{p} \not \in S$, the canonical fan $\Sigma_{0, \textup{can}}$ is a subfan of $\widetilde{\Sigma}_{\mathfrak{p}}$.
			\item A morphism $f \colon \mathcal{X}_{\overline{\Sigma}_1} \rightarrow \mathcal{X}_{\overline{\Sigma}_2}$ of toric projective models of $\mathcal{X}_{\Sigma_{0},S}$ restricts to the identity on the torus $\mathcal{U}_{S}$ contained in  $\mathcal{X}_{\Sigma_{0},S}$. Then, Proposition~\ref{proper-toric-morphisms} implies that $f$ corresponds to a refinement of arithmetic fans $\overline{\Sigma}_1 \geq \overline{\Sigma}_2$.
			\item By Lemma~\ref{common-refinement}, the category $\textup{PM}_{\mathbb{T}}(\mathcal{X}_{\Sigma_{0},S} /\mathcal{O}_K)$ of toric projective models of $\mathcal{X}_{\Sigma_{0},S}$ over $\mathcal{O}_K$ is cofiltered.
			\item The assignment $\overline{\Sigma} \mapsto \mathcal{X}_{\overline{\Sigma}}$ identifies the poset consisting of smooth effective arithmetic fans $\overline{\Sigma} = (\widetilde{\Sigma}_{\mathfrak{p}})$ in $(N_{\mathbb{R}},K)$ satisfying the Condition~(i) and ordered by refinement, as a full and cofinal subcategory of $\textup{PM}_{\mathbb{T}}(\mathcal{X}_{\Sigma_{0},S} /\mathcal{O}_K)$.
		\end{enumerate}
	\end{desc}
	
	\subsection{Toric divisors}\label{3-4}
	The final task in this section is to describe the group of toric $\mathbb{Q}$-divisors on a projective toric arithmetic variety. As before, we will use the theory of fpqc descent to obtain a characterization in terms of support functions or piecewise affine functions.
	\begin{desc}\label{toric-global-sf}
		Let $\mathcal{X}_{\overline{\Sigma}}/\mathcal{O}_K$ be a projective toric arithmetic variety given by the arithmetic fan $\overline{\Sigma} = (\widetilde{\Sigma}_{\mathfrak{p}})$ in $(N_{\mathbb{R}}, K)$. Given a toric Cartier $\mathbb{Q}$-divisor $\mathcal{D}$ on $\mathcal{X}_{\overline{\Sigma}}$, denote by $\mathcal{D}_{\mathfrak{p}}$ the toric Cartier divisor on $\mathcal{X}_{\widetilde{\Sigma}_{\mathfrak{p}}}/\mathcal{O}_{K,\mathfrak{p}}$ obtained by restriction. Similarly, we write $D$ for its restriction to the generic fiber $X_{\Sigma}/K$. By Description~\ref{toric-divisors-K-DVR}, we get
		\begin{displaymath}
			\Psi_D \in \mathcal{SF}(\Sigma, \mathbb{Q}), \quad (\Phi_{\mathcal{D}_{\mathfrak{p}}}) \in  \prod_{\mathfrak{p} \in \textup{Max}(\mathcal{O}_K)} \mathcal{SF}(\widetilde{\Sigma}_{\mathfrak{p}}, \mathbb{Q})
		\end{displaymath}
		where $\Psi_D$ and $\Phi_{\mathcal{D}_{\mathfrak{p}}}$ are the support functions corresponding to the toric divisors $D$ and $\mathcal{D}_{\mathfrak{p}}$. Conversely, a $\mathfrak{p}$-tuple of rational support functions
		\begin{displaymath}
			(\Phi_{\mathfrak{p}}) \in  \prod_{\mathfrak{p} \in \textup{Max}(\mathcal{O}_K)} \mathcal{SF}(\widetilde{\Sigma}_{\mathfrak{p}}, \mathbb{Q}),
		\end{displaymath}
		it determines a toric divisor $\mathcal{D}$ on $\mathcal{X}_{\overline{\Sigma}}$ if the following compatibility conditions are satisfied:
		\begin{enumerate}
			\item There exists a rational support function $\Psi \in  \mathcal{SF}(\Sigma, \mathbb{Q})$ such that 	$\Phi_{\mathfrak{p}}|_{N_{\mathbb{R}}\times \lbrace 0 \rbrace} = \Psi$ for every $\mathfrak{p} \in \textup{Max}(\mathcal{O}_K)$.
			\item There is a finite subset $S \subset \textup{Max}(\mathcal{O}_K)$ containing the set $S(\overline{\Sigma})$ (Definition~\ref{arithmetic-fan}) such that $\Phi_{\mathfrak{p}} (x,t) = \Psi(x)$ for all $(x,t)\in N_{\mathbb{R}} \times \mathbb{R}_{\geq 0}$ and every $\mathfrak{p} \not \in S$.
		\end{enumerate}
		Indeed, for each $\mathfrak{p} \in \textup{Max}(\mathcal{O}_K)$, we have a toric divisor $\mathcal{D}_{\mathfrak{p}}$ on the localization $\mathcal{X}_{\overline{\Sigma},\mathfrak{p}}/\mathcal{O}_{K,\mathfrak{p}}$ whose support function is $\Phi_{\mathfrak{p}}$. By condition~(i), there exists a toric divisor $D$ on the generic fiber $X_{\Sigma}/K$ whose support function is $\Psi$ and such that, for each $\mathfrak{p}$, the restriction of  $\mathcal{D}_{\mathfrak{p}}$ to the generic fiber is $D$. Then, there exists a horizontal toric divisor $\mathcal{D}$ on the canonical model $\mathcal{X}_{\Sigma}/\mathcal{O}_K$ corresponding to the support function $\Psi$. Thus, its restriction to the toric scheme $\mathcal{X}_{\Sigma,S} \cong \mathcal{X}_{\overline{\Sigma},S}$ over $\mathcal{O}_{K,S}$ determines a toric divisor $\mathcal{D}_S$ which, by condition~(ii), induces the toric divisors $\mathcal{D}_{\mathfrak{p}}$ for each $\mathfrak{p} \not \in S$. Arguing as in the proof of Theorem~\ref{effective-fpqc-des-tor}, this gives a descent datum for the collection of toric divisors $\mathcal{D}_S$ and $\mathcal{D}_{\mathfrak{p}}$. The claim follows from effective fpqc descent of coherent sheaves (Theorem~14.66~of~\cite{GW20}).
	\end{desc}
	We can give an alternative description of the group of toric $\mathbb{Q}$-divisors in terms of piecewise affine functions. 
	\begin{desc}\label{toric-global-paf}
		Let $\mathcal{X}_{\overline{\Sigma}}/\mathcal{O}_K$ be a projective toric arithmetic variety given by the arithmetic fan $\overline{\Sigma} = (\widetilde{\Sigma}_{\mathfrak{p}})$ in $(N_{\mathbb{R}}, K)$. Given a toric Cartier $\mathbb{Q}$-divisor $\mathcal{D}$ on $\mathcal{X}_{\overline{\Sigma}}$, denote by $\mathcal{D}_{\mathfrak{p}}$ the toric Cartier divisor on $\mathcal{X}_{\widetilde{\Sigma}_{\mathfrak{p}}}/\mathcal{O}_{K,\mathfrak{p}}$ obtained by restriction. Similarly, we write $D$ for its restriction to the generic fiber $X_{\Sigma}/K$. By Description~\ref{toric-divisors-K-DVR}, we get
		\begin{displaymath}
			\Psi_D \in \mathcal{SF}(\Sigma, \mathbb{Q}), \quad (\gamma_{\mathcal{D}_{\mathfrak{p}}}) \in  \prod_{\mathfrak{p} \in \textup{Max}(\mathcal{O}_K)} \mathcal{PA}(\Pi_{\mathfrak{p}}, \mathbb{Q})
		\end{displaymath}
		where $\Psi_D$ and $\gamma_{\mathcal{D}_{\mathfrak{p}}}$ are the support function and piecewise affine functions corresponding to the toric divisors $D$ and $\mathcal{D}_{\mathfrak{p}}$, and $\Pi_{\mathfrak{p}}$ is the polyhedral complex obtained by intersecting the cones in the fan $\widetilde{\Sigma}_{\mathfrak{p}}$ with the hyperplane $N_{\mathbb{R}} \times \lbrace 1 \rbrace$. Conversely, a $\mathfrak{p}$-tuple of rational piecewise affine functions
		\begin{displaymath}
			(\gamma_{\mathfrak{p}}) \in  \prod_{\mathfrak{p} \in \textup{Max}(\mathcal{O}_K)} \mathcal{PA}(\Pi_{\mathfrak{p}}, \mathbb{Q}),
		\end{displaymath}
		it determines a toric divisor $\mathcal{D}$ on $\mathcal{X}_{\overline{\Sigma}}$ if the following compatibility conditions are satisfied:
		\begin{enumerate}
			\item There exists a rational support function $\Psi \in  \mathcal{SF}(\Sigma, \mathbb{Q})$ such that $\textup{rec}(\gamma_{\mathfrak{p}}) = \Psi$ for every $\mathfrak{p} \in \textup{Max}(\mathcal{O}_K)$. Here, the \textit{recession} of $f$ is the function $\textup{rec}(f) \colon N_{\mathbb{R}} \rightarrow \mathbb{R}$ is given by
			\begin{displaymath}
				\textup{rec}(f) (t) \coloneqq	\lim_{\lambda \rightarrow + \infty} f(\lambda \cdot t) / \lambda.
			\end{displaymath}
			\item There is a finite subset $S \subset \textup{Max}(\mathcal{O}_K)$ containing the set $S(\overline{\Sigma})$ (Definition~\ref{arithmetic-fan}) such that $\gamma_{\mathfrak{p}} = \Psi$ for every $\mathfrak{p} \not \in S$.
		\end{enumerate}
		This follows immediately from Description~\ref{toric-global-paf}.
	\end{desc}
	As an application of Lemmas~\ref{refinement}~and~\ref{regular}, we can compute the group of toric Weil divisors on certain regular projective toric varieties.
	\begin{prop}\label{toric-div-OK}
		Let $\overline{\Sigma}$ be smooth arithmetic fan in $(N_{\mathbb{R}}, K)$ refining the canonical fan induced by $\Sigma$. Then, the isomorphism between Cartier and Weil divisors on the regular projective toric arithmetic variety $\mathcal{X}_{\overline{\Sigma}}/\mathcal{O}_K$ induces an isomorphism
		\begin{displaymath}
			\textup{Div}_{\mathbb{T}}(\mathcal{X}_{\overline{\Sigma}}) \cong \bigoplus_{\rho \in \Sigma(1)} \mathbb{Z} \, \oplus \bigoplus_{\mathfrak{p} \in \textup{Max}(\mathcal{O}_K)} \mathbb{Z}^{r(\mathfrak{p})}.
		\end{displaymath}
		where $\Sigma(1)$ and $\widetilde{\Sigma}_{\mathfrak{p}}(1)$ are the set of rays of the corresponding fan, and $r(\mathfrak{p}) \coloneqq |\widetilde{\Sigma}_{\mathfrak{p}} (1) \setminus  \Sigma(1)|$ is the number of vertical rays in $\widetilde{\Sigma}_{\mathfrak{p}}$. In particular, each toric divisor $\mathcal{D}$ uniquely decomposes as $\mathcal{D} = \mathcal{D}_{\textup{hor}} + \mathcal{D}_{\textup{vert}}$, where $\mathcal{D}_{\textup{hor}}$ and $\mathcal{D}_{\textup{vert}}$ are horizontal and vertical toric divisors, respectively.
	\end{prop}
	\begin{proof}
		Follows from Proposition~\ref{toric-div-canonical-OK} combined with Lemmas~\ref{refinement}~and~\ref{regular}. The main difference is that the irreducible components of the fibers $\mathcal{X}_{\overline{\Sigma},k(\mathfrak{p})}$ over $\mathfrak{p}$ are in bijection with the set of vertical rays $\widetilde{\Sigma}_{\mathfrak{p}}(1) \setminus  \Sigma(1)$ of $\widetilde{\Sigma}_{\mathfrak{p}}$. Then, the exceptional divisors given by the finite number of blow-ups decompose into the toric Weil divisors corresponding to the vertical rays which are not generated by $(0,1) \in N_{\mathbb{R}} \times \mathbb{R}_{\geq 0}$.
	\end{proof}
	A nice feature of these descriptions is that properties of a toric divisor can be read from the associated $\mathfrak{p}$-tuple of functions. We finish the section with the following two examples.
	\begin{cor}\label{div-effective}
		Let $\mathcal{X}_{\overline{\Sigma}}/\mathcal{O}_K$ be a projective toric arithmetic variety given by the arithmetic fan $\overline{\Sigma} = (\widetilde{\Sigma}_{\mathfrak{p}})$ in $(N_{\mathbb{R}}, K)$. Given a toric Cartier $\mathbb{Q}$-divisor $\mathcal{D}$ on $\mathcal{X}_{\overline{\Sigma}}$, let $(\Phi_{\mathcal{D}_{\mathfrak{p}}})$ and $(\gamma_{\mathcal{D}_{\mathfrak{p}}})$ be the associated $\mathfrak{p}$-tuples of support functions and piecewise affine functions, respectively. Then, the following conditions are equivalent:
		\begin{enumerate}
			\item The toric divisor $\mathcal{D}$ is effective.
			\item For each $\mathfrak{p}$, we have $\Phi_{\mathcal{D}_{\mathfrak{p}}} (x,t) \leq 0$ for every $(x,t) \in N_{\mathbb{R}}\times \mathbb{R}_{\geq 0}$.
			\item For each $\mathfrak{p}$, we have $\gamma_{\mathcal{D}_{\mathfrak{p}}}(x) \leq 0$ for every $x \in N_{\mathbb{R}}$.
		\end{enumerate}
	\end{cor}
	\begin{proof}
		Follows from Lemma~2.3.6.~of~\cite{Y-Z} applied to the Descriptions~\ref{toric-global-sf}~and~\ref{toric-global-paf}. 
	\end{proof}
	\begin{cor}\label{div-nef-ample}
		Let $\mathcal{X}_{\overline{\Sigma}}/\mathcal{O}_K$ be a projective toric arithmetic variety given by the arithmetic fan $\overline{\Sigma} = (\widetilde{\Sigma}_{\mathfrak{p}})$ in $(N_{\mathbb{R}}, K)$. Given a toric Cartier $\mathbb{Q}$-divisor $\mathcal{D}$ on $\mathcal{X}_{\overline{\Sigma}}$, let $(\Phi_{\mathcal{D}_{\mathfrak{p}}})$ and $(\gamma_{\mathcal{D}_{\mathfrak{p}}})$ be the associated $\mathfrak{p}$-tuples of support functions and piecewise affine functions, respectively. Then, the following conditions are equivalent:
		\begin{enumerate}
			\item The toric divisor $\mathcal{D}$ is relatively nef (resp. ample).
			\item For each $\mathfrak{p}$, the function $\Phi_{\mathcal{D}_{\mathfrak{p}}}$ is concave (resp. strictly concave)
			\item For each $\mathfrak{p}$, we have $\gamma_{\mathcal{D}_{\mathfrak{p}}}$ is concave (resp. strictly concave).
		\end{enumerate}
	\end{cor}
	\begin{proof}
		Relative nefness follows from the local version, given in~Theorem~3.7.1~of~\cite{BPS}. For ampleness, we use the same local result together with Lemma~\ref{ample-OK} to produce an fpqc descent datum of ample line bundles. Then, we apply Proposition~14.56~of~\cite{GW20}.
	\end{proof}
	
	\section{A toric analog of Yuan-Zhang's adelic divisors}\label{4}
	This section introduces the group of toric adelic divisors over a toric arithmetic variety. This is a toric version of the constructions of Yuan and Zhang~\cite{Y-Z}, summarized in Section~\ref{global-arith-theory}. First, we summarize the work of Burgos, Philippon, and Sombra in~\cite{BPS} and their subsequent joint paper with Moriwaki~\cite{BMPS}, and adapt it to the language of toric arithmetic varieties. Then, we develop the formalism of toric adelic divisors. This will be used in Sections~\ref{5}~and~\ref{6}, where we state and prove our main results.
	
	\subsection{The Arakelov geometry of toric arithmetic varieties} \label{4-1}
	Fix a number field $K$ with ring of integers $\mathcal{O}_K$ and set of complex embeddings $\Lambda$. We use the notation introduced at the beginning of Subsections~\ref{notation},~\ref{2-1},~and~\ref{3-2} freely. In particular, we fix a positive integer $d$ and let $\mathcal{U}/\mathcal{O}_K$ be a split $d$-dimensional torus over $\mathcal{O}_K$. We start by recalling the notion of $\mathbb{S}$-invariance. Given a projective toric arithmetic variety $\mathcal{X}_{\overline{\Sigma}}/\mathcal{O}_K$ with torus action $\mu \colon \mathcal{U} \times_{\mathcal{O}_K} \mathcal{X}_{\overline{\Sigma}} \rightarrow \mathcal{X}_{\overline{\Sigma}}$, the analytification functor induces a complex Lie group action $\mu_{\Lambda}^{\textup{an}}$ of $\mathcal{U}_{\Lambda}(\mathbb{C})$ on $\mathcal{X}_{\overline{\Sigma},\Lambda}(\mathbb{C})$, where $ \mathcal{U}_{\Lambda}(\mathbb{C})$ and $\mathcal{X}_{\overline{\Sigma},\Lambda}(\mathbb{C})$ are the complex manifolds given by the complex points of $\mathcal{U}$ and $\mathcal{X}_{\overline{\Sigma}}$. This action decomposes as a disjoint union of actions
	\begin{displaymath}
		\mu_{\lambda}^{\textup{an}} \colon \mathcal{U}_{\lambda}(\mathbb{C}) \times \mathcal{X}_{\overline{\Sigma},\lambda}(\mathbb{C}) \longrightarrow  \mathcal{X}_{\overline{\Sigma},\lambda}(\mathbb{C}), \quad \lambda \in \Lambda.
	\end{displaymath}
	Observe that each $\mathcal{U}_{\lambda}(\mathbb{C})$ is isomorphic to the complex Lie group $(\mathbb{C}^{\times})^{d}$, and it contains a maximal compact torus $\mathbb{S}_{\lambda} \cong (\mathbb{S}^{1})^{d}$, where $\mathbb{S}^{1}$ is the unit circle in $\mathbb{C}^{\times}$. By restriction, we get an action
	\begin{displaymath}
		\mu_{\lambda}^{\textup{an}} \colon \mathbb{S}_{\lambda} \times \mathcal{X}_{\overline{\Sigma},\lambda}(\mathbb{C}) \longrightarrow  \mathcal{X}_{\overline{\Sigma},\lambda}(\mathbb{C}), \quad \lambda \in \Lambda.
	\end{displaymath}
	All together, we get an action of $\mathbb{S}_{\Lambda} \coloneqq \amalg_{\lambda} \, \mathbb{S}_{\lambda}$ on $\mathcal{X}_{\overline{\Sigma},\Lambda}(\mathbb{C})$. Then, a function $f \colon \mathcal{X}_{\overline{\Sigma},\Lambda}(\mathbb{C}) \rightarrow \mathbb{R}_{\pm \infty}$ is said to be $\mathbb{S}$\textit{-invariant} if it is invariant under the action of $\mathbb{S}_{\Lambda}$. We can now recall the definition of a toric divisor on a projective toric arithmetic variety.
	\begin{dfn}
		An arithmetic divisor $\overline{\mathcal{D}}$ on $\mathcal{X}_{\overline{\Sigma}}/\mathcal{O}_K$ is said to be \textit{toric} if its divisorial part $\mathcal{D}$ is toric and its Green's function $g_{\mathcal{D},\Lambda}$ is $\mathbb{S}$-invariant. We denote by $\overline{\textup{Div}}_{\mathbb{T}}(\mathcal{X}_{\overline{\Sigma}})_{\mathbb{Q}}$ the group of toric arithmetic divisors of continuous type on $\mathcal{X}_{\overline{\Sigma}}/\mathcal{O}_K$.
	\end{dfn}
	\begin{rem}
		By definition, a Green's function $g_{\mathcal{D},\Lambda}$ is invariant under complex conjugation. Since a complex embedding $\lambda \colon K \rightarrow \mathbb{C}$ and its complex conjugate $\overline{\lambda}$ induce the same infinite place $v \in \mathfrak{M}(K)_{\infty}$, the Green's function $g_{\mathcal{D},\Lambda}$ canonically identifies with a function
		\begin{displaymath}
			g_{\mathcal{D},\infty} \colon X_{\Sigma, \infty}^{\textup{an}} \longrightarrow \mathbb{R}_{\pm \infty}, \quad   X_{\Sigma, \infty}^{\textup{an}} \coloneqq \coprod_{v \in \mathfrak{M}(K)_{\infty}}  X_{\Sigma,v}^{\textup{an}},
		\end{displaymath}
		where $X_{\Sigma}$ is the generic fiber of $\mathcal{X}_{\overline{\Sigma}}$, $K_v$ is the completion of $K$ at the place $v$, and $X_{\Sigma,v}^{\textup{an}}$ is the Berkovich analytification of the base change $X_{\Sigma,v} \coloneqq X_{\Sigma} \times_{K} \textup{Spec}(K_v)$. We use these functions interchangeably, and we simplify notation by writing $g_{\mathcal{D}}$ for $g_{\mathcal{D},\infty}$ or $g_{\mathcal{D},\Lambda}$. We also identify the function $g_{\mathcal{D},v}$ obtained by restriction to $X_{\Sigma,v}^{\textup{an}}$ with the pair of functions $g_{\mathcal{D},\lambda}$ and $g_{\mathcal{D},\overline{\lambda}}$. This will be clear from the context and should not cause any confusion.
	\end{rem}
	We can describe the $\mathbb{S}$-invariant Green's functions in a similar way as in Description~\ref{toric-global-paf}. To do this, we will use the tropicalization functor.
	\begin{cons}[Tropicalization functor]\label{trop-functor}
		Let $X_{\Sigma}/K$ be a projective toric variety and fix a place $v \in \mathfrak{M}(K)$. We sketch the definition of the \textit{tropicalization functor} and list some of its properties. The motivation behind this construction is that $\mathbb{S}$-invariant objects on $X_{\Sigma,v}^{\textup{an}}$ descend to an object on the \textit{tropical toric variety} $N_{\Sigma}$ associated to $\Sigma$, introduced below. Elements of $N_{\mathbb{R}}$ are linear functionals on $M_{\mathbb{R}}$. Then, there is a \textit{tropicalization map} $\textup{Trop}_v \colon U_{v}^{\textup{an}} \rightarrow N_{\mathbb{R}}$, given by
		\begin{displaymath}
			p \mapsto (m \mapsto - \log |\chi^{m} (p)| ),
		\end{displaymath}
		where $\chi^{m} \in K\left[ M \right]$ and $|\chi^{m} (p)|$ denotes the evaluation of the seminorm $p$ at $\chi^m$. By the cone-orbit correspondence, we may write $X_{\Sigma}$ as a (disjoint) union of torus orbits $\lbrace O(\sigma)\rbrace_{\sigma \in \Sigma}$. Then, we obtain a family $\lbrace \textup{Trop}_v \colon O(\sigma)_{v}^{\textup{an}} \rightarrow N(\sigma)_{\mathbb{R}}\rbrace_{\sigma \in \Sigma}$, which glues into a map $\textup{Trop}_v \colon X_{\Sigma,v}^{\textup{an}} \rightarrow N_{\Sigma}$. Here, the set $N_{\Sigma}$ is given by the disjoint union of the $N(\sigma)_{\mathbb{R}}$'s. For the definition of its topology, we refer to Section~4.1~p.118~of~\cite{BPS}. The map $\textup{Trop}_v$ is continuous, (topologically) proper, and compatible with the stratification of $X_{\Sigma}$ by torus orbits. In particular, $N_{\Sigma}$ is compact. The fibers of $\textup{Trop}_v$ are exactly the orbits of the compact torus action of $\mathbb{S}_v$ on $X_{\Sigma,v}^{\textup{an}}$. The assignment $\Sigma \mapsto N_{\Sigma}$ is functorial: Given a toric morphism $f \colon X_{\Sigma_1} \rightarrow X_{\Sigma_2}$, we let $\phi \colon (N_1)_{\mathbb{R}} \rightarrow (N_2)_{\mathbb{R}}$ be the corresponding linear map compatible with $\Sigma_1$ and $\Sigma_2$. Then, there is a continuous extension $\phi \colon N_{1,\Sigma_1} \rightarrow N_{2,\Sigma_2}$ of $\phi$ such that $\textup{Trop}_v \circ f^{\textup{an}} = \phi \circ \textup{Trop}_v$. We refer to Sections~4.1~and~4.2~of~\cite{BPS} for details on the tropicalization map, and we follow the notation introduced in Subsection~4.1~of~\cite{Per26}.
	\end{cons}
	Now, consider a toric arithmetic divisor $\overline{\mathcal{D}}$ on the projective toric variety $\mathcal{X}_{\overline{\Sigma}}/\mathcal{O}_K$, and fix an infinite place $v \in \mathfrak{M}(K)_{\infty}$. Denote by $D$ the restriction of the divisorial part $\mathcal{D}$ to the generic fiber $X_{\Sigma}$. The support $|D|$ of $D$ does not intersect the torus $U/K$ embedded in $X_{\Sigma}$. Then, the function $g_{\mathcal{D},v}$ restricts to a continuous function on the analytic torus $U_{v}^{\textup{an}}$. By Construction~\ref{trop-functor}, there exists a unique continuous function $\gamma_{\overline{\mathcal{D}}_v} \colon N_{\mathbb{R}} \rightarrow \mathbb{R}$ satisfying 
	\begin{displaymath}
		g_{\mathcal{D},v} = - \gamma_{\overline{\mathcal{D}}_v} \circ \textup{Trop}_v.
	\end{displaymath}
	We call  $\gamma_{\overline{\mathcal{D}}_v}$\textit{ the tropical Green's function of }$\overline{\mathcal{D}}$ \textit{at the place} $v \in \mathfrak{M}(K)_{\infty}$. By density of the analytic torus $U_{v}^{\textup{an}}$ in $X_{\Sigma,v}^{\textup{an}}$, the function $ \gamma_{\overline{\mathcal{D}}_v}$ uniquely determines the Green's function $g_{\mathcal{D},v}$. Then, we state a weaker version of Proposition~4.3.10~of~\cite{BPS}, which characterizes the functions appearing as the tropical Green's function of a given toric divisor.
	\begin{prop}\label{trop-gf-archim}
		Let $\mathcal{D}$ be a toric divisor on the projective toric variety $\mathcal{X}_{\overline{\Sigma}}/\mathcal{O}_K$, $D$ be its restriction to the generic fiber $X_{\Sigma}$, and $\Psi_D$ be the associated rational support function. Then, the assignment $g_{\mathcal{D}} \mapsto (\gamma_{\overline{\mathcal{D}}_v})_{v \in \mathfrak{M}(K)_{\infty}}$ induces a bijection between:
		\begin{enumerate}
			\item The set of Green's functions $g_{\mathcal{D}}$ of continuous type for $\mathcal{D}$.
			\item The set of $v$-tuples $(\gamma_v)_{v \in \mathfrak{M}(K)_{\infty}}$ of continuous functions $\gamma_v \colon N_{\mathbb{R}} \rightarrow \mathbb{R}$ such that, for each $v \in \mathfrak{M}(K)_{\infty}$, the function $\gamma_v - \Psi_D$ extends continuously to the tropical toric variety $N_{\Sigma}$.
		\end{enumerate}
		In particular, if $\overline{\mathcal{D}}$ is a toric arithmetic divisor, the functions $\gamma_{\overline{\mathcal{D}}_v} - \Psi_D$ are bounded.
	\end{prop}
	\begin{rem}\label{metrized-div}
		The global-local approach from Subsection~\ref{2-2} allows us to treat the characterizations from Description~\ref{toric-global-paf} and Proposition~\ref{trop-gf-archim} in a uniform way, without distinguishing between the finite and infinite places. Recall that each maximal ideal $\mathfrak{p}$ corresponds to a unique finite place $v \in \mathfrak{M}(K)_{\textup{fin}}$. Then, the assignments $(\mathcal{D}, g_{\mathcal{D}}) \mapsto \mathcal{D}_v$ and $\mathcal{D}_v \mapsto \overline{D}_v = (D_v, g_{D,v})$ induce a sequence of maps
		\begin{center}
			\begin{tikzcd}
				\overline{\textup{Div}}(\mathcal{X}_{\overline{\Sigma}})_{\mathbb{Q}} \arrow[r] & \textup{Div}(\mathcal{X}_{\widetilde{\Sigma}_v})_{\mathbb{Q}} \arrow[r, "{\textup{an}}"] & \overline{\textup{Div}}(X_{\Sigma,v})_{\mathbb{Q}}.
			\end{tikzcd}
		\end{center}
		The above sequence restricts to the respective subgroups of toric divisors. Now, for each $v$-adic toric arithmetic divisor $\overline{D}_v   = (D_v, g_{D,v}) \in  \overline{\textup{Div}}(X_{\Sigma,v})_{\mathbb{Q}}$, we can define the tropical Green's function $\gamma_{\overline{D}_v}$ as the unique continuous function satisfying $g_{D,v} = - \gamma_{\overline{D}_v} \circ \textup{Trop}_v$. In particular, if the $v$-adic toric arithmetic divisor $\overline{D}_v$ is the analytification of the toric divisor $\mathcal{D}_v$ on $\mathcal{X}_{\widetilde{\Sigma}_v}$, the tropical Green's function $\gamma_{\overline{D}_v}$ coincides with the piecewise affine function $\gamma_{\mathcal{D}_v}$ from Description~\ref{toric-global-paf}, and $\gamma_{\mathcal{D}_v} - \Psi_D$ extends to $N_{\Sigma}$. This is proven in Proposition~4.5.3~of~\cite{BPS}. It follows that the equivariant arithmetic divisor $(D, ( g_{D,v} ))$ (Definition~\ref{eqv}) obtained by analytification of a toric arithmetic divisor $\overline{\mathcal{D}}$ on $\mathcal{X}_{\overline{\Sigma}}$ is \textit{quasi-algebraic} in the sense of Definitions~1.5.13~of~\cite{BPS}~or~3.2~of~\cite{BMPS}. Therefore, all of the results in~\cite{BPS}~and~\cite{BMPS} apply to our situation.
	\end{rem}
	The previous remark justifies the following definition, which unifies the terminologies from Description~\ref{toric-global-paf} and Proposition~\ref{trop-gf-archim}.
	\begin{dfn}\label{trop-gf-OK-proj}
		Let $\overline{\mathcal{D}}$ be a toric arithmetic divisor on the he projective toric arithmetic variety  $\mathcal{X}_{\overline{\Sigma}}/\mathcal{O}_K$ given by the arithmetic fan $\overline{\Sigma} \in (N_{\mathbb{R}},K)$. The \textit{tropical Green's function associated to }$\overline{\mathcal{D}}$ is the $v$-tuple of continuous functions
		\begin{displaymath}
			\gamma_{\overline{\mathcal{D}}} \coloneqq ( \gamma_{\overline{\mathcal{D}}_v}) \in \prod_{v \in \mathfrak{M}(K)} C^0 (N_{\mathbb{R}}),
		\end{displaymath}
		where $\gamma_{\overline{\mathcal{D}}_v}$ is the associated rational piecewise function if $v \in \mathfrak{M}(K)_{\textup{fin}}$, or the associated tropical Green's function if $v \in \mathfrak{M}(K)_{\infty}$. We call $\gamma_{\overline{\mathcal{D}}_v}$ the $v$\textit{-adic tropical Green's function of }$\overline{\mathcal{D}}$.
	\end{dfn}
	The tropical Green's function $\gamma_{\overline{\mathcal{D}}}$ satisfies the following functorial property.
	\begin{prop}\label{trop-fun-functorial}
		Let $f \colon \mathcal{X}_{\overline{\Sigma}_1} \rightarrow  \mathcal{X}_{\overline{\Sigma}_2}$ be a proper toric morphism between projective arithmetic toric varieties over $\mathcal{O}_{K}$, induced by the morphism of arithmetic fans $\overline{\phi} \colon \overline{\Sigma}_1 \rightarrow \overline{\Sigma}_2$. Given a toric arithmetic divisor $\overline{\mathcal{D}}$ on $\mathcal{X}_{\overline{\Sigma}_2}$, we let $f^{\ast} \overline{\mathcal{D}}$ be its image under the induced pullback morphism $f^\ast \colon \overline{\textup{Div}}_{\mathbb{T}}(\mathcal{X}_{\overline{\Sigma}_2})_{\mathbb{Q}} \rightarrow \overline{\textup{Div}}_{\mathbb{T}}(\mathcal{X}_{\overline{\Sigma}_1})_{\mathbb{Q}}$. Then, the tropical Green's function of  $f^{\ast} \overline{\mathcal{D}}$ is given by
		\begin{displaymath}
			\gamma_{f^\ast \overline{\mathcal{D}}} = ( \gamma_{\overline{\mathcal{D}}_v} \circ \phi_{v,\mathbb{R}}), \quad \phi_{v,\mathbb{R}} \coloneqq \begin{cases}
				\phi_{\mathfrak{p},\mathbb{R}}, & v = v_{\mathfrak{p}} \in \mathfrak{M}(K)_{\textup{fin}} \\
				\phi_{\mathbb{R}}, & v \in \mathfrak{M}(K)_{\infty}.
			\end{cases}
		\end{displaymath}
		In particular, if the morphism $f$ is induced by a refinement $\overline{\Sigma}_1 \geq \overline{\Sigma}_2$, we have $\gamma_{f^\ast \overline{\mathcal{D}}} = \gamma_{\overline{\mathcal{D}}}$. 
	\end{prop}
	\begin{proof}
		It follows from the local case, proven in Propositions~3.3.5~and~4.1.4~of~\cite{Per26} for the finite and infinite places, respectively.
	\end{proof}
	The next task is to state Theorem~4.8.1~of~\cite{BPS}. This is an analytic version of Corollary~\ref{div-nef-ample}, which characterizes Green's functions of psh type for a given relatively nef toric divisor in terms of the associated tropical Green's functions. We assume familiarity with the convex-analytic tools described in Appendix~A~of~\cite{Per26}, and we stick to the notation introduced there. In particular, we denote by $f^{\vee}$ the \textit{Legendre-Fenchel transform} of a concave function $f$. Now, let $\mathcal{D}$ be a relatively nef toric divisor on the projective arithmetic toric variety $\mathcal{X}_{\overline{\Sigma}}/\mathcal{O}_K$ given by the arithmetic fan $\overline{\Sigma} \in (N_{\mathbb{R}},K)$. Denote by $D$ its restriction to the generic fiber $X_{\Sigma}$, which is a nef toric divisor. By Description~\ref{toric-global-paf} and Corollary~\ref{div-nef-ample}, the support function $\Psi_D$ of $D$ is concave, and its stability set $\Delta_D \coloneqq \textup{stab}(\Psi_D)$ is a rational convex polytope in $M_{\mathbb{R}}$ (Theorem~3.2.7~of~\cite{Per26}). Then, we have the following result of \cite{BPS}.
	\begin{thm}\label{toric-psh-type}
		With the above notation, the assignments $g_{\mathcal{D}} \mapsto (\gamma_{\overline{\mathcal{D}}_v})_{v \in \mathfrak{M}(K)_{\infty}}$ and $\gamma_v \mapsto \gamma_{v}^{\vee}$ induce bijections between:
		\begin{enumerate}[label=(\roman*)]
			\item The set of Green's functions $g_{\mathcal{D}}$ of psh type for $\mathcal{D}$.
			\item The set of $v$-tuples $(\gamma_v)_{v \in \mathfrak{M}(K)_{\infty}}$ of concave functions $\gamma_v \colon N_{\mathbb{R}} \rightarrow \mathbb{R}$ such that, for each $v \in \mathfrak{M}(K)_{\infty}$, the function $\gamma_v - \Psi_D$ is bounded.
			\item The set of $v$-tuples $(\vartheta_v)_{v \in \mathfrak{M}(K)_{\infty}}$ of continuous concave functions $\vartheta_v \colon \Delta_D \rightarrow \mathbb{R}$.
		\end{enumerate}
	\end{thm}
	\begin{rem}
		The combination of Corollary~\ref{div-nef-ample} and Theorem~\ref{toric-psh-type} gives a complete description of the cone $\overline{\textup{Div}}_{\mathbb{T}}^{+}(\mathcal{X}_{\overline{\Sigma}})_{\mathbb{Q}}$ of semipositive toric arithmetic divisors on the projective toric arithmetic variety $\mathcal{X}_{\overline{\Sigma}}/\mathcal{O}_K$ in terms of their associated tropical Green's functions.
	\end{rem}
	Let $\overline{\mathcal{D}}$ be a semipositive toric arithmetic divisor on $\mathcal{X}_{\overline{\Sigma}}$ and $\gamma_{\overline{\mathcal{D}}} =( \gamma_{\overline{\mathcal{D}}_v})$ is its associated tropical Green's function. By Description~\ref{toric-global-paf}, for all but a finite number of places $v \in \mathfrak{M}(K)$, the $v$-adic tropical Green's function $\gamma_{\overline{\mathcal{D}}_v}$ coincides with the support function $\Psi_D$ of the restiction of $\mathcal{D}$ to the generic fiber $X_{\Sigma}$. By Remark~A.1.7, the Legendre-Fenchel dual of $\Psi_D$ is the indicator function $\iota_{\Delta_D}$ of the convex polytope $\Delta_D$ associated to $D$. Therefore, for each $y \in \Delta_D$, we have $\gamma_{\overline{\mathcal{D}}_v}^{\vee}(y)=0$ for all but a finite number of places $v$. This justifies the following definition.
	\begin{dfn}\label{global-roof-fun}
		With the above notation, the $v$\textit{-adic roof function} of $\overline{\mathcal{D}}$ is the Legendre-Fenchel transform $\vartheta_{\overline{\mathcal{D}}_v}$ of $\gamma_{\overline{\mathcal{D}}_v}$. Then, the \textit{global roof function} of $\overline{\mathcal{D}}$ is the continuous concave function $\vartheta_{\overline{\mathcal{D}}} \colon \Delta_D \rightarrow \mathbb{R}$ given by the sum
		\begin{displaymath}
			\vartheta_{\overline{\mathcal{D}}} \coloneqq \sum_{v \in \mathfrak{M}(K)} \delta_v \cdot \vartheta_{\overline{\mathcal{D}}_v}, \quad \delta_v \coloneqq [K_v : \mathbb{Q}_v]/[K : \mathbb{Q}].
		\end{displaymath}
		We emphasize that this is a finite sum.
	\end{dfn}
	Now, we state Theorem~6.1~of~\cite{BMPS}, which characterizes (arithmetic) nefness of toric arithmetic divisors in terms of the associated global roof functions.
	\begin{thm}\label{global-nef}
		Let $\overline{\Sigma}$ be an arithmetic fan in $(N_{\mathbb{R}},K)$ and  $\mathcal{X}_{\overline{\Sigma}}/\mathcal{O}_K$ the corresponding arithmetic toric variety. Then, a semipositive toric arithmetic divisor $\overline{\mathcal{D}}$ is nef if and only if its global roof function $\vartheta_{\overline{\mathcal{D}}}$ is non-negative, i.e., $\vartheta_{\overline{\mathcal{D}}}(y) \geq 0$ for all $y \in \Delta_D$.
	\end{thm}
	\begin{ex}\label{canonical-mod-div-2}
		Recall the situation from Example~\ref{canonical-mod-div-1}: Given a smooth projective fan $\Sigma$ in $N_{\mathbb{R}}$ and a toric divisor $D$ on the toric variety $X_{\Sigma}/K$, the canonical model $\mathcal{D}_{\textup{can}}$ of $D$ is the toric divisor on the canonical model $\mathcal{X}_{\Sigma}/\mathcal{O}_K$ of $X_{\Sigma}$ satisfying $\gamma_{\mathcal{D}_{\mathfrak{p}}} = \Psi_D$ for every $\mathfrak{p} \in \textup{Max}(\mathcal{O}_K)$. By Proposition~\ref{trop-gf-archim}, the $v$-tuple $(\Psi_D)_{v \in \mathfrak{M}(K)_{\infty}}$ determines a Green's function $g_{\mathcal{D}_{\textup{can}}}$ for $\mathcal{D}_{\textup{can}}$ of continuous type. Then, the pair the pair $\overline{\mathcal{D}}_{\textup{can}} = (\mathcal{D}_{\textup{can}},g_{\mathcal{D}_{\textup{can}}})$ is called the \textit{canonical toric arithmetic divisor induced by }$D$. By the above discussion, $D$ is nef if and only if $\overline{\mathcal{D}}_{\textup{can}}$ is arithmetically nef. Moreover, its global roof function and all of its $v$-adic roof functions are equal to the indicator function $\iota_{\Delta_D}$ of its associated convex polytope $\Delta_D$.
	\end{ex}
	Now, we recall Proposition~5.2.4~of~\cite{BPS}, which relates the arithmetic intersection numbers $\overline{\mathcal{D}}_0 \cdot \ldots \cdot \overline{\mathcal{D}}_d$ in Theorem~\ref{global-intersection} with the local toric heights $\textup{h}^{\textup{tor}}(X_{\Sigma,v} ; \overline{D}_{0,v}, \ldots , \overline{D}_{d,v} )$ of~Definition~4.1.11~of~\cite{Per26}. We emphasize that the local toric height was already introduced in~\cite{BPS}, but it was stated in terms of metrized line bundles, instead of arithmetic divisors.
	\begin{prop}\label{sum-of-heights}
		Let $\overline{\Sigma}$ be an arithmetic fan in $(N_{\mathbb{R}},K)$, $\mathcal{X}_{\overline{\Sigma}}/\mathcal{O}_K$ be the corresponding arithmetic toric variety of relative dimension $d$, and $X_{\Sigma}/K$ its generic fiber. For each $0 \leq i \leq d$, let $\overline{\mathcal{D}}_i$ be a semipositive toric arithmetic divisor on $\mathcal{X}_{\overline{\Sigma}}$ and denote by $\overline{D}_{i,v}$ the corresponding toric arithmetic divisor on $X_{\Sigma,v}$ (Remark~\ref{metrized-div}). Then, the following identity holds
		\begin{displaymath}
			\overline{\mathcal{D}}_0 \cdot \ldots \cdot \overline{\mathcal{D}}_d = \sum_{v \in \mathfrak{M}(K)} \delta_{v} \cdot \textup{h}^{\textup{tor}}(X_{\Sigma,v} ; \overline{D}_{0,v}, \ldots , \overline{D}_{d,v} ),
		\end{displaymath}
		where the $\delta_v$ are the constants in Definition~\ref{global-roof-fun}. This extends by linearity to integrable toric arithmetic divisors.
	\end{prop}
	The final result of this subsection is Theorem~5.2.5~of~\cite{BPS}, which gives a convex-analytic formula for the arithmetic intersection numbers of semipositive toric arithmetic divisors. We first recall the definition of the mixed integral, which is necessary to state this result.
	\begin{dfn}\label{mixed-integral}
		For each $i = 0, \ldots , d$, let $\Delta_i$ be a compact convex subset of $M_{\mathbb{R}}$ and $g_i \colon \Delta_i \rightarrow \mathbb{R}$ be a concave function. The \textit{mixed integral} of $g_0, \ldots , g_d$ is defined as
		\begin{displaymath}
			\textup{MI}_M (g_0, \ldots , g_d ) \coloneqq  \sum_{j=0}^{d} (-1)^{d-j} \sum_{0 \leq i_0 < \ldots < i_j \leq d } \int_{\Delta_{i_0} + \ldots + \Delta_{i_j} } g_{i_0} \boxplus \ldots \boxplus g_{i_j} \textup{ dVol}_M . 
		\end{displaymath}
	\end{dfn}
	\begin{thm}\label{global-toric-h-proj}
		Let $\overline{\mathcal{D}}_0, \ldots, \overline{\mathcal{D}}_d \in  \overline{\textup{Div}}^{+}_{\mathbb{T}}(\mathcal{X}_{\overline{\Sigma}})_{\mathbb{Q}}$. For each $ 0 \leq i \leq d$, denote by $\vartheta_{i,v}$ the $v$-adic roof function of $\overline{\mathcal{D}}_i$. Then,
		\begin{displaymath}
			\overline{\mathcal{D}}_0 \cdot \ldots \cdot \overline{\mathcal{D}}_d = \sum_{v \in \mathfrak{M}(K)} \delta_{v} \cdot  \textup{MI}_M ( \vartheta_{0,v}, \ldots , \vartheta_{d,v}).
		\end{displaymath}
		In particular, if $\overline{\mathcal{D}}_0 = \ldots = \overline{\mathcal{D}}_d = \overline{\mathcal{D}}$, we have the formula
		\begin{displaymath}
			\overline{\mathcal{D}}^{d+1} = (d+1)! \int_{\Delta_D} \vartheta_{\overline{\mathcal{D}}} \, \textup{dVol}_M
		\end{displaymath}
		where $\Delta_D$ and $\vartheta_{\overline{\mathcal{D}}}$ are the corresponding convex polytope and global roof function of $\overline{\mathcal{D}}$.
	\end{thm}
	
	\subsection{Toric adelic divisors}\label{4--2}
	In this subsection, we introduce a toric analog of the group of adelic divisors of Yuan and Zhang~\cite{Y-Z}. We freely use the notation introduced in Section~\ref{global-arith-theory}. Now, we fix a (quasi-projective) toric arithmetic variety $\mathcal{X}_0 / \mathcal{O}_K$. Then, the group of  \textit{toric model arithmetic divisors} on $\mathcal{X}_0 / \mathcal{O}_K$ is defined as the direct limit
	\begin{displaymath}
		\overline{\textup{Div}}_{\mathbb{T}}(\mathcal{X}_0/\mathcal{O}_K )_{\textup{mod}} \coloneqq \varinjlim_{\mathcal{X} \in \textup{PM}_{\mathbb{T}}(\mathcal{X}_0 /\mathcal{O}_K)} \overline{\textup{Div}}_{\mathbb{T}}(\mathcal{X})_{\mathbb{Q}}.
	\end{displaymath}
	Here, $\textup{PM}_{\mathbb{T}}(\mathcal{X}_0 /\mathcal{O}_K)$ is the category of toric projective models of $\mathcal{X}_0$ over $\mathcal{O}_K$, and the maps between arithmetic divisors are the induced pullback morphisms. By Theorem~\ref{effective-fpqc-des-tor}, any such model must be of the form $\mathcal{X}_{\overline{\Sigma}}$ for some arithmetic fan $\overline{\Sigma}$ in $(N_{\mathbb{R}},K)$. Then, a boundary arithmetic divisor $(\mathcal{X},\overline{\mathcal{B}})$ of $\mathcal{X}_{0} / \mathcal{O}_{K}$ is \textit{toric} if both the projective model $\mathcal{X} = \mathcal{X}_{\overline{\Sigma}}$ and the divisor $\overline{\mathcal{B}}$ are toric. In this case, there is a boundary norm $\| \cdot \|_{\overline{\mathcal{B}}}$ on the group $\overline{\textup{Div}}_{\mathbb{T}}(\mathcal{X}_0/\mathcal{O}_K )_{\textup{mod}}$. It is given by
	\begin{displaymath}
		\| \overline{\mathcal{D}} \|_{\overline{\mathcal{B}}} \coloneqq \inf \lbrace \varepsilon \in \mathbb{Q}_{> 0} \, \vert \, - \varepsilon \cdot \overline{\mathcal{B}} \leq \overline{\mathcal{D}} \leq \varepsilon \cdot \overline{\mathcal{B}} \rbrace.
	\end{displaymath}
	This norm induces a boundary topology on $\overline{\textup{Div}}_{\mathbb{T}}(\mathcal{X}_0/\mathcal{O}_K )_{\textup{mod}}$, which is independent of the choice of toric boundary divisor. Then, we have the following definition.
	\begin{dfn}\label{adelic-dfn}
		The \textit{group of toric adelic divisors} of $\mathcal{X}_{0}$ over $\mathcal{O}_K$ is defined as the completion $\overline{\textup{Div}}_{\mathbb{T}} (\mathcal{X}_0/\mathcal{O}_K)$ of the group $\overline{\textup{Div}}_{\mathbb{T}} (\mathcal{X}_0/\mathcal{O}_K)_{\textup{mod}}$ with respect to a boundary norm. A \textit{toric adelic divisor} on $\mathcal{X}_0/\mathcal{O}_K$ is an element of this group. The semipositive cone $ \overline{\textup{Div}}_{\mathbb{T}}^{+}(\mathcal{X}_0/\mathcal{O}_K )$ and the nef cone $ \overline{\textup{Div}}_{\mathbb{T}}^{\textup{nef}}(\mathcal{X}_0 / \mathcal{O}_K )$ are the closures in the boundary topology of the respective cones of toric model divisors. Then, the space of \textit{integrable toric adelic divisors} is the difference
		\begin{displaymath}
			\overline{\textup{Div}}_{\mathbb{T}}^{\textup{int}}(\mathcal{X}_0/ \mathcal{O}_K ) \coloneqq  \overline{\textup{Div}}_{\mathbb{T}}^{\textup{nef}}(\mathcal{X}_0 / \mathcal{O}_K ) - \overline{\textup{Div}}_{\mathbb{T}}^{\textup{nef}}(\mathcal{X}_0 / \mathcal{O}_K ).
		\end{displaymath}
	\end{dfn}
	\begin{rem}
		In the following, we will restrict to the case when the quasi-projective toric arithmetic variety $\mathcal{X}_0/\mathcal{O}_K$ is of the form $\mathcal{X}_{\Sigma_{0},S}$, where $\Sigma_{0}$ is a quasi-projective fan in $N_{\mathbb{R}}$ and $S$ is a finite set of maximal ideals of $\mathcal{O}_K$. This is not a severe restriction. By Description~\ref{toric-models-OK}, every quasi-projective toric variety $\mathcal{X}_0$ has a torus invariant open subscheme of this form. Then, after realizing $\overline{\textup{Div}}_{\mathbb{T}} (\mathcal{X}_0/\mathcal{O}_K)$ as a closed subgroup of the group of the adelic divisors $\overline{\textup{Div}} (\mathcal{X}_0/\mathcal{O}_K)$, Lemma~\ref{functoriality-general} implies that the group $\overline{\textup{Div}}_{\mathbb{T}} (\mathcal{X}_0/\mathcal{O}_K)$ is a subgroup of $\overline{\textup{Div}}_{\mathbb{T}} (\mathcal{X}_{\Sigma_{0},S}/\mathcal{O}_K)$.
	\end{rem}  
	The main task of this subsection is to extend the notion of a tropical Green's function to the level of toric adelic divisors. To do this, will use the descriptions from Subsections~\ref{3-4}~and~\ref{4-1}, together with the convex-analytic tools from Appendix~A~of~\cite{Per26}. In particular, we recall the definition and basic properties of the spaces of continuous conical functions and sublinear functions.
	\begin{desc}\label{G-def}
		Recall that a function $f \colon  N_{\mathbb{R}} \rightarrow \mathbb{R}$ is \textit{conical} if, for each $\lambda \in \mathbb{R}_{\geq 0}$ and $x \in N_{\mathbb{R}}$, we have $f(\lambda x) = \lambda f(x)$. We denote by $\mathcal{C}(N_{\mathbb{R}})$ the vector space of real-valued continuous conical functions on $N_{\mathbb{R}}$. It becomes Banach space after equipping it with the $\mathcal{C}$-norm, given by
		\begin{displaymath}
			\| f \|_{\mathcal{C}} \coloneqq \inf \lbrace \varepsilon \in \mathbb{Q}_{>0} \, \vert \, |f(x)| \leq \varepsilon \cdot \| x \| \textup{ for all } x \in N_{\mathbb{R}} \rbrace.
		\end{displaymath}
		Now, the function $f$ is \textit{sublinear} if for each $\varepsilon > 0$ there exist $R_{\varepsilon}>0$ such that, whenever $\|x \|>R_{\varepsilon}$, we have $|f(x)| \leq \varepsilon \cdot (1 + \|x\|).$ We denote by $\mathcal{SL}(N_{\mathbb{R}})$ the vector space of all continuous sublinear functions on $N_{\mathbb{R}}$. It becomes a Banach space after equipping it with the $\mathcal{G}$-norm, given by
		\begin{displaymath}
			\| f \|_{\mathcal{G}} \coloneqq \inf \lbrace \varepsilon \in \mathbb{Q}_{>0} \, \vert \, |f (x) | \leq \varepsilon \cdot (1 + \| x \| ) \textup{ for all } x \in N_{\mathbb{R}} \rbrace.
		\end{displaymath}
		Then, we define the Banach space $\mathcal{G}(N_{\mathbb{R}}) \coloneqq \mathcal{SL}(N_{\mathbb{R}}) \oplus \mathcal{C}(N_{\mathbb{R}})$. The elements of $\mathcal{G}(N_{\mathbb{R}})$ naturally identify with continuous functions on $N_{\mathbb{R}}$, the direct sum norm on it is equivalent to the $\mathcal{G}$-norm, and the assignment $f \mapsto \textup{rec}(f)$ induces a continuous retraction $\textup{rec} \colon \mathcal{G}(N_{\mathbb{R}}) \rightarrow \mathcal{C}(N_{\mathbb{R}})$. Finally, we denote by $\mathcal{C}^{+}(N_{\mathbb{R}})$ and $\mathcal{G}^{+}(N_{\mathbb{R}})$ the respective cones of concave functions, which are closed in their respective spaces. Moreover, the cone $\mathcal{G}^{+}(N_{\mathbb{R}})$ coincides with the set of concave, globally Lipschitz continuous functions on $N_{\mathbb{R}}$. For details, see Appendix~A.2~of~\cite{Per26}.
	\end{desc}
	The first step is to show that tropical Green's functions descend to the level of toric model arithmetic divisors.
	\begin{lem}\label{GK-mod}
		Let $\Sigma_0$ be a smooth quasi-projective fan in $N_{\mathbb{R}}$, $S$ be a finite subset of $\textup{Max}(\mathcal{O}_K)$, and $\mathcal{X}_{\Sigma_0 , S} / \mathcal{O}_{K}$ be the associated toric arithmetic variety. Then, the assignment $\overline{\mathcal{D}} \mapsto \gamma_{\overline{\mathcal{D}}} = ( \gamma_{\overline{\mathcal{D}}_v})$, where $ \gamma_{\overline{\mathcal{D}}}$ is the tropical Green's function of $\overline{\mathcal{D}}$ (\ref{trop-gf-OK-proj})induces an injective $\mathbb{Q}$-linear map
		\begin{displaymath}
			\mathcal{G}_{K} \colon \overline{\textup{Div}}_{\mathbb{T}} (\mathcal{X}_{\Sigma_0, S}/\mathcal{O}_K )_{\textup{mod}} \longrightarrow \prod_{v \in \mathfrak{M}(K)} \mathcal{G}(N_{\mathbb{R}}).
		\end{displaymath}
	\end{lem}
	\begin{proof}
		By Description~\ref{toric-models-OK}.(iv), we have an isomorphism
		\begin{displaymath}
			\overline{\textup{Div}}_{\mathbb{T}}(\mathcal{X}_{\Sigma_0,S}/\mathcal{O}_K )_{\textup{mod}} \cong \varinjlim_{\overline{\Sigma}} \, \overline{\textup{Div}}_{\mathbb{T}}(\mathcal{X}_{\overline{\Sigma}})_{\mathbb{Q}}.
		\end{displaymath}
		where $\overline{\Sigma} = (\widetilde{\Sigma}_{\mathfrak{p}})$ runs over all smooth effective arithmetic fans in $(N_{\mathbb{R}},K)$ such that, for each $\mathfrak{p} \not \in S$, the canonical fan $\Sigma_{0, \textup{can}}$ is a subfan of $\widetilde{\Sigma}_{\mathfrak{p}}$, and ordered by refinement. Therefore, by Proposition~\ref{trop-fun-functorial}, the assignment $\overline{\mathcal{D}} \mapsto \gamma_{\overline{\mathcal{D}}}$ is well-defined. Moreover, for each place $v$, the function $\gamma_{\overline{\mathcal{D}}_v} - \Psi_D$ is bounded. Thus, it belongs to $\mathcal{G}(N_{\mathbb{R}})$. Finally, given two toric model arithmetic divisors $\overline{\mathcal{D}}_1$ and $\overline{\mathcal{D}}_2$, we can choose representatives of their equivalence class which are defined on the same projective toric arithmetic vatiety $\mathcal{X}_{\overline{\Sigma}}/\mathcal{O}_K$. Then, by Description~\ref{toric-global-paf} and Proposition~\ref{trop-gf-archim}, if their tropical Green's functions $\gamma_{\overline{\mathcal{D}}_1}$ and $\gamma_{\overline{\mathcal{D}}_2}$ coincide, they must be the same. We conclude that the $\mathbb{Q}$-linear map $\mathcal{G}_K$ is injective.
	\end{proof}
	\begin{rem}
		It is possible to characterize the image of the map $\mathcal{G}_K$. The $v$-tuple $\gamma = (\gamma_v)$ belongs to the image of $\mathcal{G}_K$ if and only if there exist an arithmetic fan $\overline{\Sigma}$ as in Description~\ref{toric-models-OK}.(iv), and $\gamma = (\gamma_v)$ satisfies the requirements from Description~\ref{toric-global-paf} and Proposition~\ref{trop-gf-archim} are satisfied.
	\end{rem}
	Our next task is to describe the boundary topology in terms of tropical Green's functions. This is done in the next two lemmas.
	\begin{lem}\label{global-bdry-div}
		Let $\Sigma_0$ be a smooth quasi-projective fan in $N_{\mathbb{R}}$, $S$ be a finite subset of $\textup{Max}(\mathcal{O}_K)$, and $\mathcal{X}_{\Sigma_0 , S} / \mathcal{O}_{K}$ be the associated toric arithmetic variety, and $\overline{\Sigma}$ be an arithmetic fan as in Description~\ref{toric-models-OK}.(iv). Then, the toric arithmetic divisor $\overline{\mathcal{B}} =  (\mathcal{B}, g_{\mathcal{B}}) \in \overline{\textup{Div}}_{\mathbb{T}}(\mathcal{X}_{\overline{\Sigma}})_{\mathbb{Q}}$ is a boundary divisor of $\mathcal{X}_{\Sigma_0 , S} / \mathcal{O}_{K}$ if and only if its tropical Green's function $\gamma_{\overline{\mathcal{B}}} = (\gamma_{\overline{\mathcal{B}}_v})$ satisfies:
		\begin{enumerate}[label=(\roman*)]
			\item For each $v \in \mathfrak{M}(K)_{\textup{fin}} \setminus S$ we have $\gamma_{\overline{\mathcal{B}}_v} (x) = 0$ for all $x \in |\Sigma_0|$, and $\gamma_{\overline{\mathcal{B}}_v}(x) < 0$ otherwise.
			\item If $v \in S \cup \mathfrak{M}(K)_{\infty}$, then $\gamma_{\overline{\mathcal{B}}_v}(x)<0$ for all $x \in N_{\mathbb{R}}$.
		\end{enumerate}
	\end{lem}
	\begin{proof}
		By the cone-orbit correspondence~\ref{cone-orbit-OK}, the open set $\mathcal{X}_{\Sigma_0}$ is given by the union of horizontal torus orbits
		\begin{displaymath}
			\mathcal{X}_{\Sigma_0} = \bigcup_{\sigma \in \Sigma_0} \mathbf{O}(\sigma ).
		\end{displaymath}
		Changing base to $\textup{Spec}(\mathcal{O}_{K,S})$ has the effect of removing the fibers of $\mathcal{X}_{\Sigma_0} $ over the primes $\mathfrak{p} \in S$. Then, by Proposition~\ref{toric-div-OK}, the closed set $\mathcal{X}_{\Sigma} \setminus \mathcal{X}_{\Sigma_0,S}$ is given by the union of:
		\begin{enumerate}[label=(\roman*')]
			\item The horizontal toric divisors corresponding to the rays $\rho \in \Sigma(1) \setminus \Sigma_0 (1)$.
			\item The vertical toric divisors corresponding to the rays $\rho \in  \widetilde{\Sigma}_{\mathfrak{p}} (1)$ which do not belong to either $\Sigma$ or $\Sigma_{0,\textup{can}}$, if $\mathfrak{p} \not \in S$,
			\item And the vertical fibers $\mathcal{X}_{\overline{\Sigma},k(\mathfrak{p})}$ over the primes $\mathfrak{p} \in S$.
		\end{enumerate}
		By the definition of arithmetic fan, $\overline{\Sigma}_{\mathfrak{p}} = \Sigma_{\textup{can}}$ for almost all $\mathfrak{p}$. Then, there is only a finite number of divisors of type (ii'). Then, conditions~(i)~and~(ii) follow immediately from Corollary~\ref{div-effective} and computing supports with Proposition~\ref{toric-div-OK}. On the other hand, it is immediate from the definition that $g_{\mathcal{B}} > 0$ if and only if $\gamma_{\overline{\mathcal{B}}_v} <0 $ for each place $v \in \mathfrak{M}(K)_{\infty}$. Thus, we get the result.
	\end{proof}
	\begin{lem}\label{global-bdry-top}
		Let $\overline{\mathcal{B}}$ be a toric boundary divisor of $\mathcal{X}_{\Sigma_0,S}$ over $\mathcal{O}_K$ and $\gamma_{\overline{\mathcal{B}}} = (\gamma_{\overline{\mathcal{B}}_v})$ be its tropical Green's function. Then, for each model toric arithmetic divisor $\overline{\mathcal{D}}$ with associated tropical Green's function $\gamma_{\overline{\mathcal{D}}} = (\gamma_{\overline{\mathcal{D}}_v})$, the following identity holds
		\begin{displaymath}
			\| \overline{\mathcal{D}} \|_{\overline{\mathcal{B}}} = \inf \lbrace \varepsilon \in \mathbb{Q}_{>0} \, \vert \, | \gamma_{\overline{\mathcal{D}}_{v}}(x) | \leq \varepsilon \cdot | \gamma_{\overline{\mathcal{B}}_{v}}(x) | \textup{ for all } x \in N_{\mathbb{R}} \textup{ and all } v \in \mathfrak{M}(K) \rbrace. 
		\end{displaymath}
	\end{lem}
	\begin{proof}
		By the characterization of effectivity in Corollary~\ref{div-effective}, the inequality $- \varepsilon \cdot \overline{\mathcal{B}} \leq \overline{\mathcal{D}} \leq \varepsilon \cdot \overline{\mathcal{B}}$ holds if and only if, for all $x \in N_{\mathbb{R}}$ and all $v \in \mathfrak{M}(K)$, we have
		\begin{displaymath}
			\varepsilon \cdot \gamma_{\overline{\mathcal{B}}_v}(x) \leq \gamma_{\overline{\mathcal{D}}_v}(x) \leq -\varepsilon \cdot \gamma_{\overline{\mathcal{B}}_v}(x).
		\end{displaymath}
		Then, the desired identity follows immediately from the definition of the boundary norm.
	\end{proof}
	We combine the previous lemmas into the following result.
	\begin{thm}\label{GK-adelic}
		Let $\Sigma_0$ be a smooth quasi-projective fan in $N_{\mathbb{R}}$, $S$ be a finite subset of $\textup{Max}(\mathcal{O}_K)$, and $\mathcal{X}_{\Sigma_0, S} / \mathcal{O}_K$ be the associated toric arithmetic variety. Then, the map $\mathcal{G}_K$ in Lemma~\ref{GK-mod} extends to an injective group morphism
		\begin{displaymath}
			\mathcal{G}_{K} \colon \overline{\textup{Div}}_{\mathbb{T}} (\mathcal{X}_{\Sigma_0, S}/\mathcal{O}_K ) \longrightarrow \prod_{v \in \mathfrak{M}(K)} \mathcal{G}(N_{\mathbb{R}})
		\end{displaymath}
		and satisfies the following continuity property: Given a convergent sequence $\lbrace \overline{\mathcal{D}}_n \rbrace_{n \in \mathbb{N}}$, then each of the induced sequences $\lbrace \gamma_{\overline{\mathcal{D}}_{n,v}} \rbrace_{n \in \mathbb{N}}$ converges in the $\mathcal{G}$-norm uniformly in $v \in \mathfrak{M}(K)$.
	\end{thm}
	\begin{proof}
		Let $\overline{\mathcal{D}} = \lbrace \overline{\mathcal{D}}_n \rbrace_{n \in \mathbb{N}}$ be a Cauchy sequence of toric model arithmetic divisors on $\mathcal{X}_{\Sigma_0,S}$ and $\lbrace \gamma_{\overline{\mathcal{D}}_n} \rbrace_{n \in \mathbb{N}}$ be the corresponding sequence of tropical Green's functions. By Lemma~\ref{global-bdry-top}, for each $\varepsilon >0$ there exists $n_{\varepsilon}$ such that, for all $n,m \geq n_{\varepsilon}$, all places $v \in \mathfrak{M}(K)$, and all $x \in N_{\mathbb{R}}$, we have
		\begin{displaymath}
			|\gamma_{\overline{\mathcal{D}}_{n,v}} (x)- \gamma_{\overline{\mathcal{D}}_{m,v}}(x) | \leq \varepsilon \cdot |\gamma_{\overline{\mathcal{B}}_v}(x)|, 
		\end{displaymath}
		where $\overline{\mathcal{B}}$ is a toric boundary divisor for $\mathcal{X}_{\Sigma_0 ,S}/\mathcal{O}_K$. Now, by Description~\ref{toric-global-paf}, almost all of the $v$-adic Green's functions $\gamma_{\overline{\mathcal{B}}_v}$ are equal to the support function $\Psi_{B}$. Then, consider the constant $C = \max_{v \in \mathfrak{M}(K)} \| \gamma_{\overline{\mathcal{B}}_v} \|_{\mathcal{G}}$. By definition of the $\mathcal{G}$-norm (Description~\ref{G-def}), we get
		\begin{displaymath}
			|\gamma_{\overline{\mathcal{D}}_{n,v}} (x)- \gamma_{\overline{\mathcal{D}}_{m,v}}(x) | \leq \varepsilon \cdot C \cdot (1 + \|x \|). 
		\end{displaymath}
		It follows that sequences $\lbrace \gamma_{\overline{\mathcal{D}}_{n,v}} \rbrace_{n \in \mathbb{N}}$ converge in the $\mathcal{G}$-norm, uniformly in $v \in \mathfrak{M}(K)$. Then, define $\gamma_{\overline{\mathcal{D}}} \coloneqq (\gamma_{\overline{\mathcal{D}}_v})$, where $\gamma_{\overline{\mathcal{D}}_v}$ is the limit of the sequence $\lbrace \gamma_{\overline{\mathcal{D}}_{n,v}} \rbrace_{n \in \mathbb{N}}$. If $\gamma_{\overline{\mathcal{D}}_v} = 0$ for all places $v$, then for each $\varepsilon >0$, there exist an $m_{\varepsilon,v}$ such that for all $m \geq m_{\varepsilon,v}$ and all $x \in N_{\mathbb{R}}$ we have
		\begin{displaymath}
			|\gamma_{\overline{\mathcal{D}}_{m,v}}(x) | \leq \varepsilon \cdot |\gamma_{\overline{\mathcal{B}}_v}(x)|.
		\end{displaymath}
		Now, for each place $v$, we choose $m \geq m_{\varepsilon,v}$. Then, for all $n \geq n_{\varepsilon}$ and all $x \in N_{\mathbb{R}}$, we get
		\begin{displaymath}
			|\gamma_{\overline{D}_{n,v}}(x) | \leq |\gamma_{\overline{\mathcal{D}}_{n,v}} (x)- \gamma_{\overline{\mathcal{D}}_{m,v}}(x) |  + |\gamma_{\overline{\mathcal{D}}_{m,v}}(x) | \leq 2 \, \varepsilon \cdot |\gamma_{\overline{\mathcal{B}}_v}(x)|.
		\end{displaymath}
		By Lemma~\ref{global-bdry-top}, the sequence $\overline{\mathcal{D}} = \lbrace \overline{\mathcal{D}}_n \rbrace_{n \in \mathbb{N}}$ converges to $0$ in the boundary topology. Therefore, $\mathcal{G}_K$ is injective. The result follows.
	\end{proof}
	Now we discuss the functorial properties of the map $\mathcal{G}_K$. Let $\Sigma_1$ and $\Sigma_2$ be smooth quasi-projective fans in $N_{\mathbb{R}}$ and $S_1 \subset S_2 \subset \textup{Max}(\mathcal{O}_K)$. Suppose that the identity map on $N_{\mathbb{R}}$ is compatible with the fans $\Sigma_1$ and $\Sigma_2$. Then, there is an induced toric birational morphism $\iota \colon \mathcal{X}_{\Sigma_1, S_1} \rightarrow \mathcal{X}_{\Sigma_2, S_2}$. Then, we have the following result.
	\begin{lem}\label{GK-functorial}
		The toric birational morphism $\iota \colon \mathcal{X}_{\Sigma_1, S_1} \rightarrow \mathcal{X}_{\Sigma_2, S_2}$ induces a pullback morphism $\iota^{\ast} \colon \overline{\textup{Div}}_{\mathbb{T}}(X_{\Sigma_2,S_2}/\mathcal{O}_K) \rightarrow \overline{\textup{Div}}_{\mathbb{T}}(X_{\Sigma_1,S_1}/\mathcal{O}_K)$ which is continuous and injective. It is a topological embedding if and only if the map $\iota$ is proper. Moreover, the map $\mathcal{G}_K$ satisfies $\mathcal{G}_K = \mathcal{G}_K \circ \iota^{\ast}$.
	\end{lem}
	\begin{proof}
		The proof is the same as in Lemma~3.2.14~of~\cite{Per26}.
	\end{proof}
	Now, we compare our constructions with the ones of Yuan~and~Zhang~\cite{Y-Z}.
	\begin{rem}[Compatibility with Yuan-Zhang's divisors]
		Let $\Sigma_0$ be a smooth quasi-projective fan in $N_{\mathbb{R}}$, $S$ be a finite subset of $\textup{Max}(\mathcal{O}_K)$, and $\mathcal{X}_{\Sigma_0, S} / \mathcal{O}_K$ be the associated toric arithmetic variety. Trivially, we have a continuous group morphism
		\begin{displaymath}
			\overline{\textup{Div}}_{\mathbb{T}} (\mathcal{X}_{\Sigma_0, S}/\mathcal{O}_K ) \longrightarrow  \overline{\textup{Div}} (\mathcal{X}_{\Sigma_0, S}/\mathcal{O}_K ).
		\end{displaymath}
		We now show that it is injective. By Proposition~\ref{global-analytification}, there is a natural continuous and injective group morphism $\textup{an}\colon \overline{\textup{Div}}(\mathcal{X}_{\Sigma_0, S}/\mathcal{O}_K) \rightarrow \overline{\textup{Div}}(\mathcal{X}_{\Sigma_0, S}^{\textup{an}})_{\textup{eqv}}$, and the image of an adelic divisor $\overline{\mathcal{D}}$ is uniquely determined by a $v$-tuple of local arithmetic divisors $(D_v, g_{D,v})_v$. If the divisor $\overline{\mathcal{D}}$ is toric, by the construction of the map $\mathcal{G}_K$ in Theorem~\ref{GK-adelic}, and the properties of the tropicalization functor from Description~\ref{trop-functor} we get an identity $g_{D,v} = - \gamma_{\overline{\mathcal{D}}_v} \circ \textup{Trop}_v$ on the analytic torus $U_{v}^{\textup{an}}$. It follows that the tropical Green's function $\gamma_{\overline{\mathcal{D}}} = (\gamma_{\overline{\mathcal{D}}_v})$ uniquely determines the equivariant divisor determined by the $v$-tuple $(D_v, g_{D,v})_v$. Therefore, we have a commutative diagram
		\begin{center}
			\begin{tikzcd}
				\overline{\textup{Div}}_{\mathbb{T}} (\mathcal{X}_{\Sigma_0, S}/\mathcal{O}_K ) \arrow[r] \arrow[d, "{\mathcal{G}_K}"] & \overline{\textup{Div}} (\mathcal{X}_{\Sigma_0, S}/\mathcal{O}_K ) \arrow[d, "{\textup{an}}"] \\
				\displaystyle \prod_{v \in \mathfrak{M}(K)} \mathcal{G}(N_{\mathbb{R}}) \arrow[r] & \overline{\textup{Div}}(\mathcal{X}_{\Sigma_0, S}^{\textup{an}})_{\textup{eqv}}
			\end{tikzcd}
		\end{center}
		The vertical maps and the lower horizontal map are injective. Therefore, the upper horizontal map must be injective. Then, the group morphism $\overline{\textup{Div}}_{\mathbb{T}} (\mathcal{X}_{\Sigma_0, S}/\mathcal{O}_K ) \rightarrow  \overline{\textup{Div}} (\mathcal{X}_{\Sigma_0, S}/\mathcal{O}_K )$ identifies the group of toric adelic divisors as a closed topological subgroup of the group of adelic divisors.
	\end{rem}
	We finish this section by working out some of these descriptions for the case of the torus $\mathcal{U}_S /\mathcal{O}_K$.
	\begin{ex}[A boundary divisor of the torus $\mathcal{U}_S$]\label{bdry-div-torus-OKS}
		Identify $N = \mathbb{Z}^{d}$ and let $e_1,...,e_d$ be the standard basis. Write $e_0 = - (e_1 + ... + e_d)$. Let $\Sigma$ be the fan whose cones are spanned by all proper subsets of $\lbrace e_0, ..., e_d \rbrace$. The toric variety $X_{\Sigma}/K$ is the projective space $\mathbb{P}^{d}/K$ with the usual toric structure. Its canonical model $\mathcal{X}_{\Sigma}/\mathcal{O}_K$ is the projective space $\mathbb{P}^{d}_{\mathcal{O}_K}$. Then, for each finite set $S \subset \textup{Max}(\mathcal{O}_K)$, the open immersion $\mathcal{U}_{S} \rightarrow \mathbb{P}_{\mathcal{O}_K}^{d}$ is a toric projective model of $\mathcal{U}_{S}$ over $\mathcal{O}_K$. Consider the toric divisor
		\begin{displaymath}
			\mathcal{B} = \sum_{i = 0}^{d} \mathcal{H}_i + \sum_{\mathfrak{p} \in S} \mathbb{P}^{d}_{k(\mathfrak{p})}, \quad \mathcal{H}_i \coloneqq \lbrace [X_0, \ldots, X_d] \in \mathbb{P}_{\mathcal{O}_K}^{d} \, \vert \, X_i =0 \rbrace.
		\end{displaymath}
		Observe that this divisor satisfies the support condition of a boundary divisor. We equip it with the toric Green's function $g_{\mathcal{B}}$ corresponding to the tropical Green's function
		\begin{displaymath}
			\gamma_{\overline{\mathcal{B}}_v}(x) = \Psi_{B}(x) -1, \quad v \in \mathfrak{M}(K)_{\infty}.
		\end{displaymath}
		Here, $\Psi_B$ is the associated support function of the restriction $B$ of $\mathcal{B}$ to the generic fiber $\mathbb{P}^{d}_{K}$. Lemma~\ref{global-bdry-div} implies that the pair $\overline{\mathcal{B}}= (\mathcal{B}, g_{\mathcal{B}})$ is a toric boundary arithmetic divisor for $\mathcal{U}_{S}$, with tropical Green's function $\gamma_{\overline{\mathcal{B}}} = (\gamma_{\overline{\mathcal{B}}_v})$ given by
		\begin{displaymath}
			\gamma_{\overline{\mathcal{B}}_v}(x) = \begin{cases}\Psi_{B}(x) -1, &  v \in S \cup\mathfrak{M}(K)_{\infty}\\ \Psi_{B}(x), & \textup{otherwise.} \end{cases}
		\end{displaymath}
		Moreover, Corollary~\ref{div-nef-ample} and Theorem~\ref{toric-psh-type} show that $\overline{\mathcal{B}}$ is semipositive. Taking Legendre-Fenchel transforms and summing over $v$, the global roof function $\vartheta_{\overline{\mathcal{B}}} \colon \Delta_B \rightarrow \mathbb{R}$ is the constant function $1$ on its associated convex polytope $\Delta_B$. By Theorem~\ref{global-nef}, the divisor $\overline{\mathcal{B}}$ is arithmetically nef.
	\end{ex}
	\begin{ex}[The boundary topology for the torus $\mathcal{U}_S$]\label{bdry-top-torus-OKS}
		Continuing with Example~\ref{bdry-div-torus-OKS}, observe that a sequence $\lbrace \overline{\mathcal{D}}_n \rbrace_{n \in \mathbb{N}}$ in $\overline{\textup{Div}}_{\mathbb{T}}(\mathcal{U}_{S}/\mathcal{O}_K)$ converges to a toric adelic divisor $\overline{\mathcal{D}}$ if and only if for each $\varepsilon>0$ there exists $n_{\varepsilon}$ such that the following holds:
		\begin{enumerate}
			\item For each $v \in S \cup \mathfrak{M}(K)_{\infty}$ and every $x \in N_{\mathbb{R}}$ we have
			\begin{displaymath}
				|\gamma_{\overline{\mathcal{D}}_{n,v}}(x) - \gamma_{\overline{\mathcal{D}}_v}(x)| \leq \varepsilon \cdot (1 + |x|).
			\end{displaymath}
			\item For each $v \in \mathfrak{M}(K)_{\textup{fin}} \setminus S$ and every $x \in N_{\mathbb{R}}$ we have
			\begin{displaymath}
				|\gamma_{\overline{\mathcal{D}}_{n,v}}(x) - \gamma_{\overline{\mathcal{D}}_v}(x)| \leq \varepsilon \cdot |x|.
			\end{displaymath}
		\end{enumerate}
		In other words, the sequence $\lbrace \gamma_{\overline{\mathcal{D}}_{n,v}} \rbrace_{n \in \mathbb{N}}$ converges to $\gamma_{\overline{\mathcal{D}}_v}$ in the $\mathcal{G}$-norm (resp. $\mathcal{C}$-norm) uniformly in $v \in \mathfrak{M}(K)$.
	\end{ex}
	\begin{ex}[Canonical models of divisors]\label{global-can-mod}
		We keep working with Examples~\ref{bdry-div-torus-OKS}~and~\ref{bdry-top-torus-OKS}. By Theorem~3.2.15~of~\cite{Per26}, every continuous and conical function $\Psi$ on $N_{\mathbb{R}}$ appears as the support function of a toric compactified divisor $D \in \textup{Div}_{\mathbb{T}}(U/K)$, where $U/K$ is the torus given by the generic fiber of $\mathcal{U}_S /\mathcal{O}_K$. Let $\lbrace D_n \rbrace_{n \in \mathbb{N}}$ be a Cauchy sequence of toric model divisors of $U/K$ representing $D$. By the above example, the sequence of canonical models $\lbrace \overline{\mathcal{D}}_{n,\textup{can}} \rbrace_{n \in \mathbb{N}}$ is Cauchy in the boundary topology of $\mathcal{U}_S /\mathcal{O}_K$. It determines a toric adelic divisor $\overline{\mathcal{D}}_{\textup{can}}$, and its tropical Green's function is $\gamma_{\overline{\mathcal{D}}_{\textup{can}}} = (\Psi_D)_{v \in \mathfrak{M}(K)}$. Then, the pair $\overline{\mathcal{D}}_{\textup{can}} = (\mathcal{D}_{\textup{can}}, g_{\mathcal{D}_\textup{can}})$ is called the  \textit{canonical arithmetic model} of $D$.
	\end{ex}
	\begin{rem}[The image of the map $\mathcal{G}_K$]
		Ideally, we want to describe the group $\overline{\textup{Div}}_{\mathbb{T}} (\mathcal{U}_{S}/\mathcal{O}_K )$ in terms of the associated tropical Green's functions. In other words, we would like to compute the image of the map $\mathcal{G}_K$. Then, Theorems~3.2.15,~3.3.13~and~4.2.10~of~\cite{Per26} impose necessary conditions for a $v$-tuple of functions $(\gamma_v)_{v \in \mathfrak{M}(K)}$ to appear as the tropical Green's function of a toric adelic divisor. However, we are still missing a coherence condition as in Remark~\ref{coherence}. This is a difficult task: By the previous remark, we would need an approximation result similar to the one in Theorem~A.4.10~of~the~loc.~cit. that works for arbitrary rational piecewise-affine functions, not just concave ones. We were not able to obtain such a result in this article. Alternatively, one could adapt Song's~ideas~\cite{Son24} to obtain a convex-analytic description of the group of toric equivariant divisors on the Berkovich analytic torus $\mathcal{U}_{S}^{\textup{an}} \rightarrow \mathcal{M}(\mathcal{O}_K)$.
	\end{rem}
	
	\section{Positivity properties of toric adelic divisors}\label{5}
	In this section, we extend the descriptions of the cones of semipositive and nef toric arithmetic divisors in Description~\ref{toric-global-paf} and Theorems~\ref{toric-psh-type}~and~\ref{global-nef} to the adelic setting. In particular, we will focus on the case of the split $d$-dimensional torus $\mathcal{U}_{S}$, where $S$ is a finite subset of $\textup{Max}(\mathcal{O}_K)$.
	\subsection{The semipositive cone}\label{4-2} The first observation is that the tropical Green's function associated to a semipositive toric adelic divisor consists of concave functions.
	\begin{prop}\label{adelic-semipositive-GF}
		Let $\Sigma_0$ be a smooth quasi-projective fan in $N_{\mathbb{R}}$, $S \subset \textup{Max}(\mathcal{O}_K)$ be finite, and $\mathcal{X}_{\Sigma_0 , S} / \mathcal{O}_{K}$ be the associated toric arithmetic variety. Then, the map $\mathcal{G}_K$ in Proposition~\ref{GK-adelic} restricts to an injective morphism of cones
		\begin{displaymath}
			\mathcal{G}_{K} \colon \overline{\textup{Div}}_{\mathbb{T}}^{+} (\mathcal{X}_{\Sigma_0, S}/\mathcal{O}_K ) \longrightarrow \prod_{v \in \mathfrak{M}(K)} \mathcal{G}^{+}(N_{\mathbb{R}}),
		\end{displaymath}
		where $\mathcal{G}^{+}(N_{\mathbb{R}})$ is the cone of globally Lipschitz concave functions on $N_{\mathbb{R}}$ (\ref{G-def}).  
	\end{prop}
	\begin{proof}
		Let $\overline{\mathcal{D}}$ be a semipositive toric adelic divisor, represented by the Cauchy sequence $\lbrace \overline{\mathcal{D}}_n \rbrace_{n \in \mathbb{N}}$ of semipositive toric model arithmetic divisors. Write $\gamma_{\overline{\mathcal{D}}} = (\gamma_{\overline{\mathcal{D}}_v})$ and $\gamma_{\overline{\mathcal{D}}_n} = (\gamma_{\overline{\mathcal{D}}_{n,v}})$ for the corresponding tropical Green's functions. By Description~\ref{toric-global-paf} and Theorem~\ref{toric-psh-type}, for each $n \in \mathbb{N}$ and $v \in \mathfrak{M}(K)$, the function $\gamma_{\overline{\mathcal{D}}_{n,v}}$ is concave. Proposition~\ref{GK-adelic} shows that, for each $v \in \mathfrak{M}(K)$, the sequence $\lbrace \gamma_{\overline{\mathcal{D}}_{n,v}} \rbrace_{n \in \mathbb{N}}$ converges in the $\mathcal{G}$-norm to the function $\gamma_{\overline{\mathcal{D}}_v}$. Therefore, it belongs to the closed cone  $\mathcal{G}^{+}(N_{\mathbb{R}})$ (See~Description~\ref{G-def}).
	\end{proof}
	The next task is to characterize the image of the map $\mathcal{G}_K$. To do this, we will restrict our discussion to the case of the torus $\mathcal{U}_S$. To justify this simplification, observe that the torus $\mathcal{U}_S$ is an open dense toric subscheme of the toric arithmetic variety $\mathcal{X}_{\Sigma_0 , S} / \mathcal{O}_{K}$. Then, by Lemma~\ref{GK-functorial}, we have a continuous injective group morphism
	\begin{displaymath}
		\iota^{\ast} \colon \overline{\textup{Div}}_{\mathbb{T}}(\mathcal{X}_{\Sigma_0,S}/\mathcal{O}_K) \longrightarrow \overline{\textup{Div}}_{\mathbb{T}}(\mathcal{U}_{S}/\mathcal{O}_K).
	\end{displaymath}
	which commutes with the map $\mathcal{G}_K$. Trivially, this map restricts to the corresponding cones of semipositive toric adelic divisors (resp. nef toric adelic divisors). Hence, we may always reduce to the case of the torus. This has the additional advantage that the boundary topology on the torus is easier to work with. We then state the main result of this subsection.
	\begin{thm}\label{semipositive-torus-OKS}
		Let $S \subset \textup{Max}(\mathcal{O}_K)$ be a finite set, and let $\mathcal{U}_{S}$ be a split torus of relative dimension $d$ over $\textup{Spec}(\mathcal{O}_{K,S})$. The map
		\begin{displaymath}
			\mathcal{G}_{K} \colon \overline{\textup{Div}}_{\mathbb{T}}^{+} (\mathcal{U}_{S}/\mathcal{O}_K ) \longrightarrow \prod_{v \in \mathfrak{M}(K)} \mathcal{G}^{+}(N_{\mathbb{R}})
		\end{displaymath}
		induces a bijection from the cone of semipositive toric adelic divisors to the cone of $v$-tuples of globally Lipschitz concave functions $\gamma = (\gamma_v)$ satisfying the following conditions:
		\begin{enumerate}[label=(\roman*)]
			\item For all places $v \in \mathfrak{M}(K)_{\textup{fin}} \setminus S$, we have $\gamma_v(0) \in \mathbb{Q}$.
			\item There exists a conical concave function $\Psi \colon N_{\mathbb{R}} \rightarrow \mathbb{R}$ such that $\textup{rec}(\gamma_v) = \Psi$ for all places $v \in \mathfrak{M}(K)$.
			\item For each $\varepsilon >0$ there is a finite set $S(\varepsilon) \subset \mathfrak{M}(K)_{\textup{fin}}$ containing the set $S$ and such that $\| \gamma_v - \Psi  \|_{\mathcal{C}} \leq \varepsilon$ for all places $v \in \mathfrak{M}(K)_{\textup{fin}} \setminus S(\varepsilon)$. In particular, the definition of the $\mathcal{C}$-norm (\ref{G-def}) implies that $\gamma_v (0) = \Psi (0) = 0$ for almost all finite places $v$.
		\end{enumerate}
	\end{thm}
	\begin{proof}
		First, we show that these conditions are necessary. Let $\overline{\mathcal{D}}$ be any semipositive toric adelic divisor and $\overline{\mathcal{B}}$ be the toric boundary divisor from Example~\ref{bdry-div-torus-OKS}. Since $\overline{\mathcal{B}}$ is semipositive, Remark~\ref{decreasing} implies the existence of a decreasing Cauchy sequence of semipositive toric model divisors $\lbrace \overline{\mathcal{D}}_n \rbrace_{n \in \mathbb{N}}$ representing $\overline{\mathcal{D}}$. Write $\gamma_{\overline{\mathcal{D}}} = (\gamma_{\overline{\mathcal{D}}_v})$ and $\gamma_{\overline{\mathcal{D}}_n} = (\gamma_{\overline{\mathcal{D}}_{n,v}})$ for the corresponding tropical Green's functions. In particular, for each finite place $v$, the functions $\gamma_{\overline{\mathcal{D}}_{n,v}}$ are rational and piecewise linear. By Example~\ref{bdry-top-torus-OKS}.(ii), for each $v \in \mathfrak{M}(K)_{\textup{fin}} \setminus S$, the sequence of rational numbers $\lbrace \gamma_{\overline{\mathcal{D}}_{n,v}} (0)\rbrace_{ n \in \mathbb{N}}$ is eventually constant. This means that $\gamma_{\overline{\mathcal{D}}_{v}} (0)$ is a rational number, and therefore, Condition~\textit{(i)} is satisfied. Combining Description~\ref{toric-global-paf} and Theorem~\ref{toric-psh-type}, we obtain that $\textup{rec}(\gamma_{\overline{\mathcal{D}}_{n,v}}) = \Psi_{D_n}$ for each $n \in \mathbb{N}$ and $v \in \mathfrak{M}(K)_{\textup{fin}}$. Here, $\Psi_{D_n}$ is the support function of the restriction of $\mathcal{D}_n$ to the generic fiber, and in particular, it is concave. By Description~\ref{G-def}, the recession map $\textup{rec} \colon \mathcal{G}(N_{\mathbb{R}}) \rightarrow \mathcal{C}(N_{\mathbb{R}})$ is a continuous retraction. Then, for each $v \in \mathfrak{M}(K)$, we have
		\begin{displaymath}
			\lim_{n \in \mathbb{N}}  \Psi_{D_n} = 	\lim_{n \in \mathbb{N}} \textup{rec}(\gamma_{\overline{\mathcal{D}}_{n,v}}) = \textup{rec} \left( \lim_{n \in \mathbb{N}} \gamma_{\overline{\mathcal{D}}_{n,v}} \right) = \textup{rec}(\gamma_{\overline{\mathcal{D}}_v}).
		\end{displaymath}
		Since each $\Psi_{D_n}$ is conical and concave, so is their limit $\Psi_D$. Therefore, Condition~\textit{(ii)} is satisfied. Now we show that Condition~\textit{(iii)} holds, so we fix an $\varepsilon >0$. Description~\ref{toric-global-paf} implies that, for each $n \in \mathbb{N}$, there exists a finite set $S(n) \in \mathfrak{M}(K)_{\textup{fin}}$ such that $\gamma_{\overline{\mathcal{D}}_{n,v}} = \Psi_{D_n}$ for all $v \in \mathfrak{M}(K)_{\textup{fin}} \setminus S(n)$. Enlarging $S(n)$ if necessary, we may assume $S \subset S(n)$. By Example~\ref{bdry-top-torus-OKS}, there exists $n(\varepsilon)$ such that, for all $n \geq n(\varepsilon)$ and all places $v \in \mathfrak{M}(K)_{\textup{fin}} \setminus S$, we have
		\begin{displaymath}
			\| \Psi_{D_n} - \Psi_{D} \|_{\mathcal{C}} \leq \varepsilon/2 \quad \textup{and} \quad \| \gamma_{\overline{\mathcal{D}}_{n,v}} - \gamma_{\overline{\mathcal{D}}_v} \|_{\mathcal{C}} \leq \varepsilon/2.
		\end{displaymath}
		Then, let $S(\varepsilon ) = S(n(\varepsilon))$. The triangle inequality and the two inequalities above give
		\begin{displaymath}
			\| (\gamma_{\overline{\mathcal{D}}_v} - \Psi_{D}) + (\Psi_{D_{n(\varepsilon)}} -\gamma_{\overline{\mathcal{D}}_{n(\varepsilon),v}})  \|_{\mathcal{C}} \leq \| \Psi_{D_{n(\varepsilon)}} - \Psi_{D} \|_{\mathcal{C}}  + \| \gamma_{\overline{\mathcal{D}}_{n(\varepsilon),v}} - \gamma_{\overline{\mathcal{D}}_v} \|_{\mathcal{C}}  \leq \varepsilon.
		\end{displaymath}
		Condition~\textit{(iii)} follows from the fact that $\Psi_{D_{n(\varepsilon)}} = \gamma_{\overline{\mathcal{D}}_{n(\varepsilon),v}}$ for all $v \in \mathfrak{M}(K)_{\textup{fin}} \setminus S(\varepsilon)$.
		
		Now, we show that Conditions~\textit{(i)--(iii)} are sufficient. Let $\gamma = (\gamma_v)$ and $\Psi$ be as in the statement. We want to find a Cauchy sequence of semipositive toric model arithmetic divisors $ \lbrace \overline{\mathcal{D}}_n \rbrace_{n \in \mathbb{N}}$ such that the induced sequence of $v$-tuples $\gamma_{\overline{\mathcal{D}}_n} = (\gamma_{\overline{\mathcal{D}}_{n,v}})$ approximates $\gamma$ as in Example~\ref{bdry-top-torus-OKS}. By continuity and injectivity of the map $\mathcal{G}_K$ (Proposition~\ref{adelic-semipositive-GF}), conclude that the $v$-tuple $\gamma$ coincides with the tropical Green's function $\gamma_{\overline{\mathcal{D}}}$ of the semipositive toric adelic divisor $\overline{\mathcal{D}} = \lbrace \overline{\mathcal{D}}_n \rbrace_{n \in \mathbb{N}}$. To do this, we will apply the density results developed in Appendix~A~of~\cite{Per26} as follows:
		\begin{enumerate}[label=(\arabic*)]
			\item By Theorem~A.4.4, there exists a sequence of rational concave support functions $\lbrace \Psi_n  \rbrace_{n \in \mathbb{N}}$ on $N_{\mathbb{R}}$, which is increasing and converges in the $\mathcal{C}$-norm to $\Psi$. Denote $\varepsilon_n = \| \Psi_n - \Psi \|_{\mathcal{C}}$. Then, the sequence $\lbrace \varepsilon_n \rbrace_{n \in \mathbb{N}}$ is decreasing and converges to $0$.
			\item For each $v \in \mathfrak{M}(K)_{\textup{fin}} \setminus S$, apply Theorem~~A.4.10~of~\cite{Per26} to get an increasing sequence $\lbrace f_{n,v} \rbrace_{n \in \mathbb{N}}$ of concave rational piecewise affine functions on $N_{\mathbb{R}}$ satisfying
			\begin{displaymath}
				f_{n,v}(0)=\gamma_v (0), \quad \textup{rec}(f_{n,v})= \Psi_n, \quad \textup{and} \quad \| f_{n,v} - \gamma_v \|_{\mathcal{C}} \leq 2 \, \varepsilon_n.
			\end{displaymath}
			\item For each $v \in S \cup \mathfrak{M}(K)_{\infty}$, we apply Lemma~A.2.17 to find an increasing sequence $\lbrace \gamma_{n,v} \rbrace_{n \in \mathbb{N}}$ of globally Lipschitz concave functions on $N_{\mathbb{R}}$ which converges in the $\mathcal{G}$-norm to $\gamma_v$ and such that $|\gamma_{n,v} - \Psi_n |$ is bounded. Moreover, the proof of this Lemma makes this convergence explicit, meaning that we can choose the functions $\gamma_{n,v}$ so that $\| \gamma_{n,v} - \gamma_v \|_{\mathcal{G}} \leq \varepsilon_n$. Then, by Corollary~A.4.8, we may assume that the functions $\gamma_{n,v}$ are rational piecewise affine. 
		\end{enumerate}
		For each $n \in \mathbb{N}$, let $S(\varepsilon_n)$ be a finite set of finite places such that property \textit{(iii)} holds. Without loss of generality, we may assume that $S(\varepsilon_0 ) \subset S(\varepsilon_1 ) \subset \ldots \subset \mathfrak{M}(K)_{\textup{fin}}$. Then, for each $n \in \mathbb{N}$, define the functions
		\begin{displaymath}
			\gamma_{n,v} = \begin{cases} f_{n,v}, & \textup{ if } v \in S(\varepsilon_n) \setminus S \\ \Psi_n & \textup{ if } v \in \mathfrak{M}(K)_{\textup{fin}} \setminus S(\varepsilon_n). \end{cases}
		\end{displaymath}
		Let $w \in \mathfrak{M}(K)_{\textup{fin}} \setminus S(\varepsilon_0)$ and observe that $\gamma_w (0) = 0$. This implies $\gamma_{n,w}(0)=0$ for all $n \in \mathbb{N}$. By construction, for each $n \in \mathbb{N}$,
		\begin{displaymath}
			\| \gamma_{n,v} - \gamma_v \|_{\mathcal{C}} = \begin{cases} \| f_{n,v} - \gamma_v \|_{\mathcal{C}} \leq 2 \, \varepsilon_n &  \textup{ if } v \in S(\varepsilon_n) \setminus S \\ \| \Psi_n - \gamma_v \|_{\mathcal{C}} \leq \varepsilon_n & \textup{ if } v \in \mathfrak{M}(K)_{\textup{fin}} \setminus S(\varepsilon_n). \end{cases}
		\end{displaymath}
		Then, define the $v$-tuple $\gamma_n = (\gamma_{n,v})$. For a fixed $n \in \mathbb{N}$, all of the $\gamma_{n,v}$ are rational piecewise affine and $\gamma_{n,v} = \Psi_n$ for all but a finite number of finite places $v$. Then, iterating the procedure in Lemma~\ref{common-refinement} a finite number of times, we can find a smooth effective arithmetic fan $\overline{\Sigma}_{n} = (\widetilde{\Sigma}_{n,v})$ in $(N_{\mathbb{R}},K)$ such that $\gamma_{n,v}$ is a rational piecewise affine function on the polyhedron $\Pi_{n,v}$ obtained by intersecting $\widetilde{\Sigma}_v$ with the hyperplane $N_{\mathbb{R}} \times \lbrace 1 \rbrace$. By Description~\ref{toric-global-paf}, Corollary~\ref{div-nef-ample}, and Theorem~\ref{toric-psh-type}, the $v$-tuple $\gamma_n = (\gamma_{n,v})$ determines a semipositive toric divisor $\overline{\mathcal{D}}_n$ on the projective toric variety $\mathcal{X}_{\overline{\Sigma}_n}/\mathcal{O}_K$. Moreover, the proof of Lemma~\ref{common-refinement} and Description~\ref{toric-models-OK} imply that $\mathcal{X}_{\overline{\Sigma}_n}$ is a toric projective model over $\mathcal{O}_K$ of the torus $\mathcal{U}_S$. Then,  Example~\ref{bdry-top-torus-OKS} show that $\lbrace \overline{\mathcal{D}}_n \rbrace_{n \in \mathbb{N}}$ is a Cauchy sequence in the boundary topology of $\mathcal{U}_{S}$, whose limit $\overline{\mathcal{D}}$ satisfies $\gamma_{\mathcal{\overline{\mathcal{D}}}} = \gamma$.
	\end{proof}
	By Legendre-Fenchel duality, the above result can be stated in terms of the $v$-adic roof functions.
	\begin{cor}\label{semipositive-OKS-roofs}
		The assignment $\overline{\mathcal{D}} \mapsto (\vartheta_{\overline{\mathcal{D}}_v})$ induces a bijection between the cone of semipositive toric adelic divisors $\overline{\mathcal{D}}$ and the cone of $v$-tuples $(\vartheta_v)$ of closed concave functions $\vartheta_v \colon M_{\mathbb{R}} \rightarrow \mathbb{R}_{-\infty}$ satisfying the following conditions:
		\begin{enumerate}[label=(\roman*)]
			\item For each place $v \in \mathfrak{M}(K)_{\textup{fin}}\setminus S$, the supremum $\sup \vartheta_v$ is a rational number.
			\item There is a compact convex set $\Delta \subset M_{\mathbb{R}}$ such that, for each $v \in \mathfrak{M}(K)$, the closure of the effective domain $\textup{dom}(\vartheta_v)\coloneqq \lbrace y \in M_{\mathbb{R}} \, \vert \, \vartheta_v (y) > - \infty \rbrace$ coincides with $\Delta$.
			\item For each $\varepsilon >0$ there is a finite set $S(\varepsilon) \subset \mathfrak{M}(K)_{\textup{fin}}$ containing the set $S$ and such that $ \sup \vartheta_v = 0$ and $\vartheta_v \boxplus \iota_{\overline{\textup{B}}(0,\varepsilon)} \geq \iota_{\Delta}$ for every $v \in \mathfrak{M}(K)_{\textup{fin}}\setminus S(\varepsilon)$.
		\end{enumerate}
	\end{cor}
	\begin{proof}
		Follows from taking Legendre-Fenchel transforms. Applying Propositions~A.1.11,~A.2.3, and Lemma~A.2.16~of~\cite{Per26} to Conditions~\textit{(i)--(iii)} respectively, they are shown to be equivalent to the ones in Theorem~\ref{semipositive-torus-OKS}.
	\end{proof}
	
	\subsection{The nef cone} The main goal of this subsection is to extend the notion of global roof function from Definition~\ref{global-roof-fun} to the level of semipositive toric adelic divisors. Afterwards, we will use this function to characterize (arithmetic) nefness. First, we make the following remark, recalling the situation from the proof of Theorem~\ref{semipositive-torus-OKS} and fixing notation that will be used repeatedly throughout the rest of the article.
	\begin{rem}\label{notation-decreasing-seq}
		Let $S$ be a finite subset of maximal ideals of $\mathcal{O}_K$ and denote by $\mathcal{U}_S$ a split torus of dimension $d$ over $\mathcal{O}_{K,S}$, regarded as a toric arithmetic variety over $\mathcal{O}_K$. Let $\overline{\mathcal{D}}$ be a semipositive toric adelic divisor of $\mathcal{U}_S/\mathcal{O}_K$ and $\overline{\mathcal{B}}$ be the toric boundary divisor from Example~\ref{bdry-div-torus-OKS}. Since $\overline{\mathcal{B}}$ is semipositive, Remark~\ref{decreasing} implies the existence of a decreasing Cauchy sequence of semipositive toric model divisors $\lbrace \overline{\mathcal{D}}_n \rbrace_{n \in \mathbb{N}}$ representing $\overline{\mathcal{D}}$. Then, we fix the following notation:
		\begin{enumerate}
			\item For each $n \in \mathbb{N}$, denote by $D_n$ the nef toric divisor obtained by restriction to the generic fiber. Then, we write $\Psi_{D_n}$ and $\Delta_{D_n}$ for the corresponding support function and compact convex set (Theorem~).
			\item Denote by $\gamma_{\overline{\mathcal{D}}_n} = (\gamma_{\overline{\mathcal{D}}_{n,v}})$ the tropical Green's function associated to $\overline{\mathcal{D}}_n$. Then, we write $\vartheta_{\overline{\mathcal{D}}_v}$ and $\vartheta_{\overline{\mathcal{D}}_n}$ for the corresponding $v$-adic and global roof functions, respectively.
			\item Similarly, we let $D$, $\Psi_{D}$, $\Delta_{D}$, $\overline{\mathcal{D}}_v$, $\gamma_{\overline{\mathcal{D}}}=(\gamma_{\overline{\mathcal{D}}_v})$ and $\vartheta_{\overline{\mathcal{D}}_{v}}$ be the corresponding limit objects associated to $\overline{\mathcal{D}}$. 
			\item Moreover, if $\overline{\mathcal{D}}$ is nef, we will assume that each $\overline{\mathcal{D}}_n$ is nef.
		\end{enumerate}
	\end{rem}
	For the reader's convenience, we also include Lemma~A.2.7~of~\cite{Per26} below.
	\begin{lem}\label{geom-convergence}
		Let $\Psi: N_{\mathbb{R}} \rightarrow \mathbb{R}$ be conical and concave, with stabilty set $\textup{stab}(\Psi) = \Delta$. Let $\lbrace \Psi_n \rbrace_{n \in \mathbb{N}}$ be a sequence of conical concave functions on $N_{\mathbb{R}}$ and write $\Delta_n = \textup{stab}(\Psi_n)$. Then, the sequence $\lbrace \Psi_n  \rbrace_{n \in \mathbb{N}}$ converges to $\Psi$ in the $\mathcal{C}$-norm if and only if the sequence $\lbrace \Delta_n  \rbrace_{n \in \mathbb{N}}$ converges to $\Delta$ in the Hausdorff distance, i.e. for each $\varepsilon > 0$ there is a $n_{\varepsilon}$ such that for all $n \geq n_{\varepsilon}$ we have 
		\begin{displaymath}
			\Delta_n \subset \Delta + \overline{\textup{B}}(0,\varepsilon) \quad \textup{and} \quad \Delta \subset \Delta_n + \overline{\textup{B}}(0,\varepsilon).
		\end{displaymath}
		Moreover, if the sequence $\lbrace \Psi_n  \rbrace_{n \in \mathbb{N}}$ converges to $\Psi$ in the $\mathcal{C}$-norm and $\Delta$ has non-empty interior, then for every $\varepsilon >0$ there is $n_{\varepsilon}$ such that for all $n \geq n_\varepsilon$ we have $\Delta_{\varepsilon} \subset \Delta_n$, where
		\begin{displaymath}
			\Delta_{\varepsilon} \coloneqq \textup{Conv} (\lbrace y \in \Delta \, \vert \, \textup{dist}(y,\textup{rb}(\Delta)) \geq \varepsilon \rbrace)
		\end{displaymath}
		and $\textup{dist}(y,\textup{rb}(\Delta))$ is the distance from the point $y$ to the relative boundary $\textup{rb}(\Delta)$ of $\Delta$.
	\end{lem}
	Let $X/K$ be a proper variety of dimension $d$ and $D \in \textup{Div}(X)$ be a Cartier divisor. The \textit{volume} of $D$ is defined as the limit
	\begin{displaymath}
		\textup{vol}(X,D) \coloneqq \frac{1}{d!}\limsup_{n \rightarrow \infty} \frac{\textup{dim}_{K}(\Gamma (X, \mathcal{O}(D)^{\otimes n})}{n^{d}},
	\end{displaymath}
	where $\mathcal{O}(D)$ is the line bundle associated to $D$ and $\textup{dim}_{K}(\Gamma (X, \mathcal{O}(D))$ is the dimension over $K$ of the space of global sections of $\mathcal{O}(D)$. The divisor $D$ is said to be \textit{big} if its volume is positive. This definition extends naturally to $\mathbb{Q}$-divisors and $\mathbb{R}$-divisors. In the case where $X=X_{\Sigma}/K$ is a toric variety, and $D$ is a toric Cartier divisor, we get that $D$ is big if and only if $\textup{dim}(\Delta_D) = d$, where $\Delta_D$ is the polyhedron associated to $D$. This is implied by~Proposition~4.9~of~\cite{BMPS}. Therefore, a toric divisor $D$ is nef and big if and only if $\Delta_D$ is convex and $\textup{Vol}_M (\Delta_D)>0$. We will use this description to generalize this notion to the compactified (resp. adelic) setting.
	\begin{dfn}
		We say that a nef toric compactified divisor $D \in \textup{Div}_{\mathbb{T}}^{\textup{nef}}(U/K)$ is \textit{big} if $\textup{dim}(\Delta_D) = d$. Then, a semipositive toric adelic divisor $\overline{\mathcal{D}} \in  \overline{\textup{Div}}_{\mathbb{T}}^{+} (\mathcal{U}_{S}/\mathcal{O}_K )$ is \textit{generically big} if the nef toric compactified divisor $D$ on $U/K$ induced by restriction is big.
	\end{dfn}
	Now, we state a technical lemma. It will be used to extend the notion of global roof functions to the adelic setting. 
	\begin{lem}\label{finite-sum}
		With the notation of Remark~\ref{notation-decreasing-seq}, suppose that $\overline{\mathcal{D}}$ is generically big. For each $\varepsilon>0$, let $\Delta_{D,\varepsilon}$ be the compact convex set from Lemma~\ref{geom-convergence} and $\Psi_{\Delta_{D,\varepsilon}}$ be its corresponding conical concave function. Then, there exists a number $n_{\varepsilon}$ and a finite set of finite places $R(\varepsilon) \subset \mathfrak{M}(K)_{\textup{fin}}$ containing $S$ such that $\vartheta_{\overline{\mathcal{D}}_{n,v}}(y) = 0$ for all $n \geq n_{\varepsilon}$, $v \in \mathfrak{M}(K)_{\textup{fin}} \setminus R(\varepsilon)$ and $y \in \Delta_{D, \varepsilon}$.
	\end{lem}
	\begin{proof}
		Observe that the sequence of compact convex sets $\lbrace \Delta_{D_n} \rbrace_{n \in \mathbb{N}}$ is decreasing, and each of the local roof functions $\vartheta_{\overline{\mathcal{D}}_{n,v}}$ continuous on $\Delta_{D_n}$. Thus, it makes sense to evaluate $\vartheta_{\overline{\mathcal{D}}_{n,v}}$ at points
		\begin{displaymath}
			y \in \Delta_{D,\varepsilon} \subset \Delta_{D} \subset \Delta_{D_n}. 
		\end{displaymath}
		Without loss of generality, we may assume that $\Delta_{D,\varepsilon} \neq \emptyset$; since $\Delta_D$ has a non-empty interior, we only need to choose $\varepsilon$ small enough. Then, the value $\vartheta_{\overline{\mathcal{D}}_{n,v}}(y) $ exists and is finite. Define $R(\varepsilon) \coloneqq S(\varepsilon / 2)$, where $S(\varepsilon)$ is the set satisfying property \textit{(iii)} in Theorem~\ref{semipositive-torus-OKS}. By Remark~\ref{bdry-top-torus-OKS}, there exists $n_{\varepsilon}$ such that for all $v \in \mathfrak{M}(K)_{\textup{fin}} \setminus S$ and all $n \geq n_{\varepsilon}$ we have
		\begin{displaymath}
			\| \gamma_{\overline{\mathcal{D}}_{n,v}} - \gamma_{\overline{\mathcal{D}}_v} \|_{\mathcal{C}} \leq \varepsilon / 2. 
		\end{displaymath}
		By property  \textit{(iii)} in Theorem~\ref{semipositive-torus-OKS} and the triangle inequality, we have
		\begin{displaymath}
			\| \gamma_{\overline{\mathcal{D}}_{n,v}} - \Psi_D \|_{\mathcal{C}} \leq  \| \gamma_{\overline{\mathcal{D}}_{n,v}} - \gamma_{\overline{\mathcal{D}}_v} \|_{\mathcal{C}} +  \| \gamma_{\overline{\mathcal{D}}_v} - \Psi_{D} \|_{\mathcal{C}} \leq \varepsilon.
		\end{displaymath}
		By definition of the $\mathcal{C}$-norm, this implies
		\begin{displaymath}
			\Psi_D (x)  \geq \gamma_{\overline{\mathcal{D}}_{n,v}}(x) - \varepsilon \cdot |x|,
		\end{displaymath}
		for all $x \in N_{\mathbb{R}}$. On the other hand, by the proof of Lemma~\ref{geom-convergence}, we have $\Delta_{D,\varepsilon} + \overline{\textup{B}}(0,\varepsilon) \subset \Delta_D$. This is equivalent to
		\begin{displaymath}
			\Psi_{\Delta_{D,\varepsilon}}(x) - \varepsilon \cdot |x| \geq \Psi_D (x),
		\end{displaymath}
		for all $x \in N_{\mathbb{R}}$. Therefore, we get
		\begin{displaymath}
			\Psi_{\Delta_{D,\varepsilon}}(x) - \varepsilon \cdot |x|  \geq  \gamma_{\overline{\mathcal{D}}_{n,v}}(x) - \varepsilon \cdot |x|,
		\end{displaymath}
		for all $v \in \mathfrak{M}(K)_{\textup{fin}} \setminus S$, $n \geq n_{\varepsilon}$  and $x \in N_{\mathbb{R}}$. We conclude that $\Psi_{\Delta_{D,\varepsilon}}  \geq \gamma_{\overline{\mathcal{D}}_{n,v}}$. Taking the Legendre-Fenchel transform, we obtain $\vartheta_{\overline{\mathcal{D}}_{n,v}} \geq \iota_{\Delta_{D,\varepsilon}}$, where $\iota_{\Delta_{D,\varepsilon}}$ is the indicator function of the set $\Delta_{D,\varepsilon}$. In particular $	\vartheta_{\overline{\mathcal{D}}_{n,v}}(y) \geq 0$ for all $v \in \mathfrak{M}(K)_{\textup{fin}} \setminus S$, $n \geq n_{\varepsilon}$ and $y \in \Delta_{D,\varepsilon}$. The reverse inequality follows from the fact that $\gamma_{\overline{\mathcal{D}}_{n,v}}(0) = 0$. Indeed, by the definition of the Legendre-Fenchel transform
		\begin{align*}
			0 = \gamma_{\overline{\mathcal{D}}_{n,v}}(0) = \vartheta_{\overline{\mathcal{D}}_{n,v}}^{\vee}(0) =  \inf \lbrace - \vartheta_{\overline{\mathcal{D}}_{n,v}}(y) \, \vert \,  y \in \Delta_{D_{n,v}} \rbrace = - \sup \lbrace \vartheta_{\overline{\mathcal{D}}_{n,v}}(y) \, \vert \,  y \in \Delta_{D_{n,v}} \rbrace.
		\end{align*}
	\end{proof}
	\begin{rem}\label{local-roof-neg}
		With the notation of Lemma~\ref{notation-decreasing-seq}, property \textit{(iii)} in Theorem~\ref{semipositive-torus-OKS} implies that $\vartheta_{\overline{\mathcal{D}}_v} \leq 0$ for almost all $v \in \mathfrak{M}(K)_{\textup{fin}}$. We just need to copy the last part of the previous proof.
	\end{rem}
	Then, we show the existence and expected properties of the global roof function of a generically big semipositive toric adelic divisor (resp. nef toric adelic divisor).
	\begin{prop}\label{global-roof-OKS}
		With the notation of Remark~\ref{notation-decreasing-seq}, suppose that $\overline{\mathcal{D}}$ is nef or generically big. Then, the following statements hold:
		\begin{enumerate}[label=(\roman*)]
			\item The sequence of global roof functions $\lbrace \vartheta_{\overline{\mathcal{D}}_n} \rbrace_{n \in \mathbb{N}}$ converges decreasingly (pointwise) to a closed concave function
			\begin{displaymath}
				\vartheta_{\overline{\mathcal{D}}} \colon \Delta_{D} \longrightarrow \mathbb{R}_{-\infty}.
			\end{displaymath}
			Moreover, if $\overline{\mathcal{D}}$ is nef, then $\vartheta_{\overline{\mathcal{D}}}(y) \geq 0$ for all $y \in \Delta_D$.
			\item For each $y \in \textup{ri}(\Delta_D)$ we have
			\begin{displaymath}
				\vartheta_{\overline{\mathcal{D}}} (y) = \sum_{v \in \mathfrak{M}(K)}  \delta_v \cdot \vartheta_{\overline{\mathcal{D}}_{v}}(y) > - \infty,
			\end{displaymath}
			where the $\delta_v$ are the weights from Definition~\ref{global-roof-fun}. If $\overline{\mathcal{D}}$ is generically big, the above sum is finite.
			\item If $\overline{\mathcal{D}}$ is nef then there exists a constant $\beta$ such that $|\vartheta_{\overline{\mathcal{D}}_v}(y)| \leq \beta $ for each place $v \in \mathfrak{M}(K)$ and $y \in \Delta_D$.
		\end{enumerate}
	\end{prop}
	\begin{proof}
		The arguments are different, depending on whether $\overline{\mathcal{D}}$ is nef or generically big. First, we will prove the results for the nef case. By Remark~\ref{notation-decreasing-seq}, the decreasing sequence $\lbrace \overline{\mathcal{D}}_n \rbrace_{n \in \mathbb{N}}$ consists of nef divisors and converges in the boundary topology to $\mathcal{\overline{\mathcal{D}}}$. On the other hand, Theorem~\ref{global-nef} implies that $\vartheta_{\overline{\mathcal{D}}_n}(y) \geq 0$ for each $n \in \mathbb{N}$ and every $y \in \Delta_{D_n}$. This shows that the sequence of closed concave functions $\lbrace \vartheta_{\overline{\mathcal{D}}_n} \rbrace_{n \in \mathbb{N}}$ is decreasing and bounded below by $0$ on $\Delta_D$. Therefore, it converges pointwise to a non-negative, closed concave function on $\Delta_D$. Concretely, for each $y \in \Delta_D$,
		\begin{displaymath}
			\vartheta_{\overline{\mathcal{D}}}(y) \coloneqq \lim_{n \in \mathbb{N}}  \vartheta_{\overline{\mathcal{D}}_n}(y) = \lim_{n \in \mathbb{N}}  \sum_{v \in \mathfrak{M}(K)} \delta_v \cdot \vartheta_{\overline{\mathcal{D}}_{n,v}}(y) \geq 0.
		\end{displaymath}
		Given $v \in \mathfrak{M}(K)$, the sequence $\lbrace  \vartheta_{\overline{\mathcal{D}}_{n,v}}(y)  \rbrace_{n \in \mathbb{N}}$ is decreasing and converges to $\vartheta_{\overline{\mathcal{D}}_{v}}(y)$. Then, the monotone convergence theorem from measure theory shows that we can interchange limits and series. Therefore, we obtain the identity
		\begin{displaymath}
			\vartheta_{\overline{\mathcal{D}}}(y) = \lim_{n \in \mathbb{N}}  \sum_{v \in \mathfrak{M}(K)} \delta_v \cdot \vartheta_{\overline{\mathcal{D}}_{n,v}}(y) =\sum_{v \in \mathfrak{M}(K)} \delta_v \cdot \lim_{n \in \mathbb{N}}  \vartheta_{\overline{\mathcal{D}}_{n,v}}(y)= \sum_{v \in \mathfrak{M}(K)}  \delta_v \cdot \vartheta_{\overline{\mathcal{D}}_{v}}(y) \geq 0.
		\end{displaymath}
		Finally, we show that the functions $\vartheta_{\overline{\mathcal{D}}_v}$ are bounded. By Remark~\ref{local-roof-neg}, there exists a finite set $R \subset \mathfrak{M}(K)_{\textup{fin}}$ containing $S$ such that $0 \geq \vartheta_{\overline{\mathcal{D}}_v}(y)$ for all places $v \in \mathfrak{M}(K)_{\textup{fin}} \setminus R$. Since the set $R \cup \mathfrak{M}(K)_{\infty}$ is finite, Proposition~A.1.11~of~\cite{Per26} implies the existence of a constant $C$ such that $C \geq \vartheta_{\overline{\mathcal{D}}_v}(y)$ for all $v \in R \cup \mathfrak{M}(K)_{\infty}$ and $y \in \Delta_D$. Then, for each $y \in \Delta_D$, we have
		\begin{displaymath}
			0 \geq  \sum_{v \in \mathfrak{M}(K)_{\textup{fin}} \setminus R}  \delta_v \cdot \vartheta_{\overline{\mathcal{D}}_{v}}(y) \geq  - \sum_{v \in  R \cup \mathfrak{M}(K)_{\infty}}  \delta_v \cdot \vartheta_{\overline{\mathcal{D}}_{v}}(y) \geq - \sum_{v \in  R \cup \mathfrak{M}(K)_{\infty}} \delta_v \, C.
		\end{displaymath}
		The above inequality, together with the previous upper bounds, gives the following estimates:
		\begin{enumerate}[label=(\arabic*)]
			\item If $w \in \mathfrak{M}(K)_{\textup{fin}} \setminus R$, we have
			\begin{displaymath}
				0 \geq \delta_w \cdot \vartheta_{\overline{\mathcal{D}}_{w}}(y) \geq \sum_{v \in \mathfrak{M}(K)_{\textup{fin}} \setminus R}  \delta_v \cdot \vartheta_{\overline{\mathcal{D}}_{v}}(y) \geq  -\sum_{v \in  R \cup \mathfrak{M}(K)_{\infty}} \delta_v \, C
			\end{displaymath}
			for each $y \in \Delta_D$. Therefore, the function $\vartheta_{\overline{\mathcal{D}}_{w}}$ is bounded.
			\item If $w \in R \cup \mathfrak{M}(K)_{\infty}$, we have
			\begin{displaymath}
				\delta_w \, C \geq \delta_w  \cdot \vartheta_{\overline{\mathcal{D}}_{w}} (y) \geq - \sum_{v \in  (R \cup \mathfrak{M}(K)_{\infty}) \setminus \lbrace w \rbrace}  \delta_v \cdot \vartheta_{\overline{\mathcal{D}}_{v}}(y) \geq - \sum_{v \in  (R \cup \mathfrak{M}(K)_{\infty}) \setminus \lbrace w \rbrace}  \delta_v \, C
			\end{displaymath}
			for each $y \in \Delta_D$. Therefore, each function $\vartheta_{\overline{\mathcal{D}}_w}$ is bounded.
		\end{enumerate}
		By definition, the weights $\delta_v$ are positive numbers bounded by $1$. Then, define
		\begin{displaymath}
			\beta \coloneqq C \cdot \vert R \cup \mathfrak{M}(K)_{\infty} \vert.
		\end{displaymath}
		The inequalities (1) and (2) show that $\vartheta_{\overline{\mathcal{D}}_v}$ is bounded by $\beta$, completing the proof in the nef case.
		
		Now, we prove the results when $\overline{\mathcal{D}}$ is generically big. For each $y \in \textup{ri}(\Delta_D)$, there exists $\varepsilon >0$ such that $y \in \Delta_{D, \varepsilon}$. Indeed, any $\varepsilon$ smaller than $\textup{dist}(y, \textup{rb}(\Delta_D ))$ suffices. By Lemma~\ref{finite-sum}, for each $n \geq n_{\varepsilon}$ and $v \in \mathfrak{M}(K)_{\textup{fin}} \setminus R(\varepsilon)$, we have $\vartheta_{\overline{\mathcal{D}}_{n,v}}(y) =0$. Then, the global roof function $\vartheta_{\overline{\mathcal{D}}_n}$ evaluated at $y$ is given by a sum of at most $|R(\varepsilon) \cup \mathfrak{M}(K)_{\infty}|$ non-zero terms. Concretely,
		\begin{displaymath}
			\vartheta_{\overline{\mathcal{D}}_n}(y) = \sum_{v \in \mathfrak{M}(K)} \delta_v \cdot \vartheta_{\overline{\mathcal{D}}_{n,v}}(y) = \sum_{v \in R(\varepsilon) \cup \mathfrak{M}(K)_{\infty}} \delta_v \cdot \vartheta_{\overline{\mathcal{D}}_{n,v}}(y).
		\end{displaymath}
		On the other hand, for each $v \in \mathfrak{M}(K)$, we have
		\begin{displaymath}
			\vartheta_{\overline{\mathcal{D}}_{v}}(y) = \lim_{n \in \mathbb{N}} \vartheta_{\overline{\mathcal{D}}_{n,v}}(y)
		\end{displaymath}
		In particular, $\vartheta_{\overline{\mathcal{D}}_{v}}(y) = 0$ for all $v \in 
		\mathfrak{M}(K)_{\textup{fin}} \setminus R(\varepsilon)$. We conclude that
		\begin{align*}
			\vartheta_{\overline{\mathcal{D}}}(y) & \coloneqq \lim_{n \in \mathbb{N}}  \vartheta_{\overline{\mathcal{D}}_n}(y) = \lim_{n \in \mathbb{N}}  \sum_{v \in R(\varepsilon) \cup \mathfrak{M}(K)_{\infty} } \delta_v \cdot \vartheta_{\overline{\mathcal{D}}_{n,v}}(y) = \sum_{v \in R(\varepsilon)\cup \mathfrak{M}(K)_{\infty} } \delta_v \cdot \vartheta_{\overline{\mathcal{D}}_{v}}(y) \\
			& = \sum_{v \in \mathfrak{M}(K)}  \delta_v \cdot \vartheta_{\overline{\mathcal{D}}_{v}}(y).
		\end{align*}
	\end{proof}
	We recall the following well-known lemma on interchanging limits and infinite sums.
	\begin{lem}\label{interchanging-limits-1}
		Let $f \colon \mathbb{N}^2 \rightarrow \mathbb{R}$ be a function decreasing on the second variable, i.e.,
		\begin{displaymath}
			f(m,n_1) \geq f(m,n_2)
		\end{displaymath}
		for each $m \in \mathbb{N}$ and $n_1 \leq n_2$. Suppose that $\sum_{m \in \mathbb{N}} f(m,1)$ is finite. Then,
		\begin{displaymath}
			\lim_{n \in \mathbb{N}} \sum_{m \in \mathbb{N}} f(m,n) = \sum_{m \in \mathbb{N}} \lim_{n \in \mathbb{N}} f(m,n).
		\end{displaymath}
	\end{lem}
	\begin{proof}
		This is a well-known result implied by the monotone convergence theorem in measure theory (or Fatou's lemma). Note that the above limit is not necessarily finite.
	\end{proof}
	\begin{rem}
		With the notation of Remark~\ref{notation-decreasing-seq}, the previous lemma shows that if one of the following limits
		\begin{displaymath}
			\lim_{n \in \mathbb{N}}  \vartheta_{\overline{\mathcal{D}}_n}(y) \quad \textup{or} \quad \sum_{v \in \mathfrak{M}(K)}  \delta_v \cdot \vartheta_{\overline{\mathcal{D}}_{v}}(y)
		\end{displaymath}
		exists, then the other exists, and they are equal. This motivates the following definition.
	\end{rem}
	\begin{dfn}
		With the notation of Remark~\ref{notation-decreasing-seq}, the function $\vartheta_{\overline{\mathcal{D}}} \colon \Delta_{D} \rightarrow \mathbb{R}_{-\infty}$ given by
		\begin{displaymath}
			\vartheta_{\overline{\mathcal{D}}} = \sum_{v \in \mathfrak{M}(K)}  \delta_v \cdot \vartheta_{\overline{\mathcal{D}}_{v}}
		\end{displaymath}
		is called the \textit{global roof function} of the semipositive toric adelic divisor $\overline{\mathcal{D}}$, whenever it exists. That is, it is not the constant function $-\infty$.
	\end{dfn}
	Now, we state and prove the main theorem of this subsection, which characterizes nefness in terms of the global roof function.
	\begin{thm}\label{nef-torus-OKS}
		Let $S \subset \mathfrak{M}(K)_{\textup{fin}}$ be a finite set and $\mathcal{U}_{S}$ be the $d$-dimensional torus over $\textup{Spec}(\mathcal{O}_{K,S})$. Then, a semipositive toric adelic divisor $\overline{\mathcal{D}}$ on $\mathcal{U}_{S}/\mathcal{O}_{K}$ is nef if and only if its global roof function $\vartheta_{\overline{\mathcal{D}}} \colon \Delta_{D} \rightarrow \mathbb{R}_{-\infty}$ exists and $\vartheta_{\overline{\mathcal{D}}}(y) \geq 0$ for all $y \in \Delta_D$.
	\end{thm}
	\begin{proof}
		By Proposition~\ref{global-roof-OKS}, if $\overline{\mathcal{D}}$ is nef, then its roof function exists and is positive. We only need to show the converse. By definition, there exists a decreasing sequence $\lbrace \overline{\mathcal{D}}_n \rbrace_{n \in \mathbb{N}}$ of semipositive toric model arithmetic divisors converging to $\overline{\mathcal{D}}$. We will use the notation in Remark~\ref{notation-decreasing-seq} freely. Let $\lbrace \varepsilon_n \rbrace_{n \in \mathbb{N}}$ be a decreasing sequence of positive rational numbers converging to $0$. Denote by $\overline{\mathcal{B}}_{\textup{can}}$ the canonical arithmetic model of $B$ (see Example~\ref{global-can-mod}). It is immediate that $\| \overline{\mathcal{B}}_{\textup{can}} \|_{\overline{\mathcal{B}}} = 1$. For each $n \in \mathbb{N}$, define the semipositive toric model arithmetic divisor 
		\begin{displaymath}
			\overline{\mathcal{E}}_{n} \coloneqq \overline{\mathcal{D}}_n + \varepsilon_n \cdot \overline{\mathcal{B}}_{\textup{can}}.
		\end{displaymath}
		For each place $v$, the corresponding roof function $\vartheta_{\overline{\mathcal{E}}_{n,v}}$ is a continous concave function on $\Delta_{E_n} = \Delta_{D_n} + \varepsilon_n \cdot \Delta_{B}$. If the place $v$ is finite, then $\vartheta_{\overline{\mathcal{E}}_{n,v}}$ is rational piecewise affine. Moreover,
		\begin{displaymath}
			\vartheta_{\overline{\mathcal{E}}_{n,v}} = \vartheta_{\overline{\mathcal{D}}_{n,v}} \boxplus \iota_{\varepsilon_n \cdot \Delta_B},
		\end{displaymath}
		where $\iota_{\varepsilon_n \cdot \Delta_B}$ is the indicator function of $\varepsilon_n \cdot \Delta_{B}$. The inequality
		\begin{displaymath}
			\overline{\mathcal{E}}_n = \overline{\mathcal{D}}_n + \varepsilon_n \cdot \overline{\mathcal{B}}_{\textup{can}} \geq   \overline{\mathcal{D}}+ \varepsilon_n \cdot \overline{\mathcal{B}}_{\textup{can}}
		\end{displaymath}
		implies the following inequality of $v$-adic roof functions
		\begin{displaymath}
			\vartheta_{\overline{\mathcal{D}}_{n,v}} \boxplus \iota_{\varepsilon_n \cdot \Delta_B} \geq \vartheta_{\overline{\mathcal{D}}_v} \boxplus \iota_{\varepsilon_n \cdot \Delta_B}.
		\end{displaymath}
		summing over all the places, we obtain
		\begin{displaymath}
			\vartheta_{\overline{\mathcal{E}}_n} (y)= \sum_{v \in \mathfrak{M}(K)} \delta_v \cdot (\vartheta_{\overline{\mathcal{D}}_{n,v}} \boxplus \iota_{\varepsilon_n \cdot \Delta_B}) (y) \geq \sum_{v \in \mathfrak{M}(K)} \delta_v \cdot ( \vartheta_{\overline{\mathcal{D}}_v} \boxplus \iota_{\varepsilon_n \cdot \Delta_B} ) (y) \geq 0
		\end{displaymath}
		for all $y \in \Delta_{D} + \varepsilon_n \cdot \Delta_{B}$. The positivity on the right-hand side is implied by the hypothesis $\vartheta_{\overline{\mathcal{D}}} \geq 0$. We have no control on the values of $\vartheta_{\overline{\mathcal{E}}_n}$ at the set $\Delta_{E_n} \setminus (\Delta_{D} + \varepsilon_n \cdot \Delta_{B})$. To fix this, for each $n \in \mathbb{N}$, we choose a rational convex polytope $\Delta_{F_n}$ satisfying
		\begin{displaymath}
			\Delta_D \subset \Delta_{F_n} \subset \Delta_{D} + \varepsilon_n \cdot \Delta_{B} \subset \Delta_{D_n} + \varepsilon_n \cdot \Delta_{B} =  \Delta_{E_n}.
		\end{displaymath}
		Since $\Delta_B$ has a non-empty interior, this is always possible. Without loss of generality, we may assume that $\lbrace \Delta_{F_n} \rbrace_{n \in \mathbb{N}}$ is a decreasing sequence. Then, for each $n \in \mathbb{N}$ and $v \in \mathfrak{M}(K)$, define the function
		\begin{displaymath}
			h_{n,v}(y) \coloneqq \begin{cases} \vartheta_{\overline{\mathcal{E}}_{n,v}}(y) & y \in \Delta_{F_n} \\ -\infty & \textup{otherwise}.\end{cases}
		\end{displaymath}
		By construction, $h_{n,v}$ is a continuous concave function on $\Delta_{F_n}$. If the place $v$ is finite, the function $h_{n,v}$ is a rational piecewise affine function, and it is equal to the indicator function of $\Delta_{F_n}$ for almost all places. Denote by $f_{n,v}$ the Legendre-Fenchel transform of $h_{n,v}$. By Description~\ref{toric-global-paf}, Corollary~\ref{div-nef-ample} and Theorem~\ref{toric-psh-type}, the $v$-tuple $f_n = (f_{n,v})$ determines a semipositive toric model arithmetic divisor $\overline{\mathcal{F}}_n$. By construction, for each $y \in \Delta_{F_n}$ we have
		\begin{displaymath}
			\vartheta_{\overline{\mathcal{F}}_n}(y) = \sum_{v \in \mathfrak{M}(K)} h_{n,v}(y) \geq 0.
		\end{displaymath}
		Therefore $\overline{\mathcal{F}}_n$ is nef. Again by construction, for each $n \in \mathbb{N}$ and $v \in \mathfrak{M}(K)$, we have an inequality
		\begin{displaymath}
			\vartheta_{\overline{\mathcal{E}}_{n,v}} \geq h_{n,v} \geq \vartheta_{\overline{\mathcal{D}}_v}.
		\end{displaymath}
		This translates into an inequality
		\begin{displaymath}
			\overline{\mathcal{E}}_n = \overline{\mathcal{D}}_n + \varepsilon_n \cdot \overline{\mathcal{B}}_{\textup{can}} \geq  \overline{\mathcal{F}}_n \geq \overline{\mathcal{D}}.
		\end{displaymath}
		The sequences $\lbrace \overline{\mathcal{D}}_n \rbrace_{n \in \mathbb{N}}$  $\lbrace \varepsilon_n \rbrace_{n \in \mathbb{N}}$  converge to $\overline{\mathcal{D}}$ and $0$, respectively. Therefore, the decreasing sequence of nef toric model arithmetic divisors $\lbrace \overline{\mathcal{F}}_n \rbrace_{n \in \mathbb{N}}$ converges in the boundary topology to the semipositive toric adelic divisor $\overline{\mathcal{D}}$. We conclude that  $\overline{\mathcal{D}}$ is nef. 
	\end{proof}
	\begin{cor}\label{integrable-bounded}
		Let $\overline{\mathcal{D}}$ be an integrable toric adelic divisor on $\mathcal{U}_{S}/\mathcal{O}_K$ with support function $\Psi_D$ and tropical Green's function $(\gamma_{\overline{\mathcal{D}}_v})$. Then, there exists a constant $C$ such that the functions $|\Psi_D - \gamma_{\overline{\mathcal{D}}_v}|$ are all bounded by $C$. Additionally, if $\overline{\mathcal{D}}$ is semipositive, then the $v$-adic roof functions $\vartheta_{\overline{\mathcal{D}}_v}$ are bounded by $C$.
	\end{cor}
	\begin{proof}
		Write $\overline{\mathcal{D}} = \overline{\mathcal{E}} - \overline{\mathcal{F}}$, where $\overline{\mathcal{E}}$ and $\overline{\mathcal{F}}$ are nef. By Proposition~\ref{global-roof-OKS}, there exists a constant $C$ such that the $v$-adic roof functions $\vartheta_{\overline{\mathcal{E}}_v}$ and $\vartheta_{\overline{\mathcal{F}}_v}$ are bounded by $C/2$. This implies that the functions $|\Psi_E - \gamma_{\overline{\mathcal{E}}_v}|$ and $|\Psi_F - \gamma_{\overline{\mathcal{F}}_v}|$ are bounded by $C/2$. The equation $\overline{\mathcal{D}} = \overline{\mathcal{E}} - \overline{\mathcal{F}}$ implies that $\Psi_D (x) = \Psi_E - \Psi_F$ and $\gamma_{\overline{\mathcal{D}}_v} =  \gamma_{\overline{\mathcal{E}}_v} -  \gamma_{\overline{\mathcal{F}}_v}$. Then, for each $x \in N_{\mathbb{R}}$, the triangle inequality gives
		\begin{align*}
			|\Psi_D (x) - \gamma_{\overline{\mathcal{D}}_v} (x)| & = |(\Psi_E (x) - \Psi_F (x)) - (\gamma_{\overline{\mathcal{E}}_v} (x) -\gamma_{\overline{\mathcal{F}}_v}(x) )| \\
			& \leq |\Psi_E (x) - \gamma_{\overline{\mathcal{E}}_v} (x) | + |\Psi_F (x)  - \gamma_{\overline{\mathcal{F}}_v} (x)| \leq C.
		\end{align*}
		Finally, if $\overline{\mathcal{D}}$ is semipositive, the functions $\gamma_{\overline{\mathcal{D}}_v}$ are concave. Since the functions $|\Psi_D - \gamma_{\overline{\mathcal{D}}_v}|$ are all bounded by $C$, Proposition~A.1.9~of~\cite{Per26} implies that the functions $\vartheta_{\overline{\mathcal{D}}_v}$ are bounded by $C$.
	\end{proof}
	
	\subsection{Approximation of semipositive divisors}\label{5-3} We finish this section with a series of approximation lemmas for semipositive toric adelic divisors. We will use these approximation results in the next section, where we give a formula for the arithmetic intersection numbers of semipositive toric adelic divisors.
	\begin{lem}\label{bounded-approx-1}
		Supppose that $\lbrace \overline{\mathcal{D}}_n \rbrace_{n \in \mathbb{N}}$ is a decreasing sequence of nef toric adelic divisors on $\mathcal{U}_{S}/\mathcal{O}_K$ converging to $\overline{\mathcal{D}}$ in the boundary topology. Then, there is a constant $\beta$ satisfying
		\begin{enumerate}[label=(\roman*)]
			\item For all $v \in \mathfrak{M}(K)$ and all $y \in \Delta_{D_n}$ we have $| \vartheta_{\overline{\mathcal{D}}_{n,v}}(y)| \leq \beta$.
			\item  For all $v \in \mathfrak{M}(K)$ and all $y \in \Delta_{D}$ we have $| \vartheta_{\overline{\mathcal{D}}_{n}}(y)| \leq \beta$.
		\end{enumerate}
	\end{lem}
	\begin{proof}
		Following the proof of part \textit{(iii)} of Proposition~\ref{global-roof-OKS}, the constant $\beta= \beta (\overline{\mathcal{D}})$ is given by
		\begin{displaymath}
			\beta (\overline{\mathcal{D}}) = C (\overline{\mathcal{D}}) \cdot \vert R  (\overline{\mathcal{D}}) \cup \mathfrak{M}(K)_{\infty} \vert,
		\end{displaymath}
		where $C  (\overline{\mathcal{D}})$ is any upper bound of the $v$-adic roof functions $\vartheta_{\overline{\mathcal{D}}_v}$ and $R (\overline{\mathcal{D}})$ is any set of finite places such that $S\subset R  (\overline{\mathcal{D}})$ and $\gamma_{\overline{\mathcal{D}}_v}(0) = 0$ for all finite places not in $R  (\overline{\mathcal{D}})$. In our situation, the sequences $\lbrace \vartheta_{\overline{\mathcal{D}}_{n,v}} \rbrace_{n \in \mathbb{N}}$ are decreasing and converging to $\vartheta_{\overline{\mathcal{D}}_v}$, and the functions $\gamma_{\overline{\mathcal{D}}_{n,v}}$ converge in the $\mathcal{C}$-norm to $\gamma_{\overline{\mathcal{D}}_v}$, uniformly in $v$ for all finite places $v$ not in $S$ (see Remark~\ref{global-bdry-top}). Then, choose $C$ to be an upper bound of the $v$-adic roof functions of $\overline{\mathcal{D}}_0$ and $R$ to be the complement of the set
		\begin{displaymath}
			\lbrace v \in \mathfrak{M}(K)_{\textup{fin}} \setminus  R  (\overline{\mathcal{D}})  \, \vert \, \| \gamma_{\overline{\mathcal{D}}_{v,n}} - \gamma_{\overline{\mathcal{D}}_v} \|_{\mathcal{C}} < \infty \rbrace.
		\end{displaymath}
		and define $\beta = C \cdot  \vert R \cup \mathfrak{M}(K)_{\infty} \vert$. Then, $C$ is an upper bound for all $\vartheta_{\overline{\mathcal{D}}_{n,v}}$ and $\vartheta_{\overline{\mathcal{D}}_{v}}$. The uniform convergence in the $\mathcal{C}$-norm implies that the set $R$ is finite, and the condition $\| \gamma_{\overline{\mathcal{D}}_{v,n}} - \gamma_{\overline{\mathcal{D}}_v} \|_{\mathcal{C}} < \infty$ implies $\gamma_{\overline{\mathcal{D}}_{v,n}}(0) = \gamma_{\overline{\mathcal{D}}_v} (0)$. Note that $v$ does not belong to $R(\overline{\mathcal{D}})$, hence $ \gamma_{\overline{\mathcal{D}}_v} (0) = 0$.
	\end{proof}
	\begin{lem}\label{bounded-approx-2}
		Suppose that $\overline{\mathcal{D}}$ is a semipositive toric adelic divisor on $\mathcal{U}_{S}/ \mathcal{O}_K$ whose global roof function $\gamma_{\overline{\mathcal{D}}}$ exists and is bounded below by a constant $C$. Then, there exists a decreasing sequence $\lbrace \overline{\mathcal{D}}_n \rbrace_{n \in \mathbb{N}}$ of semipositive toric model arithmetic divisors on $\mathcal{U}_{S}/\mathcal{O}_K$ converging to $\overline{\mathcal{D}}$ in the boundary topology and a constant $\beta$ satisfying:
		\begin{enumerate}[label=(\roman*)]
			\item For all $v \in \mathfrak{M}(K)$, all $n \in \mathbb{N}$, and all $y \in \Delta_{D_n}$ we have $| \vartheta_{\overline{\mathcal{D}}_{n,v}}(y)| \leq \beta$.
			\item  For all $v \in \mathfrak{M}(K)$, and all $y \in \Delta_{D}$ we have $| \vartheta_{\overline{\mathcal{D}}_{v}}(y)| \leq \beta$.
			\item For all $n \in \mathbb{N}$, the global roof function $\gamma_{\overline{\mathcal{D}}_n}$ exists and is bounded below by $C$.
		\end{enumerate}
	\end{lem}
	\begin{proof}
		Choose any infinite place $w \in \mathfrak{M}(K)_{\infty}$ and consider the semipositive toric adelic divisor $\overline{\mathcal{E}}$ determined by the functions
		\begin{displaymath}
			\gamma_{\overline{\mathcal{E}}_v} \coloneqq \begin{cases} \gamma_{\overline{\mathcal{D}}_{v}} + C & \textup{if } v=w \\ \gamma_{\overline{\mathcal{D}}_{v}}  & \textup{otherwise}. \end{cases}
		\end{displaymath}
		Taking Legendre-Fenchel transforms, we obtain
		\begin{displaymath}
			\vartheta_{\overline{\mathcal{E}}_v} = \begin{cases} \vartheta_{\overline{\mathcal{D}}_{v}} - C & \textup{if } v=w \\ \vartheta_{\overline{\mathcal{D}}_{v}}  & \textup{otherwise}. \end{cases}
		\end{displaymath}
		Then, $\vartheta_{\overline{\mathcal{E}}}(y) \geq 0$ for all $y \in \Delta_{E}$. By Theorem~\ref{nef-torus-OKS}, the toric adelic divisor $\overline{\mathcal{E}}$ is nef. By definition, there exists a decreasing sequence $\lbrace \overline{\mathcal{E}}_n \rbrace_{n \in \mathbb{N}}$ of nef toric model arithmetic divisors converging to $\overline{\mathcal{E}}$. Then, apply Lemma~\ref{bounded-approx-1} to find a constant $\beta^{\prime}$ satisfying properties \textit{(i)--(iii)} for $\overline{\mathcal{E}}$. Consider sequence $\lbrace \overline{\mathcal{D}}_n \rbrace_{n \in \mathbb{N}}$ determined by the functions
		\begin{displaymath}
			\gamma_{\overline{\mathcal{D}}_{n,v}} \coloneqq \begin{cases} \gamma_{\overline{\mathcal{E}}_{n,v}} - C & \textup{if } v=w \\ \gamma_{\overline{\mathcal{D}}_{v}}  & \textup{otherwise}. \end{cases}
		\end{displaymath}
		Then, properties \textit{(i)--(iii)} are satisfied for $\beta \coloneqq \beta^{\prime} + |C|$.
	\end{proof}
	\begin{lem}\label{bounded-approx-3}
		Let $\overline{\mathcal{D}}$ be a semipositive toric adelic divisor on $\mathcal{U}_{S}/ \mathcal{O}_K$ whose global roof function $\vartheta_{\overline{\mathcal{D}}} \colon \Delta_D \rightarrow \mathbb{R}_{-\infty}$ exists. Then, there exists a decreasing sequence $\lbrace \overline{\mathcal{D}}_n \rbrace_{n \in \mathbb{N}}$ of semipositive toric adelic divisors on $\mathcal{U}_{S}/\mathcal{O}_K$ such that:
		\begin{enumerate}[label=(\roman*)]
			\item The sequence $\lbrace \overline{\mathcal{D}}_n \rbrace_{n \in \mathbb{N}}$ converges in the boundary topology to $\overline{\mathcal{D}}$.
			\item For each $n \in \mathbb{N}$, the toric compactified divisor $D_n$ obtained by restriction of $\overline{\mathcal{D}}_n$ to the generic fiber is equal to $D$.
			\item For each $n \in \mathbb{N}$, the global roof function $\vartheta_{\overline{\mathcal{D}}_n} \colon \Delta_D \rightarrow \mathbb{R}_{-\infty}$ exists and is bounded below.
		\end{enumerate}
	\end{lem}
	\begin{proof}
		Let $\lbrace \varepsilon_n \rbrace_{n \in \mathbb{N}}$ be a decreasing sequence of positive rational numbers converging to $0$. Denote by $\overline{\mathcal{B}}$ the boundary divisor in Example~\ref{bdry-div-torus-OKS}. For each $n \in \mathbb{N}$, define the semipositive toric model arithmetic divisor $\overline{\mathcal{E}}_{n} \coloneqq \overline{\mathcal{D}} + \varepsilon_n \cdot \overline{\mathcal{B}}$. For each place $v$, the $v$-adic roof function $\vartheta_{\overline{\mathcal{E}}_{n,v}}$ satisfies
		\begin{displaymath}
			\vartheta_{\overline{\mathcal{E}}_{n,v}} = \vartheta_{\overline{\mathcal{D}}_v} \boxplus \vartheta_{\varepsilon_n \cdot \overline{\mathcal{B}}_v} = \begin{cases} \vartheta_{\overline{\mathcal{D}}_{v}} \boxplus \iota_{\varepsilon_n \cdot \Delta_B} + \varepsilon_n & v \in S \cup \mathfrak{M}(K)_{\infty} \\ \vartheta_{\overline{\mathcal{D}}_{v}} \boxplus \iota_{\varepsilon_n \cdot \Delta_B}  & \textup{otherwise} \end{cases}
		\end{displaymath}
		where $\iota_{\varepsilon_n \cdot \Delta_B}$ is the indicator function of the convex polytope $\varepsilon_n \cdot \Delta_{B}$. Moreover,
		\begin{displaymath}
			\vartheta_{\overline{\mathcal{E}}_n}(y) = \sum_{v \in \mathfrak{M}(K)} \delta_v \cdot  (\vartheta_{\overline{\mathcal{D}}_{v}} \boxplus \iota_{\varepsilon_n \cdot \Delta_B} )(y) + \sum_{v \in  S \cup \mathfrak{M}(K)_{\infty}} \delta_v \, \varepsilon_n .
		\end{displaymath}
		Then, $\vartheta_{\overline{\mathcal{E}}_n} \colon \Delta_{E_n} \rightarrow \mathbb{R}_{-\infty}$ is a closed concave function, where $\Delta_{E_n} = \Delta_D + \varepsilon_n \cdot \Delta_B$. Since $\Delta_B$ has a non-empty interior, the functions
		\begin{displaymath}
			h_{n,v}(y) \coloneqq \begin{cases} \vartheta_{\overline{\mathcal{E}}_{n,v}}(y) & y \in \Delta_D \\ -\infty & \textup{otherwise}\end{cases} \quad \textup{and} \quad h_{n}(y) \coloneqq \begin{cases} \vartheta_{\overline{\mathcal{E}}_{n}}(y) & y \in \Delta_D \\ -\infty & \textup{otherwise} \end{cases}
		\end{displaymath}
		are closed concave functions on $M_{\mathbb{R}}$, continuous on $\Delta_D$; in particular, bounded. Observe that
		\begin{displaymath}
			\Delta_D \subset \Delta_{E_n,\varepsilon_n} = \lbrace y \in \Delta_{E_n} \, \vert \, \textup{dist}(y, \textup{rb}(\Delta_{E_n}) \geq \varepsilon_n \rbrace.
		\end{displaymath}
		Following the proof of Lemma~\ref{finite-sum}, we can find a finite set $R(\varepsilon_n ) \subset \mathfrak{M}(K)_{\textup{fin}}$ containing $S$ such that $\vartheta_{\overline{\mathcal{E}}_{n,v}}|_{\Delta_{E_n, \varepsilon_n}} = 0$ for all $v \in \mathfrak{M}(K)_{\textup{fin}} \setminus R(\varepsilon_n)$. This implies that $h_{n,v} = \iota_{\Delta_D}$. By Description~\ref{toric-global-paf}, Corollary~\ref{div-nef-ample} and Theorem~\ref{toric-psh-type}, there is a semipositive toric adelic divisor $\overline{\mathcal{D}}_{n}$ satisfying $\vartheta_{\overline{\mathcal{D}}_{n,v}} = h_{n,v}$. Then, the sequence $\lbrace \overline{\mathcal{D}}_n \rbrace_{n \in \mathbb{N}}$ is decreasing and satisfies
		\begin{displaymath}
			\vartheta_{\overline{\mathcal{E}}_{n,v}} \geq \vartheta_{\overline{\mathcal{D}}_{n,v}} \geq \vartheta_{\overline{\mathcal{D}}_v}.
		\end{displaymath}
		In terms of toric adelic divisors, this means
		\begin{displaymath}
			\overline{\mathcal{E}}_n = \overline{\mathcal{D}} + \varepsilon_n \cdot \overline{\mathcal{B}} \geq  \overline{\mathcal{D}}_n \geq \overline{\mathcal{D}}.
		\end{displaymath}
		Therefore, the sequence $\lbrace \overline{\mathcal{D}}_n \rbrace_{n \in \mathbb{N}}$ satisfies the conditions of the lemma.
	\end{proof}
	
	\section{Arithmetic intersection numbers}\label{6}
	This final section defines the arithmetic intersection numbers of semipositive toric adelic divisors. Moreover, we give an integral formula for a large class of semipositive toric adelic divisors.
	
	\subsection{The integral formula} Let $S \subset \mathfrak{M}(K)_{\textup{fin}}$ be a finite set and $\mathcal{U}_{S}$ be the $d$-dimensional torus over $\textup{Spec}(\mathcal{O}_{K,S})$. As in the geometric and local arithmetic cases, the projection formula induces an arithmetic intersection pairing on the vector space of toric model arithmetic divisors
	\begin{displaymath}
		\overline{\textup{Div}}_{\mathbb{T}}(\mathcal{U}_{S}/\mathcal{O}_K )_{\textup{mod}}^{d+1} \colon \longrightarrow \mathbb{R}.
	\end{displaymath}
	In particular, if $\overline{\mathcal{D}}$ is a semipositive toric model arithmetic divisor, Theorem~\ref{global-toric-h-proj} gives the following formula for the arithmetic intersection number
	\begin{displaymath}
		\overline{\mathcal{D}}^{d+1} = (d+1)! \int_{\Delta_D} \vartheta_{\overline{\mathcal{D}}} \, \textup{dVol}_M = (d+1)! \sum_{v \in \mathfrak{M}(K)} \delta_v \cdot \int_{\Delta_D} \vartheta_{\overline{\mathcal{D}}_v} \, \textup{dVol}_M,
	\end{displaymath}
	where $\vartheta_{\overline{\mathcal{D}}}$ and $\vartheta_{\overline{\mathcal{D}}_v}$ are the global and $v$-adic roof functions corresponding to $\overline{\mathcal{D}}$, respectively. We will extend this result to the arithmetic intersection numbers defined below. This definition is a global version of Definition~4.3.7~of~\cite{Per26}, and it is inspired by the generalized arithmetic intersection numbers of Burgos and Kramer, discussed in Remark~\ref{finale}.
	\begin{dfn}\label{global-intersection-numbers}
		Let $S \subset \mathfrak{M}(K)_{\textup{fin}}$ be a finite set and $\mathcal{U}_{S}$ be the $d$-dimensional torus over $\textup{Spec}(\mathcal{O}_{K,S})$. Given semipositive toric adelic divisors $\overline{\mathcal{D}}_0, \ldots, \overline{\mathcal{D}}_d$ on $\mathcal{U}_{S}/\mathcal{O}_K$, the \textit{arithmetic intersection number} $\overline{\mathcal{D}}_0 \cdot \ldots \cdot \overline{\mathcal{D}}_d$ is defined as follows:
		\begin{enumerate}[label=(\roman*)]
			\item Suppose that, for each $0 \leq i \leq d$, the global roof function $\vartheta_{\overline{\mathcal{D}}_i}$ exists and is bounded. The number $\overline{\mathcal{D}}_0 \cdot \ldots \cdot \overline{\mathcal{D}}_d$ is given by the limit
			\begin{displaymath}
				\overline{\mathcal{D}}_0 \cdot \ldots \cdot \overline{\mathcal{D}}_d\coloneqq \lim_{n \rightarrow \infty} \overline{\mathcal{D}}_{0,n} \cdot \ldots \cdot \overline{\mathcal{D}}_{d,n}
			\end{displaymath}
			where $\lbrace \overline{\mathcal{D}}_{i,n} \rbrace_{n \in \mathbb{N}}$ is any decreasing sequence of semipositive toric model arithmetic divisors on $\mathcal{U}_{S}/\mathcal{O}_K$, converging in the boundary topology to $\overline{\mathcal{D}}_i$, and satisfying conditions \textit{(i)--(iii)} in Lemma~\ref{bounded-approx-2}.
			\item For each $0 \leq i \leq d$, let $D_i$ be the toric compactified divisor induced by restriction of $\overline{\mathcal{D}}_i$ to the generic fiber $U/K$. Then, the number $\overline{\mathcal{D}}_0 \cdot \ldots \cdot \overline{\mathcal{D}}_d$ is given by 
			\begin{displaymath}
				\overline{\mathcal{D}}_0 \cdot \ldots \cdot \overline{\mathcal{D}}_d\coloneqq \textup{inf} \lbrace\overline{\mathcal{E}}_0 \cdot \ldots \cdot \overline{\mathcal{E}}_d \, \vert \, (\overline{\mathcal{E}}_0 , \ldots, \overline{\mathcal{E}}_d) \textup{ satisfying } (\star) \rbrace,
			\end{displaymath}
			where $(\star )$ means that, for each $0 \leq i \leq d$, the toric adelic divisor $\overline{\mathcal{E}}_i \geq \overline{\mathcal{D}}_i$ is semipositive, its restriction $E_i$ to the generic fiber is equal to $D_i$, and its global roof function $\vartheta_{\overline{\mathcal{E}}_i}$ exists and is bounded.
		\end{enumerate}
		This definition is extended to differences of semipositive toric adelic divisors by linearity.
	\end{dfn}
	\begin{rem}
		The above definition can be stated for any toric arithmetic variety of the form $\mathcal{X}_{\Sigma_0 , S} / \mathcal{O}_{K}$. In any case, we can always reduce to the case of the torus. Indeed, by Lemma~\ref{GK-functorial}, we have a continuous injective group morphism
		\begin{displaymath}
			\iota^{\ast} \colon \overline{\textup{Div}}_{\mathbb{T}}(\mathcal{X}_{\Sigma_0,S}/\mathcal{O}_K) \longrightarrow \overline{\textup{Div}}_{\mathbb{T}}(\mathcal{U}_{S}/\mathcal{O}_K).
		\end{displaymath}
		This map restricts to the corresponding cones of semipositive toric adelic divisors (resp. nef toric adelic divisors). Then, the projection formula gives the identity
		\begin{displaymath}
			\iota^{\ast}\overline{\mathcal{D}}_0 \cdot \ldots \cdot \iota^{\ast}\overline{\mathcal{D}}_d = \overline{\mathcal{D}}_0 \cdot \ldots \cdot\overline{\mathcal{D}}_d
		\end{displaymath}
		for all semipositive toric model arithmetic divisors $\overline{\mathcal{D}}_0 ,\ldots , \overline{\mathcal{D}}_d$ of $\mathcal{X}_{\Sigma_{0},S}$. By the above definition and continuity of the pullback $\iota^{\ast}$, the above identity extends to the level of semipositive toric adelic divisors. Thus, we may always compute the arithmetic intersection number $ \overline{\mathcal{D}}_0 \cdot \ldots \cdot\overline{\mathcal{D}}_d$ by pulling back to the torus.
	\end{rem}
	Then, we show the main result of this article: an integral formula for the arithmetic self-intersection number of any semipositive toric adelic divisor $\overline{\mathcal{D}}$ whose global roof function $\vartheta_{\overline{\mathcal{D}}}$ exists. This is automatically satisfied if $\overline{\mathcal{D}}$ is generically big, i.e. the compact convex set $\Delta_D$ in $M_{\mathbb{R}}$ has a non-empty interior (see Proposition~\ref{global-roof-OKS}).
	\begin{thm}\label{global-toric-height-OKS}
		Let $S \subset \mathfrak{M}(K)_{\textup{fin}}$ be a finite set and $\mathcal{U}_{S}$ be the $d$-dimensional torus over $\textup{Spec}(\mathcal{O}_{K,S})$. Given a semipositive toric adelic divisor $\overline{\mathcal{D}}$ on $\mathcal{U}_{S}/\mathcal{O}_K$ whose global roof function $\vartheta_{\overline{\mathcal{D}}} \colon \Delta_D \rightarrow \mathbb{R}_{-\infty}$ exists, we have the following integral formula
		\begin{displaymath}
			\overline{\mathcal{D}}^{d+1} = (d+1)! \int_{\Delta_D} \vartheta_{\overline{\mathcal{D}}} \, \textup{dVol}_M = (d+1)! \sum_{v \in \mathfrak{M}(K)} \delta_v \cdot \int_{\Delta_D} \vartheta_{\overline{\mathcal{D}}_v} \, \textup{dVol}_M,
		\end{displaymath}
		where $\vartheta_{\overline{D_v}}$ is the $v$-adic roof function of $\overline{\mathcal{D}}$ at $v$, and $\delta_v$ is the weight in Definition~\ref{global-roof-fun}. The arithmetic self-intersection number $\overline{\mathcal{D}}^{d+1}$ is finite if and only if $\vartheta_{\overline{\mathcal{D}}}  \in L^1 (\Delta_D)$.
	\end{thm}
	\begin{proof}
		Here, we prove the formula
		\begin{displaymath}
			\overline{\mathcal{D}}^{d+1} = (d+1)! \int_{\Delta_D} \vartheta_{\overline{\mathcal{D}}} \, \textup{dVol}_M.
		\end{displaymath}
		Then, Lemma~\ref{interchanging-limits-2} gives the identity
		\begin{displaymath}
			(d+1)! \int_{\Delta_D} \vartheta_{\overline{\mathcal{D}}} \, \textup{dVol}_M = (d+1)! \sum_{v \in \mathfrak{M}(K)} \delta_v \cdot \int_{\Delta_D} \vartheta_{\overline{\mathcal{D}}_v} \, \textup{dVol}_M.
		\end{displaymath}
		First, we assume that $\vartheta_{\overline{\mathcal{D}}}$ is bounded. Let $\lbrace \overline{\mathcal{D}}_n \rbrace_{n \in \mathbb{N}}$ be a decreasing sequence of semipositive toric model arithmetic divisors as in Lemma~\ref{bounded-approx-2}. Then, Theorem~\ref{global-toric-h-proj} gives
		\begin{displaymath}
			\overline{\mathcal{D}}^{d+1} = \lim_{n \in \mathbb{N}} \overline{\mathcal{D}}_{n}^{d+1} = (d+1)! \lim_{n \in \mathbb{N}} \int_{\Delta_{D_n}} \vartheta_{\overline{\mathcal{D}}_n} \, \textup{dVol}_M.
		\end{displaymath}
		By Theorem~3.2.17~of~\cite{Per26}, the sequence $\lbrace \Delta_{D_n} \rbrace_{n \in \mathbb{N}}$ is decreasing and converges in the Hausdorff distance to the compact convex set $\Delta_D$. For each $n \in \mathbb{N}$, write
		\begin{displaymath}
			\int_{\Delta_{D_n}} \vartheta_{\overline{\mathcal{D}}_n} \, \textup{dVol}_M = \int_{\Delta_{D}} \vartheta_{\overline{\mathcal{D}}_n} \, \textup{dVol}_M + \int_{\Delta_{D_n} \setminus \Delta_D} \vartheta_{\overline{\mathcal{D}}_n} \, \textup{dVol}_M
		\end{displaymath}
		By the monotonic convergence theorem, we know
		\begin{displaymath}
			\lim_{n \in \mathbb{N}} \int_{\Delta_{D}} \vartheta_{\overline{\mathcal{D}}_n} \, \textup{dVol}_M = \int_{\Delta_D} \vartheta_{\overline{\mathcal{D}}} \, \textup{dVol}_M.
		\end{displaymath}
		On the other hand, Beer's~Theorem~\cite{Beer} implies that $\lim_{n \in \mathbb{N}} \textup{Vol}_M (\Delta_{D_n} \setminus \Delta_D) = 0$. This, together with property \textit{(iii)} of Lemma~\ref{bounded-approx-2} shows
		\begin{displaymath}
			\lim_{n \in \mathbb{N}} \int_{\Delta_{D_n} \setminus \Delta_D} \vartheta_{\overline{\mathcal{D}}_n} \, \textup{dVol}_M = 0.
		\end{displaymath}
		This identity, together with Lemma~\ref{interchanging-limits-2}, gives
		\begin{displaymath}
			\overline{\mathcal{D}}^{d+1} = (d+1)! \int_{\Delta_D} \vartheta_{\overline{\mathcal{D}}} \, \textup{dVol}_M = (d+1)! \sum_{v \in \mathfrak{M}(K)} \delta_v \cdot \int_{\Delta_D} \vartheta_{\overline{\mathcal{D}}_v} \, \textup{dVol}_M.
		\end{displaymath}
		Now, we remove the boundedness assumption on $\vartheta_{\overline{\mathcal{D}}}$. For each $0 \leq i \leq d$, let $\overline{\mathcal{E}}_i$ be a toric adelic divisor such that $\overline{\mathcal{E}}_i \geq \overline{\mathcal{D}}$ is semipositive, its restriction $E_i$ to the generic fiber is equal to $D$, and its global roof function $\vartheta_{\overline{\mathcal{E}}_i}$ exists and is bounded. Then, combining the identity
		\begin{displaymath}
			\overline{\mathcal{E}}_{i}^{d+1} = (d+1)! \int_{\Delta_D} \vartheta_{\overline{\mathcal{E}}_i} \, \textup{dVol}_M = (d+1)! \sum_{v \in \mathfrak{M}(K)} \delta_v \cdot \int_{\Delta_D} \vartheta_{\overline{\mathcal{E}}_{i,v}} \, \textup{dVol}_M
		\end{displaymath}
		with the argument in the proof of Corollary~4.3.10~of~\cite{Per26}, we obtain the formula
		\begin{displaymath}
			\overline{\mathcal{E}}_0 \cdot \ldots \cdot \overline{\mathcal{E}}_d = \sum_{v \in \mathfrak{M}(K)} \delta_v \cdot \textup{MI}_{M}(\vartheta_{\overline{\mathcal{E}}_{0,v}}, \ldots ,\vartheta_{\overline{\mathcal{E}}_{d,v}}),
		\end{displaymath}
		where $\vartheta_{\overline{\mathcal{E}}_{i,v}}$ is the $v$-adic roof function of $\overline{\mathcal{E}}_i$ at the place $v$ and $\textup{MI}_M$ is the mixed integral, introduced in Definition~\ref{mixed-integral}. Since $\vartheta_{\overline{\mathcal{E}}_{i,v}}$ is bounded on its effective domain $\Delta_D$ and $\vartheta_{\overline{\mathcal{E}}_{i,v}} \geq \vartheta_{\overline{\mathcal{D}}_v}$, we obtain the inequality
		\begin{displaymath}
			\textup{MI}_{M}(\vartheta_{\overline{\mathcal{E}}_{0,v}}, \ldots ,\vartheta_{\overline{\mathcal{E}}_{d,v}}) \geq \textup{MI}_{M}(\vartheta_{\overline{\mathcal{D}}_{v}}, \ldots ,\vartheta_{\overline{\mathcal{D}}_{v}}) = (d+1)! \int_{\Delta_D} \vartheta_{\overline{\mathcal{D}}_{v}} \, \textup{dVol}_M.
		\end{displaymath}
		Taking the weighted sum over all the places, we get
		\begin{displaymath}
			\sum_{v \in \mathfrak{M}(K)} \delta_v \cdot \textup{MI}_{M}(\vartheta_{\overline{\mathcal{E}}_{0,v}}, \ldots ,\vartheta_{\overline{\mathcal{E}}_{d,v}}) \geq \sum_{v \in \mathfrak{M}(K)} \delta_v \cdot \int_{\Delta_D} \vartheta_{\overline{\mathcal{D}}_{v}} \, \textup{dVol}_M.
		\end{displaymath}
		Then, Lemma~\ref{interchanging-limits-2} gives
		\begin{displaymath}
			\overline{\mathcal{E}}_0 \cdot \ldots \cdot \overline{\mathcal{E}}_d \geq (d+1)! \int_{\Delta_D} \vartheta_{\overline{\mathcal{D}}} \, \textup{dVol}_M.
		\end{displaymath}
		After taking the infimum over all $(d+1)$-tuples $(\overline{\mathcal{E}}_0 , \ldots, \overline{\mathcal{E}}_d)$, we get the bound
		\begin{displaymath}
			\overline{\mathcal{D}}^{d+1} \geq (d+1)! \int_{\Delta_D} \vartheta_{\overline{\mathcal{D}}} \, \textup{dVol}_M.
		\end{displaymath}
		Finally, let $\lbrace \overline{\mathcal{D}}_n \rbrace_{n \in \mathbb{N}}$ be a sequence satisfying the conditions of Lemma~\ref{bounded-approx-3}. By definition,
		\begin{displaymath}
			\lim_{n \in \mathbb{N}} \overline{\mathcal{D}}_{n}^{d+1} \geq \overline{\mathcal{D}}^{d+1}. 
		\end{displaymath}
		On the other hand, the decreasing sequence $\lbrace \vartheta_{\overline{\mathcal{D}}_n} \rbrace_{n \in \mathbb{N}}$ converges pointwise to $\vartheta_{\overline{\mathcal{D}}}$. By hypothesis, each $ \vartheta_{\overline{\mathcal{D}}_n}$ is bounded below. Therefore, the monotone convergence theorem gives
		\begin{displaymath}
			\lim_{n \in \mathbb{N}} \overline{\mathcal{D}}_{n}^{d+1} = \lim_{n \in \mathbb{N}} (d+1)! \int_{\Delta_D} \vartheta_{\overline{\mathcal{D}}_n} \, \textup{dVol}_M =  (d+1)! \int_{\Delta_D} \vartheta_{\overline{\mathcal{D}}} \, \textup{dVol}_M.
		\end{displaymath}
		We conclude that
		\begin{displaymath}
			\overline{\mathcal{D}}^{d+1} =  (d+1)! \int_{\Delta_D} \vartheta_{\overline{\mathcal{D}}} \, \textup{dVol}_M.
		\end{displaymath}
		The next lemma finishes the proof of this result.
	\end{proof}
	\begin{lem}\label{interchanging-limits-2}
		Given a semipositive toric adelic divisor $\overline{\mathcal{D}}$ on $\mathcal{U}_{S}/\mathcal{O}_K$ whose global roof function $\vartheta_{\overline{\mathcal{D}}} \colon \Delta_D \rightarrow \mathbb{R}_{-\infty}$ exists, we have
		\begin{displaymath}
			\int_{\Delta_D} \vartheta_{\overline{\mathcal{D}}} \, \textup{dVol}_M = \sum_{v \in \mathfrak{M}(K)} \delta_v \cdot \int_{\Delta_D} \vartheta_{\overline{\mathcal{D}}_v} \, \textup{dVol}_M,
		\end{displaymath}
		where $\vartheta_{\overline{\mathcal{D}}_v}$ is the $v$-adic roof function of $\overline{\mathcal{D}}$ at the place $v$. The above integral is not necessarily finite.
	\end{lem}
	\begin{proof}
		Let $\lbrace \varepsilon_n \rbrace_{n \in \mathbb{N}}$ be a decreasing sequence of rational numbers converging to $0$. Denote by $\overline{\mathcal{B}}$ the boundary divisor in Example~\ref{bdry-div-torus-OKS} and by $\vartheta_{\overline{\mathcal{B}}_{v}} \colon \Delta_B \rightarrow \mathbb{R}$ the $v$-adic roof function of $\overline{\mathcal{B}}$ at the place $v$. Then, for each $n \in \mathbb{N}$ and $v \in \mathfrak{M}(K)$, let $\vartheta_{n,v} \colon \Delta_D \rightarrow \mathbb{R}$ be the concave function given by
		\begin{displaymath}
			\vartheta_{n,v} = (\vartheta_{\overline{\mathcal{D}}_v} \boxplus \vartheta_{\varepsilon_n \cdot \overline{\mathcal{B}}_{v}})|_{\Delta_D}.
		\end{displaymath}
		Arguing as in Remark~\ref{bounded-approx-3}, we have the following properties:
		\begin{enumerate}[label=(\roman*)]
			\item For each $n \in \mathbb{N}$ and $v \in \mathfrak{M}(K)$, the function $\vartheta_{n,v}$ is continuous on $\Delta_D$, hence bounded. Therefore, the integral
			\begin{displaymath}
				\int_{\Delta_D} \vartheta_{n,v} \, \textup{dVol}_M 
			\end{displaymath}
			is finite.
			\item For each $v \in \mathfrak{M}(K)$, the sequence $\lbrace \vartheta_{n,v} \rbrace_{n \in \mathbb{N}}$ is decreasing and converges pointwise to $\vartheta_{\overline{\mathcal{D}}_v}$. By the monotone convergence theorem,
			\begin{displaymath}
				\lim_{n \in \mathbb{N}} \int_{\Delta_D} \vartheta_{n,v} \, \textup{dVol}_M  = \int_{\Delta_D} \vartheta_{\overline{\mathcal{D}}_v} \, \textup{dVol}_M.
			\end{displaymath}
			\item For each $n \in \mathbb{N}$, there exist a finite set $R(\varepsilon_n) \subset \mathfrak{M}(K)_{\textup{fin}}$ containing $S$ and such that $\vartheta_{n,v} = \iota_{\Delta_D}$ for all $v \in  \mathfrak{M}(K)_{\textup{fin}} \setminus R(\varepsilon_n)$. Therefore, the sum
			\begin{displaymath}
				\vartheta_n = \sum_{v \in \mathfrak{M}(K)}  \delta_v \cdot \vartheta_{n,v}
			\end{displaymath}
			defines a continuous concave function on $\Delta_D$. Moreover, by Lemma~\ref{interchanging-limits-1}, the sequence $\lbrace \vartheta_n \rbrace_{n \in \mathbb{N}}$ is decreasing and converges pointwise to $\vartheta_{\overline{\mathcal{D}}}$.
		\end{enumerate}
		By properties (i) and (iii), we have
		\begin{displaymath}
			\sum_{v \in \mathfrak{M}(K)} \delta_v \cdot \int_{\Delta_D} \vartheta_{n,v} \, \textup{dVol}_M =  \int_{\Delta_D} \left( \sum_{v \in \mathfrak{M}(K)}  \delta_v \cdot \vartheta_{n,v} \right) \, \textup{dVol}_M = \int_{\Delta_D} \vartheta_n \, \textup{dVol}_M
		\end{displaymath}
		Indeed, each sum is non-zero for at most a finite number of terms. Now, by property (iii) and the monotone convergence theorem, we get
		\begin{displaymath}
			\lim_{n \in \mathbb{N}}  \int_{\Delta_D} \vartheta_n \, \textup{dVol}_M =  \int_{\Delta_D} \left(\lim_{n \in \mathbb{N}} \vartheta_n \right) \, \textup{dVol}_M = \int_{\Delta_D} \vartheta_{\overline{\mathcal{D}}} \, \textup{dVol}_M.
		\end{displaymath}
		We have shown
		\begin{displaymath}
			\lim_{n \in \mathbb{N}} \sum_{v \in \mathfrak{M}(K)} \delta_v \cdot \int_{\Delta_D} \vartheta_{n,v} \, \textup{dVol}_M = \int_{\Delta_D} \vartheta_{\overline{\mathcal{D}}} \, \textup{dVol}_M.
		\end{displaymath}
		On the other hand, define the function $h \colon \mathfrak{M}(K) \times \mathbb{N} \rightarrow \mathbb{R}$ given by
		\begin{displaymath}
			h(v,n) \coloneqq \delta_v \cdot \int_{\Delta_{D}} \vartheta_{n,v} \, \textup{dVol}_M.
		\end{displaymath}
		Monotonicity of $\lbrace \vartheta_{n,v} \rbrace_{n \in \mathbb{N}}$ shows that the function $h(v,n)$ is decreasing on the second variable. Thus, we apply Lemma~\ref{interchanging-limits-1} and property (ii) to obtain
		\begin{align*}
			\lim_{n \in \mathbb{N}} \sum_{v \in \mathfrak{M}(K)} \delta_v \cdot \int_{\Delta_{D}} \vartheta_{n,v} \, \textup{dVol}_M &= \sum_{v \in \mathfrak{M}(K)} \lim_{n \in \mathbb{N}} \left( \delta_v \cdot \int_{\Delta_{D}} \vartheta_{n,v} \, \textup{dVol}_M \right) \\
			& =  \sum_{v \in \mathfrak{M}(K)} \delta_v \cdot \int_{\Delta_D} \vartheta_{\overline{\mathcal{D}}_v} \, \textup{dVol}_M.
		\end{align*}
		Therefore, we have the following identity
		\begin{displaymath}
			\int_{\Delta_D} \vartheta_{\overline{\mathcal{D}}} \, \textup{dVol}_M = \lim_{n \in \mathbb{N}} \sum_{v \in \mathfrak{M}(K)} \delta_v \cdot \int_{\Delta_{D}} \vartheta_{n,v} \, \textup{dVol}_M = \sum_{v \in \mathfrak{M}(K)} \delta_v \cdot \int_{\Delta_D} \vartheta_{\overline{\mathcal{D}}_v} \, \textup{dVol}_M.
		\end{displaymath}
	\end{proof}
	We state a mixed version of Theorem~\ref{global-toric-height-OKS}. We remind the reader that $\textup{MI}_M$ denotes the mixed integral, introduced in Definition~\ref{mixed-integral}.
	\begin{cor}\label{global-mixed}
		For each $0 \leq i \leq d$, let $\overline{\mathcal{E}}_i$ be a semipositive toric adelic divisor such that its global roof function $\vartheta_{i} \colon \Delta_{i} \rightarrow \mathbb{R}_{-\infty}$ exists. Then:
		\begin{enumerate}[label=(\roman*)]
			\item The following inequality holds
			\begin{displaymath}
				\overline{\mathcal{E}}_0 \cdot \ldots \cdot \overline{\mathcal{E}}_d \geq \sum_{v \in \mathfrak{M}(K)} \delta_v \cdot \textup{MI}_{M}(\vartheta_{0,v}, \ldots ,\vartheta_{d,v}),
			\end{displaymath}
			where $\vartheta_{i,v}$ is the $v$-adic roof function of $\overline{\mathcal{E}}_i$ at the place $v$.
			\item For each $I \subset \lbrace 0,1, \ldots, d \rbrace$, denote
			\begin{displaymath}
				\overline{\mathcal{E}}_I \coloneqq \sum_{i \in I} \overline{\mathcal{E}}_i, \quad \Delta_I \coloneqq \sum_{i \in I} \Delta_{E_i}, \quad \vartheta_{I,v} \coloneqq \boxplus_{i \in I} \, \vartheta_{i,v}, \quad \vartheta_{I} := \vartheta_{\overline{\mathcal{E}}_I} = \sum_{v \in \mathfrak{M}(K)} \delta_v \cdot \vartheta_{I,v}.
			\end{displaymath}
			Suppose that, for each $I$, we have $\vartheta_I \in L^{1}(\Delta_I)$. Then, we have the following formula
			\begin{displaymath}
				\overline{\mathcal{E}}_0 \cdot \ldots \cdot \overline{\mathcal{E}}_d = \sum_{I \subset \lbrace 0,1, \ldots, d \rbrace} (-1)^{d-|I|} \, \overline{\mathcal{E}}_{I}^{d+1}  = \sum_{v \in \mathfrak{M}(K)} \delta_v \cdot \textup{MI}_{M}(\vartheta_{0,v}, \ldots ,\vartheta_{d,v}) > - \infty.
			\end{displaymath}
		\end{enumerate}
	\end{cor}
	\begin{proof}
		The equality for the bounded case was shown in the proof of Theorem~\ref{global-toric-height-OKS}. Then, the bound in part \textit{(i)} follows immediately from the definition.
		
		For part \textit{(ii)}, note that $\vartheta_I \in L^{1}(\Delta_I)$ implies that $\vartheta_{I,v} \in  L^{1}(\Delta_I)$ for every place $v$. Indeed, let $v \in \mathfrak{M}(K)$. By Remark~\ref{local-roof-neg}, there exists a place $w \neq v$ such that $\vartheta_{I,w} \leq 0$. This implies $\vartheta_{I,v} (y) \geq \vartheta_I$, hence $\vartheta_{I,v} \in  L^{1}(\Delta_I)$. Then,  copy the proof of Corollary~4.3.14~of~\cite{Per26} to show the formula.
	\end{proof}
	
	\subsection{Comparison between arithmetic intersection pairings} We finish this article by relating our results with the intersection pairing of Yuan and Zhang~\cite{Y-Z} in the nef case (resp. integrable case) and its generalization by Burgos and Kramer~\cite{BK24}.
	\begin{rem}[Comparison with Yuan-Zhang's pairing]\label{almost-finale}
		Let $S \subset \mathfrak{M}(K)_{\textup{fin}}$ be a finite set, $\mathcal{U}_{S}$ be the $d$-dimensional torus over $\textup{Spec}(\mathcal{O}_{K,S})$ and $\overline{\mathcal{D}}$ be a nef toric adelic divisor on $\mathcal{U}_{S}/\mathcal{O}_K$. By Proposition~\ref{global-roof-OKS}, the global roof function $\vartheta_{\overline{\mathcal{D}}} \colon \Delta_D \rightarrow \mathbb{R}_{-\infty}$ exists and is bounded. Then, by Theorem~\ref{global-toric-height-OKS}, the following identity holds
		\begin{displaymath}
			\overline{\mathcal{D}}^{d+1} = (d+1)! \int_{\Delta_D} \vartheta_{\overline{\mathcal{D}}} \, \textup{dVol}_M = (d+1)! \sum_{v \in \mathfrak{M}(K)} \delta_v \cdot \int_{\Delta_D} \vartheta_{\overline{\mathcal{D}}_v} \, \textup{dVol}_M.
		\end{displaymath}
		Moreover, the formula for the arithmetic intersection numbers in Corollary~\ref{global-mixed} induces a symmetric multilinear pairing
		\begin{displaymath}
			\overline{\textup{Div}}_{\mathbb{T}}^{\textup{int}}(\mathcal{U}_{S}/\mathcal{O}_{K} )^{d+1} \longrightarrow \mathbb{R}.
		\end{displaymath}
		By Definition~\ref{global-intersection-numbers}, this pairing coincides with the arithmetic intersection pairing in Proposition~4.1.1 of Yuan and Zhang's~\cite{Y-Z} (see Theorem~\ref{global-intersection}). Observe that the $v$-adic roof functions $\vartheta_{\overline{\mathcal{D}}_v}$ of a nef toric adelic divisor $\overline{\mathcal{D}}$ are bounded by some constant $\beta$ independent of the place $v \in \mathfrak{M}(K)$. On the other hand, Theorem~\ref{global-toric-height-OKS} is valid for semipositive toric adelic divisors with unbounded $v$-adic roof functions. Therefore, by restricting to the toric case, we obtained more general results than Yuan and Zhang.
	\end{rem}
	\begin{rem}[Comparison with Burgos-Kramer's pairing]\label{finale}
		In Definition~4.6~of~\cite{BK24}, Burgos and Kramer extended the arithmetic intersection pairing of Yuan and Zhang using the notion of mixed relative energy, introduced by Darvas, Di Nezza, and Lu (see~4B~of~\cite{Darvas}). We sketch this construction:
		\begin{enumerate}[label=(\arabic*)]
			\item Start with a regular projective arithmetic variety $\mathcal{X}/\mathbb{Z}$ of relative dimension $d$ and fix an open subscheme $\mathcal{U}$ of $\mathcal{X}$. Then, let $\overline{\mathcal{B}} \in \overline{\textup{Div}}(\mathcal{X})_{\mathbb{Q}}$ be a nef boundary divisor (which, shrinking $\mathcal{U}$ if necessary, always exist).
			\item For each $0 \leq i \leq d$, let $\overline{\mathcal{D}}_{i} = (\mathcal{D}_i, g_i)$ and $\overline{\mathcal{D}}_{i}^{\prime}=  (\mathcal{D}_i, g_{i}^{\prime})$	be semipositive adelic divisors of $\mathcal{U}/\mathbb{Z}$, with $\overline{\mathcal{D}}_{i}^{\prime}$ nef. Observe that the divisors $ \overline{\mathcal{D}}_{i}$ and $\overline{\mathcal{D}}_{i}^{\prime}$ have the same divisorial part, and the Green's functions $g_{i}$ and $g_{i}^{\prime}$ are of psh-type and of relative full mass.
			\item Choose model divisors $\overline{\mathcal{E}}_i = (\mathcal{E}_i , g_{E_i}) \in \overline{\textup{Div}}(\mathcal{X})_{\mathbb{Q}}$ such that $\overline{\mathcal{E}}_i \geq \overline{\mathcal{B}}$ and $\overline{\mathcal{E}}_i  \geq \overline{\mathcal{D}}_i + 2 \cdot \overline{\mathcal{B}}$. This forces the functions $\varphi_i \coloneqq g_i - g_{E_i}$ and $\varphi^{\prime} \coloneqq  g_{i}^{\prime} - g_{E_i}$ to be $\omega_i$-plurisubharmonic, where $\omega_i \coloneqq \omega_{E_i}(g_i)$. Then, denote $\varphi \coloneqq (\varphi_0, \ldots , \varphi_d)$ and $ \varphi^{\prime} \coloneqq (\varphi_{0}^{\prime}, \ldots , \varphi_{d}^{\prime})$.
			\item The \textit{mixed relative energy} of $\varphi$ with respect to $\varphi^{\prime}$ is given by
			\begin{align*}
				I_{\varphi^{\prime}}(\varphi) = \sum_{i=0}^{d} \int_{\mathcal{X}(\mathbb{C})} (\varphi_i - \varphi_{i}^{\prime}) \langle & (\textup{dd}^{\textup{c}}\varphi_{0} + \omega_{0}) \wedge \ldots \wedge (\textup{dd}^{\textup{c}}\varphi_{i-1} + \omega_{i-1}) \wedge \\
				& \wedge (\textup{dd}^{\textup{c}}\varphi_{i+1}^{\prime} + \omega_{i+1}) \wedge \ldots \wedge (\textup{dd}^{\textup{c}}\varphi_{d}^{\prime} + \omega_{d})\rangle \wedge \delta_{\mathcal{X}(\mathbb{C})},
			\end{align*}
			where $\langle - \wedge \ldots \wedge - \rangle$ denotes the non-pluripolar product of positive $(1,1)$-currents, introduced in~\cite{BEGZ10}. The mixed relative energy is defined~in~2.48, and the above integral formula is Theorem~2.52~of~\cite{BK24}.
			\item The \textit{generalized arithmetic intersection product} $\overline{\mathcal{D}}_0 \cdot \ldots \cdot \overline{\mathcal{D}}_d$ is defined as
			\begin{displaymath}
				\overline{\mathcal{D}}_0 \cdot \ldots \cdot \overline{\mathcal{D}}_d \coloneqq \overline{\mathcal{D}}_{0}^{\prime} \cdot \ldots \cdot \overline{\mathcal{D}}_{d}^{\prime} + I_{\varphi^{\prime}}(\varphi),
			\end{displaymath}
			where the intersection product $\overline{\mathcal{D}}_{0}^{\prime} \cdot \ldots \cdot \overline{\mathcal{D}}_{d}^{\prime}$ is the one defined by Yuan and Zhang.
		\end{enumerate}
		Now, we show the relation with Theorem~\ref{global-toric-height-OKS}. Suppose that $\mathcal{X}$ is a toric scheme over $\mathbb{Z}$. By Lemma~\ref{generic-canonical}, there is an open embedding of a torus $\mathcal{U}_{S} \rightarrow \mathcal{X}$ for some finite set $S$ of maximal ideals of $\mathbb{Z}$. Assume for simplicity that the semipositive divisors $\overline{\mathcal{D}}_0, \ldots, \overline{\mathcal{D}}_{d}$ are all toric and equal to $\overline{\mathcal{D}}$. Similarly, assume that $\overline{\mathcal{D}}_{0}^{\prime}, \ldots, \overline{\mathcal{D}}_{d}^{\prime}$ are all toric and equal to $\overline{\mathcal{D}}^{\prime}$. Then, we have the following observations:
		\begin{enumerate}[label=(\arabic*')]
			\item The divisor $\overline{\mathcal{D}}^{\prime}$ is nef and has the same divisorial part as $\overline{\mathcal{D}}$. Then, for each finite place $v$ of $\mathbb{Q}$, we get the following equality of $v$-adic roof functions $\vartheta_{\overline{\mathcal{D}}_v} = \vartheta_{\overline{\mathcal{D}}_{v}^{\prime}}$. By Lemma~\ref{global-roof-OKS}, the $v$-adic roof functions $\vartheta_{\overline{\mathcal{D}}_{v}^{\prime}}$ are bounded uniformly on $v$. Hence, the only $v$-adic roof function that can be unbounded is $\vartheta_{\overline{\mathcal{D}}_{\infty}}$.
			\item Applying the formula in Theorem~\ref{global-toric-height-OKS}, we get
			\begin{displaymath}
				\overline{\mathcal{D}}^{d+1} = (d+1)! \sum_{v \in \mathfrak{M}(\mathbb{Q})} \delta_v \cdot \int_{\Delta_D} \vartheta_{\overline{\mathcal{D}}_v} \, \textup{dVol}_M.
			\end{displaymath}
			Similarly for $\overline{\mathcal{D}}^{\prime}$. 
			\item Since $\vartheta_{\overline{\mathcal{D}}_v} = \vartheta_{\overline{\mathcal{D}}_{v}^{\prime}}$ for all finite places $v$, we get
			\begin{displaymath}
				\overline{\mathcal{D}}^{d+1} =  (\overline{\mathcal{D}}^{\prime})^{d+1} + \delta_{\infty} \, (d+1)! \int_{\Delta_D} (\vartheta_{\overline{\mathcal{D}}_v} - \vartheta_{\overline{\mathcal{D}}^{\prime}_{v}}) \, \textup{dVol}_M.
			\end{displaymath}
			The integral gives the difference between local toric heights at the infinite place.
		\end{enumerate}
		Finally, Lemma 4.1.13~of~\cite{Per26} establishes a link between the local toric height at the infinite place and the mixed relative energy. Indeed, compare~(4) with the identity below
		\begin{displaymath}
			\textup{h}^{\textup{tor}}(X_{\Sigma}; \overline{\mathcal{D}}_{\infty}) = \sum_{j=0}^{d} \int_{X_{\Sigma}^{\textup{an}}} (g_{D,\infty} - g_{D,\textup{can}}) \langle \omega_{D}(g_{D,\infty})^{\wedge j} \wedge \omega_{D}(g_{D,\textup{can}})^{\wedge (d-j)} \rangle \wedge \delta_{X_{\Sigma}},
		\end{displaymath}
		where $X_{\Sigma}/\mathbb{Q}$ is the generic fiber of $\mathcal{X}/\mathbb{Z}$. This shows that the formulas in (5) and (4') coincide by choosing $\overline{\mathcal{D}}^{\prime}_{\infty} = (D, g_{D,\textup{can}})$, that is, the toric divisor $D$ equipped with its canonical Green's function. This shows that the generalized arithmetic intersection numbers defined by Burgos and Kramer coincide with those defined in this article. Note that Theorem~\ref{global-toric-height-OKS} is not implied by their results: By restricting to the toric case, we only need to fix the generic fiber $D$ of the toric adelic divisor $\overline{\mathcal{D}}$. Therefore, we are allowed to modify every local Green's function $g_{D,v}$ at the same time, instead of only the Green's function $g_{D,\infty}$ over the archimedean place (here we used the notation of Proposition~\ref{global-analytification}).
	\end{rem}
	We now use the concrete convex-analytic descriptions developed in this article to construct the following key examples. These illustrate the differences between the aforementioned theories and our improvements upon them.
	\begin{ex}
		Our first example is a nef toric adelic divisor $\overline{\mathcal{D}}$ which is not a model divisor. Let $K = \mathbb{Q}$, $S = \emptyset$, $d=1$ and choose an enumeration $v \colon \mathbb{N} \rightarrow \mathfrak{M}(\mathbb{Q})$ such that $v(0)$ is the archimedean place $\infty$. For each $n \in \mathbb{N}$, consider the closed concave function $\vartheta_{v(n)} \colon [0,1] \rightarrow \mathbb{R}$ given by
		\begin{displaymath}
			\vartheta_{v(0)} (y) \coloneqq 1 \quad \textup{and} \quad \vartheta_{v(n)} (y) \coloneqq \begin{cases} y - 2^{-n}, & 0 \leq y \leq 2^{-n} \\ 0, & 2^{-n} \leq y \leq 1. \end{cases}
		\end{displaymath}
		We claim that the $v$-tuple of concave functions $(\vartheta_{v(n)})$ determines a semipositive toric adelic divisor $\overline{\mathcal{D}}$ on the torus $\mathbb{G}_m / \mathbb{Z}$. Choose an integer $k >0$. Then, for each $n\geq k$ and $y \in [0,1]$, a quick calculation shows that $\sup \vartheta_{v(n)} = 0$ and $(\vartheta_{v(n)} \boxplus \iota_{[-2^{-k},2^{-k}]})(y) =0$, where $\iota_{\Delta}$ is the indicator function of the set $\Delta$. By Corollary~\ref{semipositive-OKS-roofs}, the claim is true. The toric adelic divisor $\overline{\mathcal{D}}$ is not a model divisor; we have $\vartheta_{\overline{\mathcal{D}}_{v(n)}} \neq \iota_{[0,1]}$ for every $n \in \mathbb{N}$. Let us compute its roof function $\vartheta_{\overline{\mathcal{D}}} \colon [0,1] \rightarrow \mathbb{R}$. At $y = 0$, we have
		\begin{displaymath}
			\vartheta_{\overline{\mathcal{D}}}(0) = \sum_{n=0}^{\infty} \vartheta_{v(n)}(0) = 1 - \sum_{n=1}^{\infty} 2^{-n} = 0.
		\end{displaymath}
		For each $n>0$ and $2^{-(n+1)} \leq y \leq 2^{-n}$, we have
		\begin{displaymath}
			\vartheta_{\overline{\mathcal{D}}}(y) = \sum_{k=0}^{\infty} \vartheta_{v(k)}(y) = 1 + \sum_{k=1}^{n}  y - 2^{-k} = ny + 2^{-n}.
		\end{displaymath}
		We conclude that
		\begin{displaymath}
			\vartheta_{\overline{\mathcal{D}}}(y) = \sum_{k=0}^{\infty} \vartheta_{v(k)}(y) = \begin{cases} 0, & y = 0 \\ ny + 2^{-n},  & 2^{-(n+1)} \leq y \leq 2^{-n}, \, n \in \mathbb{N}. \end{cases}
		\end{displaymath}
		Note that the series defining the global roof function is a finite sum for each point $y \in (0,1)$, as shown in Lemma~\ref{finite-sum}. The difference with the model divisor case is that the number of non-zero terms $\vartheta_{v(m)} (y)$ tends to infinity as $y$ approaches $0$. Since $\vartheta_{\overline{\mathcal{D}}}(y) \geq 0$ for all $y \in [0,1]$, Theorem~\ref{nef-torus-OKS} shows that $\overline{\mathcal{D}}$ is nef. Finally, we use Theorem~\ref{global-toric-height-OKS} to compute the arithmetic self-intersection number $\overline{\mathcal{D}}^{2}$:
		\begin{align*}
			\overline{\mathcal{D}}^{2} & = 2! \sum_{n=0}^{\infty} \int_{0}^{1} \vartheta_{v(n)} (y) \, \textup{d}y = 2 + 2 \sum_{n=1}^{\infty} \int_{0}^{2^{-n}} (y - 2^{-n}) \, \textup{d}y \\
			& = 2 +  2 \sum_{n=1}^{\infty} - 2^{-2n-1} = 2 - \sum_{n=1}^{\infty} 2^{-2n} = 2 - 1/3 = 5/3.
		\end{align*}
		Alternatively, we can use Yuan and Zhang's Theorem~\ref{global-intersection} to compute the arithmetic self-intersection number $\overline{\mathcal{D}}^{2}$. For each $n \in \mathbb{N}$, the $v$-tuple of $v$-adic roof functions $(\vartheta_{\overline{\mathcal{D}}_{n,v}})$ given by
		\begin{displaymath}
			\vartheta_{\overline{\mathcal{D}}_{n,v(k)}} \coloneqq \begin{cases} \vartheta_{k}, & k \leq n \\ \iota_{[0,1]}, & k > n \end{cases}
		\end{displaymath}
		determines a nef toric model divisor $\overline{\mathcal{D}}_n$. The sequence $\lbrace \overline{\mathcal{D}}_n \rbrace_{n \in \mathbb{N}}$ is decreasing and converges in the boundary topology to $\overline{\mathcal{D}}$. By a similar calculation, for each $n \in \mathbb{N}$, the self-intersection number $\overline{\mathcal{D}}_{n}^{2}$ is given by
		\begin{displaymath}
			\overline{\mathcal{D}}_{n}^{2} =  2 \sum_{k=0}^{n} \int_{0}^{1} \vartheta_{v(k)} (y) \, \textup{d}y.
		\end{displaymath}
		The arithmetic self-intersection number $\overline{\mathcal{D}}^2$ is given by the limit of the sequence $\lbrace \overline{\mathcal{D}}^{2}_{n} \rbrace_{n \in \mathbb{N}}$. This is the same limit we computed earlier. 
	\end{ex}
	\begin{ex}
		The second example is a semipositive toric adelic divisor $\overline{\mathcal{D}}$ which is not nef nor integrable and has finite arithmetic self-intersection number $\overline{\mathcal{D}}^2$. As in the previous example, we let $K = \mathbb{Q}$, $S = \emptyset$ and $d=1$.  Consider a real number $\alpha < 0$ and a family $(\vartheta_{v})$ of concave functions $\vartheta_{v} \colon [0,1] \rightarrow \mathbb{R}_{-\infty}$ given by
		\begin{displaymath}
			\vartheta_{\infty} (y) \coloneqq  \begin{cases} -\infty, & y=0 \\ 1-y^{\alpha} , & 0 < y \leq 1 \end{cases} \quad \textup{and} \quad \vartheta_v = \iota_{[0,1]} \textup{ if } v \neq \infty.
		\end{displaymath}
		By Corollary~\ref{semipositive-OKS-roofs}, the family $(\vartheta_{v})$ of concave functions determines a semipositive toric adelic divisor $\overline{\mathcal{D}}$. Its global roof function $\vartheta_{\overline{\mathcal{D}}} \colon [0,1] \rightarrow \mathbb{R}_{-\infty}$ is given by $\vartheta_{\overline{\mathcal{D}}} = \vartheta_{\infty}$, which is negative and unbounded. Therefore, Proposition~\ref{global-roof-OKS} and Corollary~\ref{integrable-bounded} imply that the semipositive toric adelic divisor $\overline{\mathcal{D}}$ is not nef nor integrable. Now, we show that the number $\overline{\mathcal{D}}^2$ is finite for appropriate choices of $\alpha$. Indeed, Theorem~\ref{global-toric-height-OKS} gives
		\begin{displaymath}
			\overline{\mathcal{D}}^{2}  = 2! \sum_{v \in \mathfrak{M}(\mathbb{Q})} \int_{0}^{1}\vartheta_{v} (y) \, \textup{d}y = 2 \int_{0}^{1}\vartheta_{\infty} (y) \, \textup{d}y =  2 \int_{0}^{1} (1-y^{\alpha})  \, \textup{d}y.
		\end{displaymath}
		An elementary calculation shows that
		\begin{displaymath}
			\int_{0}^{1} (1-y^{\alpha})  \, \textup{d}y = \begin{cases} \frac{\alpha}{\alpha +1}, &  -1 < \alpha < 0 \\ - \infty,  & \alpha \leq -1. \end{cases}
		\end{displaymath}
		Therefore, we conclude that
		\begin{displaymath}
			\overline{\mathcal{D}}^{2} =  \begin{cases} \frac{2 \alpha}{\alpha +1}, &  -1 < \alpha < 0 \\ - \infty,  & \alpha \leq -1. \end{cases}
		\end{displaymath}
		By Remark~\ref{finale}, the arithmetic self-intersection number $\overline{\mathcal{D}}^2$ can be computed using the Burgos-Kramer method. Indeed, the divisor $\overline{\mathcal{D}}$ can be obtained by modifying the archimedean Green's function of the nef toric model arithmetic divisor $\overline{\mathcal{D}}^{\prime}$ whose $v$-adic roof functions are given by the indicator function $\iota_{[0,1]}$ at all places. This means that $\overline{\mathcal{D}}^{\prime}$ is endowed with the canonical Green's function at all places.
	\end{ex}
	\begin{ex}\label{adelic-ex-3}
		The third example is a semipositive toric adelic divisor $\overline{\mathcal{D}}$ whose $v$-adic roof functions $\vartheta_{\overline{\mathcal{D}}_v}$ are all unbounded, and the arithmetic self-intersection number $\overline{\mathcal{D}}^{2}$ is finite. By Proposition~\ref{global-roof-OKS} and Corollary~\ref{integrable-bounded}, the semipositive toric adelic divisor $\overline{\mathcal{D}}$ cannot be nef nor integrable. Moreover, the fact that every $v$-adic roof function $\vartheta_{\overline{\mathcal{D}}_v}$ is unbounded implies that the number $\overline{\mathcal{D}}^{2}$ cannot be computed using the Burgos-Kramer theory (see~(1')~in~Remark~\ref{finale}). We combine the ideas in the first two examples: Let $K = \mathbb{Q}$, $S = \emptyset$, $d=1$, $\alpha \in (-1,0)$ and choose an enumeration $v \colon \mathbb{N} \rightarrow \mathfrak{M}(\mathbb{Q})$ such that $v(0)$ is the archimedean place $\infty$. Then, for each $n \in \mathbb{N}$, consider the closed concave function $\vartheta_{v(n)} \colon [0,1] \rightarrow \mathbb{R}_{-\infty}$ given by
		\begin{displaymath}
			\vartheta_{v(n)} (y) \coloneqq  \begin{cases} -\infty, & y=0 \\ 1-2^{n \alpha} y^{\alpha}, & 0 < y \leq 2^{-n}  \\ 0 & 2^{-n} \leq y \leq 1. \end{cases}
		\end{displaymath}
		We claim that the $v$-tuple of concave functions $(\vartheta_{v(n)})$ determines a semipositive toric adelic divisor $\overline{\mathcal{D}}$ on the torus $\mathbb{G}_m / \mathbb{Z}$. Choose an integer $k >0$. Then, for each $n\geq k$ and $y \in [0,1]$, a quick calculation shows that $\sup \vartheta_{v(n)} = 0$ and $(\vartheta_{v(n)} \boxplus \iota_{[-2^{-k},2^{-k}]})(y) =0$, where $\iota_{\Delta}$ is the indicator function of the set $\Delta$. By Corollary~\ref{semipositive-OKS-roofs}, the claim is true. Finally, we use Theorem~\ref{global-toric-height-OKS} to compute its arithmetic self-intersection number $\overline{\mathcal{D}}^2$:
		\begin{displaymath}
			\overline{\mathcal{D}}^2 = 2! \sum_{n=0}^{\infty} \int_{0}^{2^{-n}} (1-2^{n \alpha} y^{\alpha} ) \, \textup{d}y.
		\end{displaymath}
		By a similar calculation to the one in the previous example, we get the following identity
		\begin{displaymath}
			\int_{0}^{2^{-n}} (1-2^{n \alpha} y^{\alpha} ) \, \textup{d}y = \frac{\alpha}{2^n (\alpha+1)}.
		\end{displaymath}
		After substitution, we obtain
		\begin{displaymath}
			\overline{\mathcal{D}}^2 = 2 \sum_{n=0}^{\infty} \frac{\alpha}{2^n (\alpha+1)} =  \frac{4 \alpha }{\alpha +1}.
		\end{displaymath}
	\end{ex}
	\begin{rem}
		None of the toric adelic divisors in the previous three examples are model divisors. However, these divisors have the same associated toric compactified geometric divisor $D$, which is a model divisor. Indeed, it corresponds to the toric divisor $(\infty) \in \mathbb{P}^{1}_{\mathbb{Q}}$, whose associated rational convex polytope is $[0,1]$. The next step is to give an example of a semipositive toric adelic divisor $\overline{\mathcal{D}}$ on a split $d$-dimensional torus $\mathcal{U}/\mathbb{Z}$ whose associated toric compactified geometric divisor $D$ is not a model divisor, all of its $v$-adic roof functions $\vartheta_{\overline{\mathcal{D}}_v}$ are unbounded, and its arithmetic self-intersection number $\overline{\mathcal{D}}^{d+1}$ is finite. From these properties, we get the following observations:
		\begin{enumerate}[label=(\arabic*)]
			\item As in Example~\ref{adelic-ex-3}, the toric adelic divisor $\overline{\mathcal{D}}$ is not nef nor integrable. In particular, the number $\overline{\mathcal{D}}^{d+1}$ cannot be computed with the methods of Yuan-Zhang (see Theorem~\ref{global-intersection} or the article~\cite{Y-Z}).
			\item The toric compactified geometric divisor $D$ induced by $\overline{\mathcal{D}}$ is not a model divisor. Then, its associated convex set $\Delta_D$ is not a polytope. Therefore, the induced Hermitian singular metric $\| \cdot \|_{\infty}$ on the trivial line bundle $\mathcal{O}_{\mathcal{U}}$ does not satisfy Chern-Weil theory. In particular, the arithmetic intersection number $\overline{\mathcal{D}}^{d+1}$ cannot be computed using the Burgos-Kramer-Kühn theory (see~\cite{BKK5}~and~\cite{BKK7}).
			\item All of the $v$-adic roof functions $\vartheta_{\overline{\mathcal{D}}_v}$ are unbounded. Therefore, the number $\overline{\mathcal{D}}^{d+1}$ cannot be computed using the Burgos-Kramer theory (see~(1')~in~Remark~\ref{finale} or the article~\cite{BK24}).
		\end{enumerate}
		To construct such an example, we must consider a split torus $\mathcal{U}/\mathbb{Z}$ of relative dimension at least two. This is done below.
	\end{rem}
	\begin{ex}\label{worst-case}
		Our fourth example is a semipositive toric adelic divisor $\overline{\mathcal{D}}$ on the split torus $\mathbb{G}^{2}_{m} / \mathbb{Z}$ whose associated toric compactified geometric divisor $D$ is not a model divisor, all of its $v$-adic roof functions $\vartheta_{\overline{\mathcal{D}}_v}$ are unbounded, and the arithmetic self-intersection number $\overline{\mathcal{D}}^{3}$ is finite. Again, we let $\alpha \in (0,1)$ and choose an enumeration $v \colon \mathbb{N} \rightarrow \mathfrak{M}(\mathbb{Q})$ such that $v(0)$ is the archimedean place $\infty$. The idea is to modify the $v$-adic roof functions in Example~\ref{adelic-ex-3} to produce closed concave functions on the closed ball $\overline{\textup{B}}(0,1) \subset \mathbb{R}^2$ which, in polar coordinates, depend only on the radius. Thus, for each $n \in \mathbb{N}$, we define the closed concave function $\vartheta_{v(n)} \colon \overline{\textup{B}}(0,1) \rightarrow \mathbb{R}_{-\infty}$ given by
		\begin{displaymath}
			\vartheta_{v(n)}(r , \varphi) \coloneqq \begin{cases} 0, &  0 \leq r \leq 1 - 2^{-n} \\ 1-2^{n\alpha}(1 - r)^{\alpha}, & 1 - 2^{-n} \leq r < 1 \\
				- \infty, & r =1, \end{cases} \quad \textup{ where} \quad  \varphi \in [0,2\pi].
		\end{displaymath}
		We claim that the $v$-tuple of concave functions $(\vartheta_{v(n)})$ determines a semipositive toric adelic divisor $\overline{\mathcal{D}}$ on the torus $\mathbb{G}^{2}_m / \mathbb{Z}$. Choose an integer $k >0$. Then, for each $n\geq k$ and $y \in \overline{\textup{B}}(0,1)$, a quick calculation shows that
		\begin{displaymath}
			\sup \vartheta_{v(n)} = 0 \quad \textup{and} \quad (\vartheta_{v(n)} \boxplus \iota_{\overline{\textup{B}}(0,2^{-k})})(y)  =0,
		\end{displaymath}
		where $\iota_{\Delta}$ is the indicator function of the set $\Delta$.  Therefore, Corollary~\ref{semipositive-OKS-roofs} shows the claim. Since the compact convex set $\overline{\textup{B}}(0,1)$ is not a polytope, the associated toric compactified geometric divisor $D$ is a model divisor. Finally, we use Theorem~\ref{global-toric-height-OKS} to compute its arithmetic self-intersection number $\overline{\mathcal{D}}^3$:
		\begin{displaymath}
			\overline{\mathcal{D}}^3 = 3! \sum_{n=0}^{\infty} \int_{\overline{\textup{B}}(0,1)} \vartheta_{v(n)}(y) \, \textup{d}y = 6 \sum_{n=0}^{\infty} \int_{1-2^{-n}}^{1} \int_{0}^{2\pi}   (1-2^{n\alpha}(1 - r)^{\alpha})  \, \textup{d}\varphi \, \textup{d}r.
		\end{displaymath}
		By a similar calculation to the one in Example~\ref{adelic-ex-3}, we get the following identity
		\begin{displaymath}
			\int_{1-2^{-n}}^{1} \int_{0}^{2\pi}   (1-2^{n\alpha}(1 - r)^{\alpha})  \, \textup{d}\varphi \, \textup{d}r = 2\pi \int_{1-2^{-n}}^{1}  (1-2^{n\alpha}(1 - r)^{\alpha}) \,  \textup{d}r  =  \frac{2 \pi \alpha}{2^n (\alpha+1)}.
		\end{displaymath}
		After substitution, we obtain
		\begin{displaymath}
			\overline{\mathcal{D}}^3 = 6 \sum_{n=0}^{\infty} \frac{2 \pi \alpha}{2^n (\alpha+1)} = \frac{24 \pi \alpha }{\alpha +1}.
		\end{displaymath}
	\end{ex}
	\begin{ex}
		Our final example is a global version of Example~4.3.13~in~\cite{Per26}. That is, we construct semipositive toric adelic divisors $\overline{\mathcal{D}}_1$ and $\overline{\mathcal{D}}_2$ on the split torus $\mathbb{G}^{2}_{m} / \mathbb{Z}$ such that the arithmetic self-intersection numbers $\overline{\mathcal{D}}_{1}^{3}$ and $\overline{\mathcal{D}}_{2}^{3}$ are finite, but $(\overline{\mathcal{D}}_1 + \overline{\mathcal{D}}_{2})^{3}$ is not. This shows that the height inequality in Corollary~\ref{global-mixed} is not always an equality. Consider the compact convex sets
		\begin{displaymath}
			\Delta_1 \coloneqq \lbrace (x,y) \, \vert \, x \geq 0, \, y \geq 0 \textup{ and } 1 \geq x+y \rbrace \quad \textup{and} \quad \Delta_2 = [0,1] \times \lbrace 0 \rbrace,
		\end{displaymath}
		the standard $2$-dimensional simplex and an embedding of the unit interval, respectively. Then, consider the closed concave functions $\vartheta_{1,\infty} \colon \Delta_1 \rightarrow \mathbb{R}_{-\infty}$ and $\vartheta_{2,\infty} \colon \Delta_2 \rightarrow \mathbb{R}_{-\infty}$, given by
		\begin{displaymath}
			\vartheta_{1,\infty} (x,y) \coloneqq -(1-y)^{-1} \quad \textup{ and } \quad \vartheta_{2,\infty} (x,y) \coloneqq 0.
		\end{displaymath}
		Then, declare $\vartheta_{\overline{\mathcal{D}}_{i,v}} = \iota_{\Delta_{i}}$ for all places $v \neq \infty$. By Corollary~\ref{semipositive-OKS-roofs}, the family of closed concave functions $(\vartheta_{\overline{\mathcal{D}}_{i,v}})$ determines a semipositive toric adelic divisor $\overline{\mathcal{D}}_i$ on the split torus $\mathbb{G}^{2}_{m} / \mathbb{Z}$. Then, Theorem~\ref{global-toric-height-OKS} gives
		\begin{displaymath}
			\overline{\mathcal{D}}_{i}^{3} = 3! \sum_{v \in \mathfrak{M}(\mathbb{Q})}\int_{\Delta_i} \vartheta_{\overline{\mathcal{D}}_{i,v}} \, \textup{dVol} = 6 \int_{\Delta_i} \vartheta_{\overline{\mathcal{D}}_{i, \infty}} \, \textup{dVol}.
		\end{displaymath}
		Following the calculation in Example~4.3.13~of~\cite{Per26}, we obtain
		\begin{displaymath}
			\overline{\mathcal{D}}_{i}^{3} = \begin{cases} -6, & i=1 \\ 0, & i=2. \end{cases}
		\end{displaymath}
		On the other hand, the semipositive toric adelic divisor $\overline{\mathcal{D}}_3 = \overline{\mathcal{D}}_1 + \overline{\mathcal{D}}_2$ has corresponding compact convex set
		\begin{displaymath}
			\Delta_3 = \Delta_1 + \Delta_2 = \lbrace (x,y) \, \vert \, x \geq 0, \, 1 \geq y \geq 0, \textup{ and } 2 \geq x+y \rbrace.
		\end{displaymath}
		By Proposition A.1.9~of~\cite{Per26}, the $v$-adic roof function $\vartheta_{\overline{\mathcal{D}}_{3,v}} \colon \Delta_3 \rightarrow \mathbb{R}_{-\infty}$ is given by the sup-convolution $\vartheta_{\overline{\mathcal{D}}_{3,v}} = \vartheta_{\overline{\mathcal{D}}_{1,v}} \boxplus \vartheta_{\overline{\mathcal{D}}_{2,v}}$. A quick calculation shows that
		\begin{displaymath}
			\vartheta_{\overline{\mathcal{D}}_{3, \infty}}(x,y) =  -(1-y)^{-1} \quad \textup{and} \quad \vartheta_{\overline{\mathcal{D}}_{3,v}} = \iota_{\Delta_3} \textup{ if } v \neq \infty.
		\end{displaymath}
		By Theorem~\ref{global-toric-height-OKS}, we know that
		\begin{displaymath}
			\overline{\mathcal{D}}_{3}^{3} = 3! \sum_{v \in \mathfrak{M}(\mathbb{Q})}\int_{\Delta_i} \vartheta_{\overline{\mathcal{D}}_{3,v}} \, \textup{dVol} = 6 \int_{\Delta_3} \vartheta_{\overline{\mathcal{D}}_{3, \infty}} \, \textup{dVol}.
		\end{displaymath}
		Since the function $\vartheta_{\overline{\mathcal{D}}_{3, \infty}}$ is negative and $[0,1]^2$ is contained in $\Delta_3$, we get the bound
		\begin{displaymath}
			\int_{\Delta_3} \vartheta_{\overline{\mathcal{D}}_{3, \infty}} \, \textup{dVol} \leq  \int_{[0,1]^2} \vartheta_{\overline{\mathcal{D}}_{3, \infty}} \, \textup{dVol} = -\int_{0}^{1} \int_{0}^{1} (1-y)^{-1} \, \textup{d}x \, \textup{d}y = - \int_{0}^{1} (1-y)^{-1} \, \textup{d}y = -\infty.
		\end{displaymath}
		Therefore, we get $\overline{\mathcal{D}}_{3}^{3}  = (\overline{\mathcal{D}}_1 + \overline{\mathcal{D}}_{2})^{3} = -\infty$.
	\end{ex}


\begin{thebibliography}{X}
		\bibitem[Ar74]{Ar} \textsc{S. J. Arakelov}. \textit{An intersection theory for divisors on an arithmetic surface}. Math. USSR Izv. 8 (1974), 1167-1180.
		\bibitem[Beer74]{Beer} \textsc{G. A. Beer}. \textit{The Hausdorff metric and convergence in measure}. Michigan Math. J. 21(1): 63-64, 1974.
		\bibitem[Ber90]{Ber} \textsc{V. G. Berkovich}. \textit{Spectral theory and analytic geometry over non-Archimedean fields}. Math. Surveys Monogr., vol.33, Amer. Math. Soc., 1990.
		\bibitem[BLR90]{BLR90} \textsc{S. Bosch, W. Lütkebohmert, M. Raynaud}. \textit{Néron models}. Ergebnisse der Mathematik und ihrer Grenzgebiete. 3. Folge. Springer-Verlag Berlin Heidelberg, 1990.
		\bibitem[Bo98]{B98} \textsc{J.-B. Bost}. \textit{Intersection theory on arithmetic surfaces and} $L_{1}^{2}$ \textit{metrics}. Letter dated March 6th, 1998.
		\bibitem[BGS94]{BGS}\textsc{J.-B. Bost, H. Gillet, C. Soul\'{e}}. \textit{Heights of projective varieties and positive Green forms}.  J. Amer. Math. Soc. 7 (1994), 903–1027.
		\bibitem[Bot19]{Bot19}\textsc{A. M. Botero}. \textit{Intersection theory of toric b-divisors in toric varieties}.  J. Algebr. Geom. 28, No. 2, 291-388 (2019).
		\bibitem[BBHdJ21]{BBHdJ21} \textsc{A. M. Botero, J. I. Burgos Gil, D. Holmes, R. de Jong}. \textit{Chern-Weil and Hilbert-Samuel formulae for singular hermitian line bundles}. Documenta Mathematica 27 (2022) 2563–2623.
		\bibitem[BBK07]{BBK07} \textsc{J. Bruinier, J. I. Burgos Gil, U. Kühn}. \textit{Borcherds products and arithmetic intersection theory on Hilbert modular surfaces}. Duke Math. J., 139(1):1–88, 2007.
		\bibitem[BEGZ10]{BEGZ10} \textsc{S. Boucksom, P. Eyssidieux, V. Guedj, A. Zeriahi}. \textit{Monge-Amp\`{e}re equations in big cohomology classes}. Acta Math., 205(2): 199–262, 2010.
		\bibitem[BFJ16]{BFJ16} \textsc{S . Boucksom, C. Favre, M. Jonsson}. \textit{Singular semipositive metrics in non- Archimedean geometry}. J. Algebraic Geom., 25(1):77–139, 2016.
		\bibitem[BGJK20]{BGJK20} \textsc{J.I. Burgos Gil, W. Gubler, P. Jell, K. Künnemann}. \textit{A comparison of positivity in complex and tropical toric geometry.} Mathematische Zeitschrift, 299 (3), (2021), 1199-1255..
		\bibitem[BGJK21]{BGJK21} \textsc{J.I. Burgos Gil, W. Gubler, P. Jell, K. Künnemann}. \textit{Pluripotential theory for tropical toric varieties and non-Archimedean Monge-Amp\`{e}re equations.} Kyoto J. Math. 65 (1), (2025), 55-152.
		\bibitem[BK24]{BK24}  \textsc{J. I. Burgos Gil, J. Kramer}. \textit{On the height of the universal abelian variety.} arXiv:2403.11745, 2024.
		\bibitem [BKK05]{BKK5} \textsc{J. I. Burgos Gil, J. Kramer, U. Kühn}. \textit{Arithmetic characteristic classes of automorphic vector bundles}, Documenta Math. 10 (2005), 619–716.
		\bibitem[BKK07]{BKK7} \textsc{J. I. Burgos Gil, J. Kramer, U. Kühn}. \textit{Cohomological arithmetic Chow rings}. J. of the Inst. of Math. Jussieu (2007) 6(1), 1–172.
		\bibitem[BKK16]{BKK16} \textsc{J. I. Burgos Gil, J. Kramer, U. Kühn}. \textit{The singularities of the invariant metric on the Jacobi line bundle}. “Recent Advances In Hodge Theory: Period Domains, Algebraic Cycles, and Arithmetic”. LMS Lecture Notes Series, vol. 427, 45-77. Cambridge University Press, 2016.
		\bibitem[BMPS16]{BMPS}\textsc{ J. I. Burgos Gil, A. Moriwaki, P. Philippon, M. Sombra}. \textit{Arithmetic positivity on toric varieties}. J. Algebr. Geom. 25, 201-272 (2016).
		\bibitem[BPS11]{BPS} \textsc{ J. I. Burgos Gil, P. Philippon, M. Sombra}. \textit{Arithmetic geometry of toric varieties. Metrics, measures and heights}. Ast\'{e}risque, 360:vi+222, 2014.
		\bibitem[CG24]{CGub24} \textsc{Y. Cai, W. Gubler}. \textit{Abstract divisorial spaces and arithmetic intersection numbers}. arXiv:2409.00611, 2024.
		\bibitem[CD12]{CD12} \textsc{A. Chambert-Loir, A. Ducros}. \textit{Formes diff\'{e}rentielles r\'{e}elles et courants sur les espaces de Berkovich}. arXiv:1204.6277, 2012.
		\bibitem[CGSZ18]{CGSZ} \textsc{D. Coman, V. Guedj, S. Sahin, A. Zeriahi}. \textit{Toric pluripotential theory}. arXiv:1804.03387, 2018.
		\bibitem[CLS11]{CLS}\textsc{D. A. Cox, J. D. Little, H. K. Schenck}. \textit{Toric varieties}. Grad. Stud. Math., vol. 124, Amer. Math. Soc., 2011.
		\bibitem[DDNL18]{Darvas}\textsc{T. Darvas, E. Di Nezza, C. H. Lu}. \textit{Monotonicity of nonpluripolar products and complex Monge–Amp\`{e}re equations with prescribed singularity}. Analysis \& PDE Vol 11, No. 8, 2018.
		\bibitem[Dem12]{Dem12} \textsc{J.-P. Demailly}. \textit{Complex analytic and differential geometry}. eBook, Version June 21st, 2012.
		\bibitem[Dmz70]{Dmz70} \textsc{M. Demazure}. \textit{Sous-groupes alge\'{e}briques de rang maximum du groupe de Cremona}. Ann. Sci. \'{E}cole Norm. Sup. (4) 3 (1970), 507–588.
		\bibitem[DN15]{DN15} \textsc{E.  Di Nezza}. \textit{Stability of Monge-Amp\`{e}re energy classes}. J. Geom. Anal., Vol. 25 (2015), no. 4, 2565-2589.
		\bibitem[DGH21]{DGH21} \textsc{V. Dimitrov, Z. Gao, P. Habegger}. \textit{Uniformity in Mordell-Lang for curves}. Ann. of Math. (2) 194 (2021), no. 1, 237–298.
		\bibitem[Fal83]{Fal83} \textsc{G. Faltings}. \textit{Endlichkeitssätze für abelsche Varietäten über Zahlkörpern}. Invent. Math. 73 (1983), no. 3, 349–366.
		\bibitem[FalCh90]{FC90} \textsc{G. Faltings, C. L. Chai}. \textit{Degeneration of abelian varieties}. Vol. 22 of Ergebnisse der Mathematik und ihrer Grenzgebiete (3) [Results in Mathematics and Related Areas (3)]. Springer-Verlag, Berlin, 1990. With an appendix by D. Mumford.
		\bibitem[Ful84]{Ful} \textsc{W. Fulton}. \textit{Intersection theory}. Ergebnisse der Mathematik und ihrer Grenzgebiete, Volume 3 (Springer, 1984).
		\bibitem[GS90]{GS} \textsc{H. Gillet, C. Soul\'{e}}. \textit{Arithmetic intersection theory}. Inst. Hautes \'{E}tudes Sci. Publ. Math. 72 (1990), 94-174.
		\bibitem[GW20]{GW20} \textsc{U. Görtz, T. Wedhorn}. \textit{Algebraic Geometry I: Schemes}. Springer Studium Mathematik - Master. Springer Spektrum Wiesbaden, Second edition, 2020.
		\bibitem[Gua14]{Gualdi} \textsc{R. Gualdi}. \href{https://roberto-gualdi.staff.upc.edu/documents/Cox_Rings_for_a_particular_class_of_Toric_Schemes.pdf}{\textit{Cox Rings for a particular class of toric schemes}}. Master thesis, Universit\'{e} de Bordeaux  (2014).
		\bibitem[Gub98]{Gub98} \textsc{W. Gubler}. \textit{Local heights of subvarieties over non-archimedean fields}. J. reine angew. Math. 498 (1998), 61-113.
		\bibitem[Gub02]{Gub02} \textsc{W. Gubler}. \textit{Basic properties of heights of subvarieties}. Habilitation thesis, ETH Zürich, 2002. 
		\bibitem[Gub03]{Gub03} \textsc{W. Gubler}. \textit{Local and canonical heights of subvarieties}. Ann. Sc. Norm. Super. Pisa Cl. Sci. (5) 2 (2003), 711-760.
		\bibitem[GZ07]{GZ07} \textsc{V. Guedj, A. Zeriahi}. \textit{The weighted Monge-Amp\`{e}re energy of quasiplurisubharmonic functions}. J. Funct. Anal., 250(2):442–482, 2007.
		\bibitem[Har77]{Har77}\textsc{R. Hartshorne}. \textit{Algebraic Geometry}. Graduate Texts in Mathematics, vol. 53, Springer New York, 1977.
		\bibitem[JvP22]{JvP22}\textsc{B. Jung and A.-M. von Pippich}. \textit{The arithmetic volume of the moduli space of abelian surfaces}. arXiv:2205.11864 [math.AG], 2022.
		\bibitem[KKMS73]{KKMS} \textsc{G. Kempf, F. Knudsen, D. Mumford, B. Saint-Donat}. \textit{Toroidal Embeddings 1}. (LNM, volume 339), Springer-Verlag Berlin Heidelberg 1973.
		\bibitem[KRY06]{KRY} \textsc{S. Kudla, S. M. Rapoport, T. Yang}. \textit{Modular forms and special cycles on Shimura curves}. volume 161 of Annals of Mathematics Studies. Princeton University Press, Princeton, NJ, 2006.
		\bibitem[Ku01]{K01} \textsc{U. Kühn}. \textit{Generalized arithmetic intersection numbers}. J. Reine Angew. Math. 534 (2001), 209–236.
		\bibitem[Kuh21]{Kuh21} \textsc{L. Kühne}. \textit{Equidistribution in Families of Abelian Varieties and Uniformity}. arXiv:2101.10272v4, 2021.
		\bibitem[Kue98]{Kue98} \textsc{K. Künnemann}. \textit{Projective regular models for abelian varieties, semistable reduction, and the height pairing}. Duke Mathematical Journal, Vol. 95, No. 1. 1998.
		
		\bibitem[LP24]{LP24} \textsc{T. Lemanissier, J. Poineau}. \textit{Espaces de Berkovich globaux : cat\'{e}gorie, topologie, cohomologie}.  arXiv:2010.08858v2, 2024.
		\bibitem[Liu02]{Liu} \textsc{Q. Liu}. \textit{Algebraic Geometry and Arithmetic Curves}. Oxford Graduate Texts in Mathematics 6, Oxford University Press, 2002.
		\bibitem[Mun00]{Mun00} \textsc{J. Munkres}. \textit{Topology}. Pearson College Div, 2nd edition, 2000.
		\bibitem[Neu99]{Neu} \textsc{J. Neukirch}. \textit{Algebraic number theory}. Grundlehren der mathematischen Wissenschaften, Springer Berlin, Heidelberg, 1999.
		\bibitem[Pay08]{Payne} \textsc{S. Payne}. \textit{Analytification is the limit of all tropicalizations}. Mathematical Research Letters 16 (2009), 543–556.
		\bibitem[PA26]{Per26} \textsc{G. Y. Peralta Alvarez}. \textit{Heights on toric varieties for singular metrics: Local theory}. arXiv:2601.14167, 2026.
		\bibitem[Roc70]{Roc} \textsc{R. T. Rockafellar}. \textit{Convex analysis}. Princeton Math. Series, vol. 28, Princeton Univ. Press, 1970.
		\bibitem[Roh10]{Roh10} \textsc{F. Rohrer}. \textit{Toric schemes}. Dissertation, Universität Zürich, 2010.
		\bibitem[Sie36]{Sie36} \textsc{C. L. Siegel}. \textit{The volume of the fundamental domain for some infinite groups}. Trans. Amer. Math. Soc., 39(2):209–218, 1936.
		\bibitem[Son24]{Son24} \textsc{Y. Song}. \textit{Norm-equivariant metrized divisors on Berkovich spaces}. arXiv:2406.19912, 2025.
		\bibitem[V91]{Vojta} \textsc{P. Vojta}. \textit{Siegel’s theorem in the compact case}. Ann. Math. \textbf{133} (1991), 509–548.
		\bibitem[Stacks]{Stacks} \textsc{The Stacks Project Authors}.  \href{https://stacks.math.columbia.edu}{\textit{The Stacks Project}}. Online open source textbook, 2024.
		\bibitem[Yua26]{Yua26} \textsc{X. Yuan}. \textit{Arithmetic bigness and a uniform Bogomolov-type result}. Ann. of Math. (2) 203(1): 15-119 (2026). 
		\bibitem[YZ26]{Y-Z} \textsc{X. Yuan, S.-W. Zhang}. \textit{Adelic line bundles over quasi-projective varieties}. Annals of Mathematics Studies, Princeton University Press, 2026.
		\bibitem[Zha95]{Z95} \textsc{S.-W. Zhang}. \textit{Small points and adelic metrics}. J. Algebraic Geom. 4 (1995), 281-300.
	\end{thebibliography}
\end{document}